\documentclass[10pt]{article}
\usepackage{amssymb,amsmath,amsfonts}
\usepackage{dsfont}
\usepackage{amsthm}
\usepackage{graphicx,color,subfig}
\usepackage{xspace}
\usepackage{mathrsfs}
\usepackage{amscd}

\newtheorem{lemma}{Lemma}[section]
\newtheorem{theorem}[lemma]{Theorem}
\newtheorem{proposition}[lemma]{Proposition}
\newtheorem{corollary}[lemma]{Corollary}

\theoremstyle{definition}

\newtheorem{definition}[lemma]{Definition}
\newtheorem{remark}[lemma]{Remark}

\makeatletter
\def\keywords{
    \vspace{1ex}
    \noindent
    \if@twocolumn
      \small{\bf  Keywords}\/---$\!$    \else
      \begin{center}\small\ {\bf Keywords}\end{center}\quotation\small
    \fi}
\def\endkeywords{\vspace{0.6em}\par\if@twocolumn\else\endquotation\fi
    \normalsize\rm}
\makeatother

\renewcommand{\O}{\ensuremath{\mathcal O}}
\renewcommand{\P}{\ensuremath{\mathcal P}}

\newcommand{\calL}{\ensuremath{\mathcal L}}

\newcommand{\Z}{\ensuremath{\mathbb Z}}
\newcommand{\T}{\ensuremath{\mathbb T}}
\newcommand{\bS}{\ensuremath{\mathbb S}}

\DeclareMathOperator{\tr}{tr}

\DeclareMathOperator{\Sp}{Sp}
\DeclareMathOperator{\loc}{loc}

\newcommand{\mb}[1]{\ensuremath{\mathbb{#1}}}
\newcommand{\N}{{\mb{N}}}

\newcommand{\R}{{\mb{R}}}
\newcommand{\C}{{\mb{C}}}

\newcommand{\de}{\delta}
\newcommand{\eps}{\varepsilon}

\newcommand{\M}{\ensuremath{\mathcal M}}

\newcommand{\A}{\ensuremath{\mathcal A}}

\let \Re \relax
\DeclareMathOperator{\Re}{Re}
\let \Im \relax
\DeclareMathOperator{\Im}{Im}

\newcommand{\ovl}[1]{\overline{#1}}

\newcommand{\Con}{\ensuremath{\mathscr C}}
\newcommand{\Dist}{\ensuremath{\mathscr D}}
\newcommand{\Cinf}{\ensuremath{\Con^\infty}}
\newcommand{\Cinfc}{\ensuremath{\Con^\infty}_{c}}
\newcommand{\Conc}{\ensuremath{\Con^0}_{c}}

\newcommand{\scrS}{\ensuremath{\mathscr S}}

\DeclareMathOperator{\supp}{supp}

\DeclareMathOperator{\Vol}{Vol}

\DeclareMathOperator{\rk}{rk}

\DeclareMathOperator{\Op}{Op}

\DeclareMathOperator{\id}{Id}

\newcommand{\transp}{\ensuremath{\phantom{}^{t}}}

\renewcommand{\d}{\ensuremath{\partial}}

\newcommand{\nhd}{neighbourhood\xspace}

\newcommand{\wrt}{w.r.t.\@\xspace}
\newcommand{\rhs}{r.h.s.\@\xspace}

\newcommand{\ie}{i.e.\@\xspace}
\newcommand{\eg}{e.g.\@\xspace}

\newcommand{\resp}{resp.\@\xspace}

\newcommand{\dsp}{\displaystyle}

\def\tB{\widetilde{B}}
\def\tzeta{\widetilde{\zeta}}

\newcommand{\E}{\mathscr E}
\newcommand{\F}{\mathscr F}
\newcommand{\G}{\mathscr G}

\renewcommand{\H}{\ensuremath{\mathcal H}}
\newcommand{\K}{\ensuremath{\mathcal K}}

\def\bbbone{{\mathchoice {1\mskip-4mu {\rm{l}}} {1\mskip-4mu {\rm{l}}}
{ 1\mskip-4.5mu {\rm{l}}} { 1\mskip-5mu {\rm{l}}}}}

\newcommand{\defi}{\stackrel{\rm{def}}{=}}

\numberwithin{equation}{section}
\begin{document}

\title{Decay rates for the damped wave equation on the torus \\
{\em With an appendix by St\'{e}phane Nonnenmacher}\footnote{\texttt{snonnenmacher@cea.fr}}}

\author{
Nalini Anantharaman\footnote{\texttt{Nalini.Anantharaman@math.u-psud.fr}}
and Matthieu L\'eautaud\footnote{\texttt{Matthieu.Leautaud@math.u-psud.fr}},
\\
\small{Universit\'e Paris-Sud 11, Math\'ematiques, B\^atiment 425, 91405 Orsay Cedex, France}
}
\maketitle

\begin{abstract}
We address the decay rates of the energy for the damped wave equation when the damping coefficient $b$ does not satisfy the Geometric Control Condition (GCC). First, we give a link with the controllability of the associated Schr\"odinger equation. We prove in an abstract setting that the observability of the Schr\"odinger group implies that the semigroup associated to the damped wave equation decays at rate $1/\sqrt{t}$ (which is a stronger rate than the general logarithmic one predicted by the Lebeau Theorem).

Second, we focus on the 2-dimensional torus. We prove that the best decay one can expect is $1/t$, as soon as the damping region does not satisfy GCC. Conversely, for smooth damping coefficients $b$, we show that  the semigroup decays at rate $1/t^{1-\eps}$, for all $\eps >0$. The proof relies on a second microlocalization around trapped directions, and resolvent estimates.

In the case where the damping coefficient is a characteristic function of a strip (hence discontinuous), 
St\'{e}phane Nonnenmacher computes in an appendix part of the spectrum of the associated damped wave operator, proving that the semigroup cannot decay faster than $1/t^{2/3}$.
In particular, our study shows that the decay rate highly depends on the way $b$ vanishes.

%

%

\end{abstract}


\setcounter{tocdepth}{2}
\tableofcontents

\part{The damped wave equation}

\section{Decay of energy: a survey of existing results}
Let $(M,g)$ be a smooth compact connected Riemannian $d$-dimensional manifold with or without boundary $\d M$. We denote by $\Delta$ the (non-positive) Laplace-Beltrami operator on $M$ for the metric $g$. Given a bounded nonnegative function, $b \in L^\infty(M)$, $b(x) \geq0$ on $M$, we want to understand the asymptotic behaviour as $t \to + \infty$ of the solution $u$ of the problem
\begin{equation}
\label{eq: stabilization}
\left\{
\begin{array}{ll}
\d_{t}^2 u - \Delta u + b(x) \d_{t} u = 0 & \text{in }\R^+ \times M , \\
u = 0                                           & \text{on }\R^+ \times \d M 
\ (\text{if} \ \d M \neq \emptyset ) , \\
(u, \d_t u)|_{t=0} = (u_0,u_1)                  & \text{in }M .\\
\end{array}
\right.
\end{equation}
The energy of a solution is defined by 
\begin{align}
\label{eq: definition energy}
E(u,t) = \frac12 (\|\nabla u(t)\|_{L^2(M)}^2 + \|\d_t u(t)\|_{L^2(M)}^2 ).
\end{align}
Multiplying~\eqref{eq: stabilization} by $\d_{t} u$ and integrating on $M$ yields the following dissipation identity 
$$
\frac{d}{dt}E(u,t) = - \int_{M} b |\d_{t} u|^2 dx ,
$$
which, as $b$ is nonnegative, implies a decay of the energy. As soon as $b \geq C >0$ on a nonempty open subset of $M$, the decay is strict 
and $E(u,t) \to 0$ as $t \to + \infty$. The question is then to know at which rate the energy goes to zero.

\bigskip 
The first interesting issue concerns uniform stabilization: under which condition does there exist a function $F(t)$, $F(t) \to 0$, such that 
\begin{equation}
\label{eq: decroissance uniforme energie}
E(u, t)\leq F(t) E(u, 0) \  ?
\end{equation}
The answer was given by Rauch and Taylor~\cite{RT:74} in the case $\d M = \emptyset$ and by Bardos, Lebeau and Rauch~\cite{BLR:92} in the general case (see also~\cite{BG:97} for the necessity of this condition): assuming that $b \in \Con^0(\ovl{M})$, uniform stabilisation occurs if and only if the set $\{b>0\}$ satisfies the Geometric Control Condition (GCC). Recall that a set $\omega \subset M$ is said to satisfy GCC if there exists $L_0>0$ such that every geodesic $\gamma$ (\resp generalised geodesic in the case $\d M \neq \emptyset$) of $M$ with length larger than $L_0$ satisfies $\gamma \cap \omega \neq \emptyset$. Under this condition, one can take $F(t) = C e^{-\kappa t}$ (for some constants $C, \kappa>0$) in \eqref{eq: decroissance uniforme energie}, and the energy decays exponentially. 
Finally, Lebeau gives in~\cite{Leb:96} the explicit (and optimal) value of the best decay rate $\kappa$ in terms of the spectral abscissa of the generator of the semigroup and the mean value of the function $b$ along the rays of geometrical optics. 

\bigskip 
In the case where $\{b>0\}$ does not satisfy GCC, \ie in the presence of ``trapped rays'' that do not meet $\{b>0\}$, what can be said about the decay rate of the energy? As soon as $b \geq C >0$ on a nonempty open subset of $M$, Lebeau shows in~\cite{Leb:96} that the energy (of smoother initial data) goes at least logarithmically to zero (see also~\cite{Burq:98}):
\begin{equation}
\label{eq: decroissance log}
E(u, t) \leq C \big( f(t) \big)^2 \left( \|u_0\|_{H^2(M) \cap H^1_0(M)}^2 
+ \|u_1\|_{H^1_0(M)}^2 \right) , \quad \text{ for all } t >0 ,
\end{equation}
with $f(t) = \frac{1}{\log(2+t)}$ (where $H^2(M) \cap H^1_0(M)$ and $H^1_0(M)$ have to be replaced by $H^2(M)$ and $H^1(M)$ respectively if $\d M = \emptyset$). Note that here, $\big( f(t) \big)^2$ characterizes the decay of the energy, whereas $f(t)$ is that of the associated semigroup.
Moreover, the author constructed a series of explicit examples of geometries for which this rate is optimal, including for instance the case where $M = \bS^2$ is the two-dimensional sphere and $\{b>0\} \cap N_\eps = \emptyset$, where $N_\eps$ is a \nhd of an equator of $\bS^2$. This result is generalised in~\cite{LR:97} for a wave equation damped on a (small) part of the boundary. 
In this paper, the authors also make the following comment about the result they obtain:

\bigskip 
``Notons toutefois qu'une \'etude plus approfondie de la localisation spectrale et des
taux de d\'ecroissance de l'\'energie pour des donn\'ees r\'eguli\`eres doit faire intervenir
la dynamique globale du flot g\'eod\'esique g\'en\'eralis\'e sur $M$. Les th\'eor\`emes \cite[Th\'eor\`eme 1]{LR:97} et \cite[Th\'eor\`eme 2]{LR:97} ne fournissent donc que les bornes {\it a priori} qu'on peut obtenir sans aucune hypoth\`ese sur la dynamique, en n'utilisant que les in\'egalit\'es de Carleman qui traduisent
``l'effet tunnel''.''

\bigskip
In all examples where the optimal decay rate is logarithmic, the trapped ray is a stable trajectory from the point of view of the dynamics of the geodesic flow. This means basically that an important amount of the energy can stay concentrated, for a long time, in a \nhd of the trapped ray, \ie away from the damping region.

If the trapped trajectories are less stable, or unstable, one can expect to obtain an intermediate decay rate, between exponential and logarithmic. We shall say that the energy decays at rate $f(t)$ if~\eqref{eq: decroissance log} is satisfied (more generally, see Definition~\ref{def: stable at rate} below in the abstract setting).
This problem has already been adressed and, in some particular geometries, several different behaviours have been exhibited. Two main directions have been investigated.

On the one hand, Liu and Rao considered in \cite{LR:05} the case where $M$ is a square and the set $\{b >0\}$ contains a vertical strip. In this situation, the trapped trajectories consist in a family of parallel vertical geodesics; these are unstable, in the sense that nearby geodesics diverge at a linear rate. They proved that the energy decays at rate $\Big(\frac{\log(t)}{t}\Big)^\frac12$ (\ie, that~\eqref{eq: decroissance log} is satisfied with $f(t) = \Big(\frac{\log(t)}{t}\Big)^\frac12$). This was extended by Burq and Hitrik~\cite{BH:07} (see also~\cite{Nish:09}) to the case of partially rectangular two-dimensional domains, if the set $\{b >0\}$ contains a \nhd of the non-rectangular part. In \cite{Phung:07b}, Phung proved a decay at rate $t^{-\de}$ for some (unprecised) $\de>0$ in a three-dimensional domain having two parallel faces. 
In all these situations, the only obstruction to GCC is due to a ``cylinder of periodic orbits''. The geometry is flat and the unstabilities of the geodesic flow around the trapped rays are relatively weak (geodesics diverge at a linear rate).

In~\cite{BH:07}, the authors argue that the optimal decay in their geometry should be of the form $\frac{1}{t^{1-\eps}}$, for all $\eps>0$. They provide conditions on the damping coefficient $b(x)$ under which one can obtain such decay rates, and wonder whether this is true in general.
Our main theorem (see Theorem~\ref{th: stabilization torus} below)
extends these results to more general damping functions $b$ on the two-dimensional torus.

\medskip
On the other hand, Christianson~\cite{Christ:10} proved that the
energy decays at rate $e^{-C\sqrt{t}}$ for some $C>0$, in the case
where the trapped set is a hyperbolic closed
geodesic. Schenck~\cite{Sch:11} proved an energy decay at rate
$e^{-Ct}$ on manifolds with negative sectional curvature, if the
trapped set is ``small enough'' in terms of topological pressure (for
instance, a small \nhd of a closed geodesic), and if the damping is
``large enough'' (that is, starting from a damping function $b$,
$\beta b$ will work for any $\beta>0$ sufficiently large).
In these two papers, the geodesic flow near the trapped set enjoys strong instability
properties: the flow on the trapped set is uniformly hyperbolic, in
particular all trajectories are exponentially unstable.

These cases confirm the idea that the decay rates of the energy
strongly depends on the stability of trapped trajectories.

\bigskip

One may now want to compare these geometric situations to situations where the Schr\"odinger group is observable (or, equivalently, controllable), \ie for which there exist $C>0$ and $T>0$ such that, for all $u_0 \in L^2(M)$, we have
\begin{equation}
\label{eq: obs schrodinger}
\|u_0\|_{L^2(M)}^2 \leq C \int_0^T \|\sqrt{b} \ e^{-it \Delta} u_0\|_{L^2(M)}^2 dt.
\end{equation}
The conditions under which this property holds are also known to be related to stability of the geodesic flow. 
In particular, the works \cite{BLR:92}, \cite{LR:05}, \cite{BH:07,Nish:09} and \cite{Christ:10, Sch:11} can be seen as counterparts for damped wave equations of the articles \cite{Leb:92}, \cite{Har:89plaque, Jaffard:90}, \cite{BZ:04} and  \cite{AR:11}, respectively, in the context of observation of the Schr\"odinger group.

\bigskip
Our main results are twofold. First, we clarify (in an abstract
setting) the link between the observability (or the controllability)
of the Schr\"odinger equation and polynomial decay for the damped wave
equation. This follows the spirit of~\cite{Har:89}, \cite{Miller:05},
exploring the links between the different equations and their control
properties (\eg observability, controllability,
stabilization...). More precisely, we prove that the controllability
of the Schr\"odinger equation implies a polynomial decay at rate
$\frac{1}{\sqrt{t}}$ for the damped wave equation (Theorem~\ref{th: schrodinger-waves}).

Second, we study precisely the damped wave equation on the flat torus
$\T^2$ in case GCC fails. We give the following {\it a priori} lower bound on the decay rate, revisiting the argument of~\cite{BH:07}: \eqref{eq: stabilization} is not stable at a better rate than $\frac{1}{t}$, provided that GCC is not satisfied. In this situation, the Schr\"odinger group is known to be controllable (see~\cite{Jaffard:90}, \cite{Kom:92} and the more recent works~\cite{AM:11} and \cite{BZ:11}). Thus, one cannot hope to have a decay better than polynomial in our previous result, \ie under the mere assumption that the Schr\"odinger flow is observable. 

The remainder of the paper is devoted to studying the gap between the {\it a priori} lower and upper bounds given respectively by $\frac{1}{t}$ and $\frac{1}{\sqrt{t}}$ on flat tori.
For {\em smooth} nonvanishing damping coefficient $b(x)$, we prove that the energy decays at rate $\frac{1}{t^{1-\eps}}$ for all $\eps >0$. This result holds without making any dynamical assumption on the damping coefficient, but only on the order of vanishing of $b$. It generalises a result of~\cite{BH:07}, which holds in the case where $b$ is invariant in one direction.
Our analysis is, again, inspired by the recent microlocal approach proposed in~\cite{AM:11} and \cite{BZ:11} for the observability of the Schr\"odinger group. More precisely, we follow here several ideas and tools introduced in \cite{Macia:10} and \cite{AM:11}.

In the situation where $b$ is a characteristic function of a vertical
strip of the torus (hence discontinuous), St\'ephane Nonnenmacher
proves in Appendix~\ref{app: stephane} that the decay rate cannot be
faster than $\frac{1}{t^{2/3}}$. This is done by explicitly computing
the high frequency eigenvalues of the damped wave operator which are
closest to the imaginary axis (see for instance the figures
in~\cite{AL:03, AL:12}). The fact that the decay rate $1/t$ is not
achieved in this situation was observed in the numerical computations presented in~\cite{AL:12}. 

In contrast to the control problem for the Sch\"odinger equation, this
result shows that the stabilization of the wave equation is not only sensitive to the global properties of the geodesic flow, but also to the rate at which the damping function vanishes.

\section{Main results of the paper}
Our first result can be stated in a general abstract setting that we now introduce. We come back to the case of the torus afterwards.

\subsection{The damped wave equation in an abstract setting}
\label{sub: abstract setting}
Let $H$ and $Y$ be two Hilbert spaces (\resp the state space and the observation/control space) with norms $\| \cdot \|_H$ and $\| \cdot \|_Y$, and associated inner products $( \cdot , \cdot )_H$ and $( \cdot , \cdot )_Y$. 

We denote by $A : D(A)\subset H \to H$ a {\em nonnegative} selfadjoint operator with compact resolvent, and $B \in \calL(Y;H)$ a control operator. We recall that $B^* \in \calL(H;Y)$ is defined by $(B^* h ,y)_Y = (h, B y)_H$ for all $h \in H$ and $y \in Y$.

\begin{definition}
We say that the system
\begin{equation}
\label{eq: obs abstract Schrodinger}
\d_t u + i A u = 0 , \quad y = B^*u  ,
\end{equation}
is observable in time $T$ if there exists a constant $K_T>0$ such that, for all solution of \eqref{eq: obs abstract Schrodinger}, we have
$$
\|u(0)\|_H^2 \leq K_T \int_0^T \|y(t)\|_Y^2 dt .
$$
\end{definition}
We recall that the observability of \eqref{eq: obs abstract Schrodinger} in time $T$ is equivalent to the exact controllability in time $T$ of the adjoint problem
\begin{equation}
\label{eq: control abstract Schrodinger}
\d_t u + i A u = B f , \quad u(0) = u_0  ,
\end{equation}
(see for instance \cite{Leb:92} or \cite{RTTT:05}). More precisely, given $T>0$, the exact controllability in time $T$ is the ability of finding for any $u_0 , u_1 \in H$ a control function $f \in L^2(0,T; Y)$ so that the solution of \eqref{eq: control abstract Schrodinger} satisfies $u(T) = u_1$. 

\bigskip
We equip $\H= D(A^\frac12)\times H$ with the graph norm 
$$
\|(u_0,u_1)\|_\H^2 = \|(A+\id)^\frac12 u_0 \|_H^2 + \| u_1 \|_H^2, 
$$
and define the seminorm
$$
|(u_0,u_1)|_\H^2 = \|A^\frac12 u_0 \|_H^2 + \| u_1 \|_H^2.
$$
Of course, if $A$ is coercive on $H$, $|\cdot|_\H$ is a norm on $\H$ equivalent to $\|\cdot \|_\H$.

We also introduce in this abstract setting the damped wave equation on the space $\H$,
\begin{align}
\label{eq: damped abstract waves}
\begin{cases}
\d_t^2 u + A u + B B^* \d_t u = 0, \\
(u, \d_t u)|_{t=0} = (u_0 , u_1) \in \H ,
\end{cases}
\end{align}
which can be recast on $\H$ as a first order system
\begin{equation}
\label{eq: first order eqation}
\left\{
\begin{array}{l}
\d_t U = \A U , \\
U|_{t=0} = \transp(u_0 , u_1) ,
\end{array}
\right.
\quad 
U = 
\left(
\begin{array}{c}
u \\
\d_t u
\end{array}
\right) , \quad
\A = 
\left(
\begin{array}{cc}
0   &  \id \\
- A & - BB^*
\end{array}
\right)  , \quad
D(\A) =  D(A)\times D(A^\frac12) .
\end{equation}
The compact injections $D(A)\hookrightarrow D(A^\frac12)\hookrightarrow H$ imply that $D(\A)\hookrightarrow \H$ compactly, and that the operator $\A$ has a compact resolvent.

We define the energy of solutions of \eqref{eq: damped abstract waves} by
$$
E(u,t) = \frac{1}{2} \big( \|A^\frac12 u\|_H^2 + \|\d_t u\|_H^2 \big)
= \frac12 |(u,\d_t u)|^2_{\H^2}.
$$

\begin{definition}
\label{def: stable at rate}
Let $f$ be a function such that $f(t) \to 0$ when $t \to +\infty$.
We say that System \eqref{eq: damped abstract waves} is stable at rate $f(t)$ if there exists a constant $C>0$ such that for all $(u_0 , u_1) \in D(\A)$, we have 
$$
E(u,t)^{\frac12} \leq C f(t) |\A (u_0 , u_1)|_{\H} , 
\quad \text{for all } t>0 .
$$
If it is the case, for all $k>0$, there exists a constant $C_k >0$ such that for all $(u_0 , u_1) \in D(\A^k)$, we have (see for instance \cite[page 767]{BD:08})
$$
E(u,t)^{\frac12} \leq C_k \big(f(t)\big)^k\|\A^k (u_0 , u_1)\|_{\H}  , 
\quad \text{for all } t>0 .
$$
\end{definition}

\begin{theorem}
\label{th: schrodinger-waves}
Suppose that there exists $T>0$ such that System~\eqref{eq: obs abstract Schrodinger} is observable in time $T$. Then System~\eqref{eq: damped abstract waves} is stable at rate $\frac{1}{\sqrt{t}}$.
\end{theorem}
Note that the gain of the $\log(t)^\frac12$ with respect to~\cite{LR:05, BH:07} is not essential in our work. It is due to the optimal characterization of polynomially decaying semigroups obtained by Borichev and Tomilov~\cite{BT:10}.

This Theorem may be compared with the works (both presented in a similar abstract setting) \cite{Har:89} by Haraux, proving that the controllability of wave-type equations in some time is equivalent to uniform stabilization of~\eqref{eq: damped abstract waves}, and \cite{Miller:05} by Miller, showing that the controllability of wave-type equations in some time implies the controllability of Schr\"odinger-type equations in any time.

\bigskip
Note that the link between this abstract setting and that of Problem~\eqref{eq: stabilization} is $H=Y=L^2(M)$, $A = -\Delta$ with $D(A) = H^2(M)$ if $\d M = \emptyset$ and $H^2(M) \cap H^1_0(M)$ otherwise, $B$ is the multiplication in $L^2(M)$ by the bounded function $\sqrt{b}$.

\bigskip
As a first application of Theorem~\ref{th: schrodinger-waves} we obtain a different proof of the polynomial decay results for wave equations of \cite{LR:05} and \cite{BH:07} as consequences of the associated control results for the Schr\"odinger equation of \cite{Har:89plaque} and \cite{BZ:04} respectively.

Moreover, Theorem~\ref{th: schrodinger-waves} provides also several new stability results for System~\eqref{eq: stabilization} in particular geometric situations; namely, in all following situations, the Schr\"odinger group is proved to be observable, and Theorem~\ref{th: schrodinger-waves} gives the polynomial stability at rate $\frac{1}{\sqrt{t}}$ for \eqref{eq: stabilization}:
\begin{itemize}
\item For any nonvanishing $b(x)\geq 0$ in the $2$-dimensional square (\resp torus), as a consequence of \cite{Jaffard:90} (\resp \cite{Macia:10, BZ:11}); for any nonvanishing $b(x)\geq 0$ in the $d$-dimensional rectangle (\resp $d$-dimensional torus) as a consequence of \cite{Kom:92} (\resp \cite{AM:11});
\item If $M$ is the Bunimovich stadium and $b(x)> 0$ on the \nhd of one half disc and on one point of the opposite side, as a consequence of \cite{BZ:04};
\item If $M$ is a $d$-dimensional manifold of constant negative curvature and the set of trapped trajectories (as a subset of $S^*M$, see \cite[Theorem 2.5]{AR:11} for a precise definition) has Hausdorff dimension lower than $d$, as a consequence of \cite{AR:11};
\end{itemize}

Moreover, Lebeau gives in~\cite[Th\'eor\`eme 1 (ii)]{Leb:96} several $2$-dimensional examples for which the decay rate $\frac{1}{\log(2+t)}$ is optimal. For all these geometrical situations, Theorem~\ref{th: schrodinger-waves} implies that the Schr\"odinger group is not observable.

\bigskip
The proof of Theorem~\ref{th: schrodinger-waves} relies on the following characterization of polynomial decay for System~\eqref{eq: damped abstract waves}. 
For $z \in \C$, we define on $H$ the operator $P(z) = A + z^2\id + z BB^*$, with domain $D(P(z)) = D (A)$.
We prove in Lemma~\ref{lemma: check assumptions} below that $P(is)$ is invertible for all $s \in \R$, $s \neq 0$. 

\begin{proposition}
\label{prop: CNS poly sable}
Suppose that 
\begin{equation}
\label{eq: unique continuation eigenvectors}
\text{for any eigenvector } \varphi \text{ of } A \text{, we have } B^* \varphi \neq 0. 
\end{equation}Then, for all $\alpha >0$, the five following assertions are equivalent:
\begin{equation}
\label{eq: CNS poly sable}
 \text{The system~\eqref{eq: damped abstract waves} is stable at rate } \frac{1}{t^\alpha},
 \end{equation}
 \begin{equation}
\label{eq: CNS Borichev tomilov}
 \text{There exist } C>0\text{ and } s_0 \geq 0 \text{ such that for all } s\in \R, |s|\geq s_0 , \
\|(is\id - \A)^{-1}\|_{\calL(\H)} \leq C |s|^{\frac{1}{\alpha}}, 
 \end{equation}
\begin{equation}
\label{eq: CNS Borichev tomilov complex}
\begin{array}{c}
 \text{There exist } C>0\text{ and } s_0 \geq 0 \text{ such that for all } z\in \C, \text{ satisfying } |z|\geq s_0 , \\ 
 \text{and }
|\Re(z)| \leq \frac{1}{C |\Im(z)|^{\frac{1}{\alpha}}}, \ \text{we have }
\|(z \id - \A)^{-1}\|_{\calL(\H)} \leq C |\Im(z)|^{\frac{1}{\alpha}}, 
\end{array}
 \end{equation}
 \begin{equation}
\label{eq: CNS poly stable resolvent 0}
 \text{There exist } C>0\text{ and } s_0 \geq 0 \text{ such that for all } s\in \R, |s|\geq s_0 , \
\|P(is)^{-1}\|_{\calL(H)} \leq C |s|^{\frac{1}{\alpha}-1}, 
\end{equation}
\begin{equation}
\label{eq: CNS poly stable resolvent}
\begin{array}{c}
 \text{There exists } C>0\text{ and } s_0 \geq 0 \text{ such that for all } s\in \R, |s|\geq s_0 \text{ and } u \in D(A), \\
 \|u\|_H^2 \leq C \big( |s|^{\frac{2}{\alpha}-2} \|P(is)u\|_H^2 + |s|^\frac{1}{\alpha} \|B^*u\|_Y^2 \big) .
 \end{array}
\end{equation}
\end{proposition}
This proposition is proved as a consequence of the characterization of polynomial decay for general semigroups in terms of resolvent estimates given in \cite{BT:10}, providing the equivalence between~\eqref{eq: CNS poly sable} and~\eqref{eq: CNS Borichev tomilov}. See also~\cite{BD:08} for general decay rates in Banach spaces.
Note in particular that the proof of a decay rate is reduced to the proof of a resolvent estimate on the imaginary axes. By the way, this estimate implies the existence of a ``spectral gap'' between the spectrum of $\A$ and the imaginary axis, given by~\eqref{eq: CNS Borichev tomilov complex}.

Note finally that the estimates~\eqref{eq: CNS Borichev tomilov},~\eqref{eq: CNS poly stable resolvent 0} and~\eqref{eq: CNS poly stable resolvent} can be equivalently restricted to $s>0$, since $P(-is)\ovl{u} = \ovl{P(is) u}$. 

\subsection{Decay rates for the damped wave equation on the torus}
The main results of this article deal with the decay rate for Problem \eqref{eq: stabilization} on the torus $\T^2 := (\R / 2 \pi \Z)^2$.
In this setting, as well as in the abstract setting, we shall write $P(z)= -\Delta +z^2 +zb(x)$.

First, we give an {\it a priori} lower bound for the decay rate of the
damped wave equation, on $\T^2$, when GCC is ``strongly violated'',
\ie assuming that $\supp(b)$ does not satisfy GCC (instead of $\{b
>0\}$). This theorem is proved by constructing explicit {\it
  quasimodes} for the operator $P(is)$.

\begin{theorem}
\label{th: lower bound torus}
Suppose that there exists $(x_0 , \xi_0) \in T^*\T^2$, $\xi_0 \neq 0$, such that 
$$
\ovl{\{b>0\}} \cap \{x_0 + \tau \xi_0, \tau \in \R\} = \emptyset .
$$
Then there exist two constants $C>0$ and $\kappa_0>0$ such that for all $n \in \N$, 
\begin{align} 
\|P(i n \kappa_0)^{-1}\|_{\calL(L^2(\T^2))} \geq C.
\end{align}
\end{theorem}
As a consequence of Proposition~\ref{prop: CNS poly sable}, polynomial stabilization at rate $\frac{1}{t^{1 + \eps}}$ for $\eps >0$ is not possible if there is a strongly trapped ray (\ie that does not intersect $\supp(b)$). 
More precisely, in such geometry, Theorem~\ref{th: lower bound torus} combined with Lemma~\ref{lemma: conditions stable resolvent} and \cite[Proposition~1.3]{BD:08} shows that  $m_1(t) \geq \frac{C}{1+t}$, for some $C>0$ (with the notation of \cite{BD:08} where $m_1(t)$ denotes the best decay rate).

\bigskip
Then, the main goal of this paper is to explore the gap between the {\it a priori} upper bound $\frac{1}{\sqrt{t}}$ for the decay rate, given by Theorem~\ref{th: schrodinger-waves}, and the {\it a priori} lower bound $\frac{1}{t}$ of Theorem~\ref{th: lower bound torus}. Our results are twofold (somehow in two opposite directions) and concern either the case of smooth damping functions $b$, or the case $b= \mathds{1}_U$, with $U\subset \T^2$.

\subsubsection{The case of smooth damping coefficients}

Our main result deals with the case of smooth damping coefficients. Without any geometric assumption, but with an additional hypothesis on the order of vanishing of the damping function $b$, we prove a weak converse of Theorem~\ref{th: lower bound torus}.

\begin{theorem}
\label{th: stabilization torus}
Let $M = \T^2$ with the standard flat metric. There exists $\eps_0 >0$ satisfying the following property. Suppose that $b$ is a nonnegative nonvanishing function on $\T^2$ satisfying $\sqrt{b}\in \Con^\infty(\T^2)$ and that there exist $\eps \in (0, \eps_0)$ and $C_\eps>0$ such that
\begin{equation}
\label{eq: assumption b}
|\nabla b (x)| \leq C_\eps b^{1-\eps}(x) , \quad  \text{for } x \in \T^2 .
\end{equation}
Then, there exist $C>0$ and $s_0 \geq 0$ such that for all $s\in \R$, $|s|\geq s_0$,
\begin{align} 
\label{eq: resolvent estimates torus}
\|P(is)^{-1}\|_{\calL(L^2(\T^2))} \leq C|s|^{\delta} , \quad \text{with } \delta = 8 \eps
\end{align}
As a consequence of Proposition~\ref{prop: CNS poly sable}, in this situation, the damped wave equation~\eqref{eq: stabilization} is stable at rate $\frac{1}{t^{\frac{1}{1+\delta}}}$.
\end{theorem}
Following carefully the steps of the proof, one sees that $\eps_0 = \frac{1}{76}$ works, but the proof is not optimized with respect to this parameter, and it is likely that it could be much improved. 

\medskip
One of the main difficulties in understanding the decay rates is that there exists no general monotonicity property of the type ``$b_1(x) \leq b_2(x)$ for all $x \ \Longrightarrow$ the decay rate associated to the damping $b_2$ is larger (or smaller) than the decay rate associated to the damping $b_1$''.
This makes a significant difference with observability or controllability problems of the type~\eqref{eq: obs schrodinger}.

\medskip
Assumption~\eqref{eq: assumption b} is only a local assumption in a \nhd of $\d \{b>0\}$ (even if it is stated here globally on $\T^2$). Far from this set, \ie on each compact set $\{b\geq b_0\}$ for $b_0>0$, the constant $C_\eps$ can be choosen uniformly, depending only on $b_0$, and not on $\eps$.
Hence, $\eps$ somehow quantifies the vanishing rate of the damping function $b$.

\medskip
An interesting situation is when the smooth function $b$ vanishes like
$e^{-\frac{1}{x^\alpha}}$ in smooth local coordinates, for some
$\alpha >0$. In this case, Assumption~\eqref{eq: assumption b} is
satisfied for any $\eps>0$, and the associated damped wave
equation~\eqref{eq: stabilization} is stable at rate
$\frac{1}{t^{1-\delta}}$ for any $\delta >0$. This shows that the
lower bound given by Theorem~\ref{th: lower bound torus}, as well as
the decay rate $\frac{1}{t}$, are sharp in general. This phenomenon
had already been remarked by Burq and Hitrik in~\cite{BH:07} in the
case where $b$ is invariant in one direction. 

Typical smooth functions not satisfying Assumption~\eqref{eq: assumption b} are for instance functions vanishing like $\sin(\frac{1}{x})^2 e^{-\frac{1}{x}}$. We do not have any idea of the decay rate achieved in this case (except for the {\it a priori} bounds $\frac{1}{\sqrt{t}}$ and $\frac{1}{t}$).

\bigskip
Theorem~\ref{th: stabilization torus} generalises the result of ~\cite{BH:07}, which only holds if $b$ is assumed to be invariant in one direction.
Our proof is based on ideas and tools developped in~\cite{Macia:10, AM:11} and especially on two-microlocal semiclassical measures. One of the key technical points appears in Section~\ref{section: existence chi}: we have to construct, for each trapped direction, a cutoff function invariant in that direction and adapted to the damping coefficient $b$. We do not know how to adapt this technical construction to tori of higher dimension, $d>2$; hence we do not know whether Theorem~\ref{th: stabilization torus} holds in higher dimension (although we have no reason to suspect it should not hold). Only in the particular case where $b$ is invariant in $d-1$ directions can our methods (or those of~\cite{BH:07}) be applied to prove the analogue of 
Theorem~\ref{th: stabilization torus}.

\bigskip
Note that if GCC is satisfied, one has (on a general compact manifold $M$) for some $C>1$ and all $|s| \geq s_0$ the estimate
\begin{align} 
\label{eq: resolvent uniform decay}
\|P(is)^{-1}\|_{\calL(L^2(M))} \leq C |s|^{-1}.
\end{align}
instead of \eqref{eq: resolvent estimates torus}. Estimate~\eqref{eq: resolvent uniform decay} is in turn equivalent to uniform stabilization (see \cite{Hua:85} together with Lemma~\ref{lemma: conditions stable resolvent} below).

\begin{remark}
\label{rem: square}
As a consequence of Theorem~\ref{th: stabilization torus} on the
torus, we can deduce that the decay rate $t^{-\frac{1}{1+\delta}}$
also holds for Equation~\eqref{eq: stabilization} if $M= (0,\pi)^2$ is
the square, one takes
with Dirichlet or Neumann boundary conditions, and the damping
function $b$ is smooth, vanishes near $\partial M$ and satisfies Assumption~\eqref{eq: assumption b}.
First, we extend the function $b$ as an even (with respect to both
variables) smooth function on the larger square $(-\pi,\pi)^2$, and using the injection $\imath :(-\pi,\pi)^2 \to \T^2$, as a smooth function on $\T^2$, still satisfying~\eqref{eq: assumption b}. Moreover, $D(\Delta_D)$ (\resp $D(\Delta_N)$) on $(0,\pi)^2$ can be identified as the closed subspace of odd (\resp even) functions of $D(\Delta_D)$ (\resp $D(\Delta_N)$) on $(-\pi,\pi)^2$. Using again the injection $\imath$, it can also be identified with a closed subspace of $H^2(\T^2)$. 
The estimate
\begin{align*} 
\|u\|_{L^2(\T^2)} \leq C |s|^\delta \|P(is) u\|_{L^2(\T^2)} \quad \text{for all } u \in H^2(\T^2) ,
\end{align*}
is thus also true on the square $(0,\pi)^2$ for Dirichlet or Neumann boundary conditions.
In particular, this strongly improves the results of~\cite{LR:05}.

The lower bound of Theorem~\ref{th: lower bound torus} can be
similarly extended to the case of a square with Dirichlet or Neumann
boundary conditions, implying that the rate $\frac{1}{t}$ is optimal if
GCC is strongly violated.
\end{remark}


\subsubsection{The case of discontinuous damping functions}

%
%
%

Appendix~\ref{app: stephane} (by St\'ephane Nonnenmacher) deals with
the case where $b$ is the characteristic function of a vertical strip,
\ie $b =\tB \mathds{1}_U$, for some $\tB>0$ and $U = (a, b)\times \T
\subset \T^2$. 
Due to the invariance of $b$
in one direction, the spectrum of the damped wave operator $\A$ splits
into countably many ``branches'' of eigenvalues. This structure of
the spectrum is illustrated in the numerics of~\cite{AL:03, AL:12}.

The branch closest to the imaginary axis is explicitly computed, it contains a sequence of eigenvalues $(z_i)_{i \in \N}$ such that $\Im z_i\to\infty$ and $|\Re z_i|\leq \frac{C_0}{(\Im z_i)^{3/2}}$.
This result is in agreement with the numerical tests given in~\cite{AL:12}.

As a consequence, for any $\eps>0$ and $C>0$, the strip  
$\big\{|\Re z|\leq C |\Im(z)|^{-3/2+\eps}\big\}$ contains infinitely
many poles of the resolvent $(z \id - \A)^{-1}$, so item~\eqref{eq: CNS Borichev tomilov complex} in Proposition~\ref{prop: CNS poly
  sable} implies the following obstruction to the
stability of this damped system~:
\begin{corollary}\label{cor:obstruction}
For any $\eps>0$, the damped wave equation~\eqref{eq: stabilization} on
$\T^2$ with the damping function \eqref{e:b(x)} cannot be stable at
the rate $\frac{1}{t^{2/3+\eps}}$.

The same result holds on the square with Dirichlet or Neumann
boundary conditions.
\end{corollary}

More precisely, in this situation, Lemma~\ref{lemma: conditions stable resolvent} and \cite[Proposition~1.3]{BD:08} yield that $m_1(t) \geq \frac{C}{(1+t)^{2/3}}$, for some $C>0$ (with the notation of \cite{BD:08} where $m_1(t)$ denotes the best decay rate).

This corollary shows in particular that the regularity conditions in Theorem~\ref{th: stabilization torus}
cannot be completely disposed of if one wants a stability at the rate $1/t^{1-\eps}$ for small $\eps$.

\subsection{Some related open questions}

The various results obtained in this article lead to several open questions.  

\begin{enumerate}
\item In the case where $b$ is the characteristic function of a
  vertical strip, our analysis shows that the best decay rate lies
  somewhere between $\frac{1}{t^{\frac12}}$ and
  $\frac{1}{t^{\frac23}}$, but the ``true'' decay rate is not yet clear.
\item 
It would also be interesting to investigate the spectrum and the decay
rates for damping functions $b$ invariant in one direction, but having
a less singular behaviour than a characteristic function. In
particular, is it possible to give a precise link between the
vanishing rate of $b$ and the decay rate?

\item In the general setting of Section~\ref{sub: abstract setting}
  (as well as in the case of the damped wave equation on the torus),
  is the {\it a priori} upper bound $\frac{1}{t^{\frac12}}$ for the
  decay rate optimal? 

\item For smooth damping functions vanishing like $e^{-\frac{1}{x^\alpha}}$, Theorem~\ref{th: stabilization torus} yields stability at rate $\frac{1}{t^{1-\delta}}$ for all $\delta>0$. Is the decay rate $\frac{1}{t}$ reached in this situation? Can one find a damping function $b$ such that the decay rate is exactly $\frac{1}{t}$?

\item The lower bound of of Theorem~\ref{th: lower bound torus} is still valid in higher dimensional tori. Is there an analogue of Theorem~\ref{th: stabilization torus} (\ie for general ``smooth'' damping functions) for  $\T^d$, with $d\geq 3$?

\end{enumerate}

%
%
%

\part{Resolvent estimates and stabilization in the abstract setting}

\section{Proof of Theorem~\ref{th: schrodinger-waves} assuming Proposition~\ref{prop: CNS poly sable}}
To prove Theorem~\ref{th: schrodinger-waves}, we express the observability condition as a resolvent estimate (also known as the Hautus test), as introduced by Burq and Zworski~\cite{BZ:04}, and further developed by Miller~\cite{Miller:05} and Ramdani, Takahashi, Tenenbaum and Tucsnak~\cite{RTTT:05}. For a survey of this notion, we refer to the book~\cite[Section~6.6]{TW:09}.

In particular \cite[Theorem~5.1]{Miller:05} (or \cite[Theorem~6.6.1]{TW:09}) yields that System~\eqref{eq: obs abstract Schrodinger} is observable in some time $T>0$ if and only if there exists a constant $C>0$ such that we have
$$
\|u\|_H^2 \leq C \big( \|(A - \lambda \id)u \|_H^2 + \|B^*u\|_Y^2 \big), \quad \text{ for all } \lambda \in \R \text{ and } u \in D(A). 
$$

As a first consequence, Assumption~\eqref{eq: unique continuation eigenvectors} is satisfied and Proposition~\ref{prop: CNS poly sable} applies in this context.
Moreover, we have, for all $s \in \R$ and $u \in D(A)$,
\begin{align}
\label{eq: estimate loss schrodinger-waves}
\|u\|_H^2 
& \leq C \big( \|(A - s^2 \id + is BB^* - is BB^*)u\|_H^2 + \|B^*u\|_Y^2 \big) \nonumber \\
& \leq C \big( \|P(is) u\|_H^2 + s^2 \|BB^*u\|_H^2 + \|B^*u\|_Y^2 \big) 
\end{align}
Since $B \in \calL(Y ;H)$, we obtain for $s \geq 1$ and for some $C>0$,
\begin{align*}
\|u\|_H^2  \leq C \big( \|P(is) u\|_H^2 + s^2 \|B^*u\|_Y^2\big)  
\leq C \big( s^2 \|P(is) u\|_H^2 + s^2 \|B^*u\|_Y^2\big) .
\end{align*}
Proposition~\ref{prop: CNS poly sable} then yields the polynomial stability at rate $\frac{1}{\sqrt{t}}$ for \eqref{eq: damped abstract waves}. This concludes the proof of Theorem~\ref{th: schrodinger-waves}.
\hfill \qedsymbol \endproof


\section{Proof of Proposition~\ref{prop: CNS poly sable}}
Our proof strongly relies on the characterization of polynomially stable semigroups, given in \cite[Theorem 2.4]{BT:10}, which can be reformulated as follows. 

\begin{theorem}[\cite{BT:10}, Theorem 2.4]
\label{th: BorTom}
Let $(e^{t\dot{\A}})_{t\geq 0}$ be a bounded $\Con^0$-semigroup on a Hilbert space $\dot{\H}$, generated by $\dot{\A}$. Suppose that $i \R \cap \Sp(\dot{\A}) = \emptyset$. Then, the following conditions are equivalent:
\begin{equation}
\label{eq: BorTom stable}
\|e^{t\dot{\A}} \dot{\A}^{-1}\|_{\calL(\dot{\H})} = \O(t^{- \alpha}), \quad \text{as } t \to + \infty ,
\end{equation}
\begin{equation}
\label{eq: BorTom resolvent}
\|(is\id - \dot{\A})^{-1}\|_{\calL(\dot{\H})} = \O(|s|^{\frac{1}{\alpha}}), \quad \text{as } s \to \infty .
\end{equation}
\end{theorem}

\bigskip
Let us first describe some spectral properties of the operator $\A$ defined in~\eqref{eq: first order eqation}.

\begin{lemma}
\label{lemma: check assumptions}
The spectrum of $\A$ contains only isolated eigenvalues and we have 
$$
\Sp(\A) \subset \left( \big( - \frac12 \|B^*\|_{\calL(H;Y)}^2 , 0 \big) + i\R\right) \cup 
\left([ - \|B^*\|_{\calL(H;Y)}^2 , 0] + 0 i \right)  ,
$$
with $\ker(\A) = \ker(A)\times \{0\}$.

Moreover, the operator $P(z)$ is an isomorphism from $D(A)$ onto $H$ if and only if $z \notin \Sp(\A)$. If this is satisfied, we have 
\begin{equation}
\label{eq: resolvent A wrt P}
(z\id - \A)^{-1} =
\left(
\begin{array}{cc}
P(z)^{-1}(BB^* + z\id) & P(z)^{-1} \\
P(z)^{-1}(zBB^* + z^2\id) -\id& z P(z)^{-1} 
\end{array}
\right) .
\end{equation}
\end{lemma}

The localization properties for the spectrum of $\A$, stated in the first part of this lemma are illustrated for instance in~\cite{AL:03} or~\cite{AL:12}. 

This Lemma leads us to introduce the spectral projector of $\A$ on $\ker(\A)$, given by 
$$
\Pi_0 = \frac{1}{2i\pi}\int_{\gamma} (z\id - \A)^{-1} dz \in \calL(\H),
$$
where $\gamma$ denotes a positively oriented circle centered on $0$ with a radius so small that $0$ is the single eigenvalue of $\A$ in the interior of $\gamma$.
We set $\dot{\H} = (\id - \Pi_0)\H$ and equip this space with the norm
$$
\|(u_0 , u_1)\|_{\dot{\H}}^2 := |(u_0 , u_1)|_{\H}^2 
= \|A^{\frac12}u_0\|_{H}^2 + \|u_1 \|_{H}^2 ,
$$
and associated inner product.
This is indeed a norm on $\dot{\H}$ since $\|(u_0 , u_1)\|_{\dot{\H}} =0$ is equivalent to $(u_0 , u_1)\in \ker(A)\times \{0\} = \Pi_0\H$.

Besides, we set $\dot{\A} = \A|_{\dot{\H}}$ with domain $D(\dot{\A}) = D(\A) \cap \dot{\H}$. 
A first remark is that $\Sp(\dot{\A}) = \Sp(\A) \setminus \{0\}$, so that $\Sp(\dot{\A}) \cap i\R = \emptyset$.

The remainder of the proof consists in applying Theorem~\ref{th: BorTom} to the operator $\dot{\A}$ in $\dot{\H}$. We first check the assumptions of Theorem~\ref{th: BorTom} and describe the solutions of the evolution problem~\eqref{eq: first order eqation} (or equivalently~\eqref{eq: damped abstract waves}).

\begin{lemma}
\label{lemma: semigroups}
The operator $\dot{\A}$ generates a contraction $\Con^0$-semigroup on~$\dot{\H}$, denoted $(e^{t\dot{\A}})_{t\geq 0}$. Moreover, for all initial data $U_0 \in \H$, Problem~\eqref{eq: first order eqation} (or equivalently~\eqref{eq: damped abstract waves}) has a unique solution $U \in \Con^0(\R^+ ;\H)$, issued from $U_0$, that can be decomposed as 
\begin{equation}
\label{eq: decomposition semigroup}
U(t) = e^{t\dot{\A}} (\id -\Pi_0) U_0 + \Pi_0 U_0 , \quad \text{for all } t\geq 0 .
\end{equation}
\end{lemma} 

As a consequence, we can apply Theorem~\ref{th: BorTom} to the semigroup generated by $\dot{\A}$. The proof of Proposition~\ref{prop: CNS poly sable} will be achieved when the following lemmata are proved.

\begin{lemma}
\label{lemma: conditions stable}
Conditions \eqref{eq: CNS poly sable} and \eqref{eq: BorTom stable} are equivalent.
\end{lemma}
\begin{lemma}
\label{lemma: equivalent conditions lemma}
Conditions \eqref{eq: CNS poly stable resolvent 0} and \eqref{eq: CNS poly stable resolvent} are equivalent. Conditions~\eqref{eq: CNS Borichev tomilov} and~\eqref{eq: CNS Borichev tomilov complex} are equivalent.
\end{lemma}
\begin{lemma}
\label{lemma: conditions stable resolvent}
There exist $C>1$ and $s_0 >0$ such that for $s \in \R$, $|s|\geq s_0$, 
\begin{align}
\label{eq: estimate A estimate dotA}
 \|(is\id - \dot{\A})^{-1}\|_{\calL(\dot{\H})} - \frac{C}{|s|} 
\leq \|(is\id - \A)^{-1}\|_{\calL(\H)} 
\leq  \|(is\id - \dot{\A})^{-1}\|_{\calL(\dot{\H})} + \frac{C}{|s|}  ,
\end{align}
and 
\begin{align}
\label{eq: estimate P equiv A}
C^{-1}|s| \|P(is)^{-1}\|_{\calL(H)} \leq  \|(is\id - \A)^{-1}\|_{\calL(\H)}
\leq C \left( 1 + |s|\|P(is)^{-1}\|_{\calL(H)} \right).
\end{align}
In particular this implies that \eqref{eq: BorTom resolvent}, \eqref{eq: CNS Borichev tomilov} and~\eqref{eq: CNS poly stable resolvent 0} are equivalent.
\end{lemma}
The proof of Lemma~\ref{lemma: conditions stable resolvent} is more or less classical and we follow~\cite{Leb:96, BH:07}.

\begin{proof}[Proof of Lemma~\ref{lemma: check assumptions}]
As $\A$ has compact resolvent, its spectrum contains only isolated eigenvalues. 
Suppose that $z \in \Sp (\A)$, then we have, for some $(u_0 ,u_1) \in D(\A) \setminus \{0\}$,
\begin{equation*}
\left\{
\begin{array}{rl}
u_1 &= z u_0  ,\\
- A u_0 - BB^* u_1 &= z u_1,
\end{array}
\right.
\end{equation*}
and in particular
\begin{equation}
\label{eq: Pzu0}
A u_0 + z^2 u_0 + zBB^* u_0 = 0,
\end{equation}
with $u_0 \in D(A)\setminus \{0\}$. 

Suppose that $z \in i\R$, then, this yields $A u_0 - \Im(z)^2 u_0 + i \Im(z) BB^* u_0 = 0$. Following~\cite{Leb:96}, taking the inner product of this equation with $u_0$ yields $i\Im(z) \|B^* u_0 \|_Y^2 = 0$. Hence, either $\Im(z) = 0$, or $B^* u_0 = 0$. In the first case, $A u_0 = 0$, \ie $u_0 \in \ker(A)$, and $u_1 = 0$. This yields $\ker(\A) \subset \ker(A) \times \{0\}$ (and the other inclusion is clear). In the second case, $u_0$ is an eigenvector of $A$ associated to the eigenvalue $\Im(z)^2$ and satisfies $B^* u_0 = 0$, which is absurd, according to Assumption~\eqref{eq: unique continuation eigenvectors}. Thus, $\Sp (\A) \cap i \R \subset \{0\}$.

Now, for a general eigenvalue $z \in \C$, taking the inner product of \eqref{eq: Pzu0} with $u_0$ yields
\begin{equation}
\label{eq: system Im Re}
\left\{
\begin{array}{rl}
(Au_0 , u_0)_H + (\Re(z)^2 - \Im(z)^2) \|u_0\|_H^2 + \Re(z)\|B^* u_0\|_Y^2 &= 0 ,\\
2 \Re(z) \Im(z) \|u_0\|_H^2 + \Im(z)\|B^* u_0\|_Y^2 &= 0 . 
\end{array}
\right.
\end{equation}
If $\Im(z) \neq 0$, then, the second equation of \eqref{eq: system Im Re} together with $\Sp (\dot{\A}) \cap i \R \subset \{0\}$ gives
\begin{equation*}
0 > \Re(z)= -\frac{1}{2} \frac{\|B^* u_0\|_Y^2}{\|u_0\|_H^2} \geq -\frac{1}{2} \|B^*\|_{\calL(H;Y)}^2 .
\end{equation*}
If $\Im(z) = 0$, then, the first equation of \eqref{eq: system Im Re} together with $(\dot{A}u_0 , u_0)_H \geq 0$ gives $-\Re(z)\|B^* u_0\|_Y^2 \geq \Re(z)^2 \|u_0\|_H^2$, which yields
\begin{equation*}
0 \geq \Re(z)\geq - \|B^*\|_{\calL(H;Y)}^2 .
\end{equation*}

Following~\cite{Leb:96}, we now give the link between $P(z)^{-1}$ and $(z\id - \A)^{-1}$ for $z \notin \Sp(\A)$. Taking $F = (f_0,f_1)\in \H$, and $U = (u_0 , u_1)$, we have
\begin{equation}
\label{eq: A invert iff P invert}
F = (z\id - \A)U
\Longleftrightarrow
\left\{
\begin{array}{l}
u_1 =  z u_0 -f_0 , \\
P(z) u_0 = f_1 + (BB^* + z\id) f_0.
\end{array}
\right.
\end{equation}
As a consequence, we obtain that $P(z) : D(A) \to H$ is invertible if and only if $(z\id - \A) : D(\A) \to \H$ is invertible, \ie if and only if $z \notin \Sp(\A)$. Moreover, for such values of $z$, System \eqref{eq: A invert iff P invert} is equivalent to 
\begin{equation*}
\left\{
\begin{array}{l}
u_0 = P(z)^{-1}f_1 + P(z)^{-1}(BB^* + z\id) f_0 ,\\
u_1 = z P(z)^{-1}f_1 + z P(z)^{-1}(BB^* + z\id) f_0 -f_0 ,
\end{array}
\right.
\end{equation*}
which can be rewritten as \eqref{eq: resolvent A wrt P}. This concludes the proof of Lemma~\ref{lemma: check assumptions}.
\end{proof}

\begin{proof}[Proof of Lemma~\ref{lemma: semigroups}]
Let us check that $\dot{\A}$ is a maximal dissipative operator on $\dot{\H}$ \cite{Pazy:83}. First, it is dissipative since, for $U = (u_0 ,u_1) \in D(\dot{\A})$,
$$
( \dot{\A} U ,U)_{\dot{\H}} = ( A^\frac12 u_1 , A^\frac12 u_0)_{H} 
-( A u_0 , u_1)_{H} - ( BB^* u_1 , u_1)_{H} =  - \|B^* u_1\|_Y^2 \leq 0.
$$
Next, the fact that $\A - \id$ is onto is a consequence of Lemma~\ref{lemma: check assumptions}. Hence, for all $F \in \dot{\H} \subset \H$, there exists $U\in D(\A)$ such that $(\A - \id) U = F$. Applying $(\id -\Pi_0)$ to this identity yields $(\dot{\A} - \id) (\id -\Pi_0)U = F$, so that $\dot{\A} - \id : D(\dot{\A}) \to \dot{\H}$ is onto.
According to the Lumer-Phillips Theorem (see for instance~\cite[Chapter 1, Theorem 4.3]{Pazy:83}) $\dot{\A}$ generates a contraction $\Con^0$-semigroup on $\dot{\H}$. Then, Formula \eqref{eq: decomposition semigroup} directly comes from the linearity of Equation~\eqref{eq: first order eqation} (or equivalently~\eqref{eq: damped abstract waves}) together with the decomposition of the initial condition $U_0 = (I-\Pi_0)U_0 + \Pi_0 U_0$.
\end{proof}

\begin{proof}[Proof of Lemma~\ref{lemma: conditions stable}]
Condition~\eqref{eq: BorTom stable} is equivalent to the existence of $C>0$ such that for all $t>0$, and $\dot{U}_0 \in \dot{\H}$, we have
\begin{equation*}
\|e^{t\dot{\A}} \dot{\A}^{-1} \dot{U}_0\|_{\dot{\H}} \leq  \frac{C}{t^{\alpha}}\|\dot{U}_0\|_{\dot{\H}} .
\end{equation*}
This can be rephrased as 
\begin{equation}
\label{eq: poly decay dot}
\|e^{t\dot{\A}} \dot{U}_0\|_{\dot{\H}} \leq  \frac{C}{t^{\alpha}}\|\dot{\A}\dot{U}_0\|_{\dot{\H}}, 
\end{equation}
for all $t>0$, and $\dot{U}_0 \in D(\dot{\A})$. Now, take any $U_0 = (u_0,u_1) \in D(\A)$, and associated projection $\dot{U}_0 = (\id - \Pi_0) U_0 \in D(\dot{\A})$. According to \eqref{eq: decomposition semigroup}, we have
$$
E(u,t) = \frac{1}{2} \big( \|A^\frac12 u(t)\|_H^2 + \|\d_t u(t)\|_H^2 \big) 
= \frac12 |e^{t\dot{\A}} \dot{U}_0 + \Pi_0 U_0|_{\H}^2
= \frac12 \|e^{t\dot{\A}}\dot{U}_0\|_{\dot{\H}}^2,
$$
and 
$$
|\A U_0|_{\H} = | \dot{\A} \dot{U}_0 + \A \Pi_0 U_0 |_{\H}
= \|\dot{\A}\dot{U}_0\|_{\dot{\H}} .
$$
This shows that \eqref{eq: poly decay dot} is equivalent to~\eqref{eq: CNS poly sable}, and concludes the proof of Lemma~\ref{lemma: conditions stable}.
\end{proof}

\begin{proof}[Proof of Lemma~\ref{lemma: equivalent conditions lemma}]
First, \eqref{eq: CNS poly stable resolvent 0} clearly implies~\eqref{eq: CNS poly stable resolvent}. To prove the converse, 
for $u \in D(\A)$, we have
$$
(P(is)u, u)_H = \big((A-s^2\id)u, u\big)_H + i s \|B^* u\|_Y^2 .
$$
Taking the imaginary part of this identity gives $ s \|B^* u\|_Y^2 = \Im (P(is)u, u)_H$, so that, using the Young inequality, we obtain for all $\eps>0$,
$$
 |s|^{\frac{1}{\alpha}} \|B^* u\|_Y^2 
 = |s|^{\frac{1}{\alpha}-1} |\Im (P(is)u, u)_H| 
 \leq \frac{|s|^{\frac{2}{\alpha}-2}}{4 \eps}\|P(is)u\|_H^2 + \eps \|u\|_H^2 .
$$
Plugging this into \eqref{eq: CNS poly stable resolvent} and taking $\eps$ sufficiently small, we obtain that for some $C>0$ and $s_0 \geq 0$, for any $s\in \R, |s|\geq s_0$, 
$$
\|u\|_{H}^2 \leq C |s|^{\frac{2}{\alpha}-2}\|P(is)u\|_H^2, 
$$
which yields~\eqref{eq: CNS poly stable resolvent 0}. Hence \eqref{eq: CNS poly stable resolvent 0} and~\eqref{eq: CNS poly stable resolvent} are equivalent. 

\bigskip
Second, Condition~\eqref{eq: CNS Borichev tomilov complex} clearly implies~\eqref{eq: CNS Borichev tomilov} and it only remains to prove the converse. For $z \in \C$, we write $r = \Re(z)$ and $s = \Im(z)$. We have the identity
\begin{equation}
\label{eq: estimate r+is}
((r+is)\id -\A)^{-1} = (is\id -\A)^{-1}\big(\id + r (is\id -\A)^{-1}\big)^{-1} .
\end{equation}
Hence, assuming 
\begin{equation}
\label{eq: condition r s}
\|r(is\id -\A)^{-1} \|_{\calL(\H)} \leq \frac12 ,
\end{equation}
this gives 
$$
\left\|\big(\id + r (is\id -\A)^{-1}\big)^{-1} \right\|_{\calL(\H)}
= \left\|\sum_{k=0}^\infty (-1)^k \big(r (is\id -\A)^{-1}\big)^k\right\|_{\calL(\H)}
\leq 2 .
$$
As a consequence of~\eqref{eq: estimate r+is} and~\eqref{eq: CNS Borichev tomilov}, we then obtain 
\begin{equation*}
\left\|((r+is)\id -\A)^{-1}\right\|_{\calL(\H)}
\leq 2\|(is\id -\A)^{-1} \|_{\calL(\H)} \leq 2C |s|^{\frac{1}{\alpha}},
\end{equation*}
for all $s\geq s_0$, under Condition~\eqref{eq: condition r s}. Finally,~\eqref{eq: CNS Borichev tomilov}  also yields
$$
\|r(is\id -\A)^{-1} \|_{\calL(\H)}\leq |r|C|s|^{\frac{1}{\alpha}}, 
$$
so that Condition~\eqref{eq: condition r s} is realised as soon as $|r|\leq \frac{1}{2C|s|^{\frac{1}{\alpha}}}$. This proves~\eqref{eq: CNS Borichev tomilov complex} and concludes the proof of Lemma~\ref{lemma: equivalent conditions lemma}.
\end{proof}

\begin{proof}[Proof of Lemma~\ref{lemma: conditions stable resolvent}]
To prove \eqref{eq: estimate A estimate dotA}, we first remark that the norms $\| \cdot \|_{\dot{\H}}$ and $\|\cdot\|_{\H}$ are equivalent on $\dot{\H}$, so that the norms $\| \cdot \|_{\calL(\dot{\H})}$ and $\| \cdot \|_{\calL(\H)}$ are equivalent on $\calL(\dot{\H})$. Next, we have $(is\id - \dot{\A})^{-1}(\id - \Pi_0) = (is\id - \A)^{-1}(\id - \Pi_0)$ and
\begin{align*}
\|(is\id - \dot{\A})^{-1}\|_{\calL(\H)} 
& = \|(is\id - \dot{\A})^{-1}(\id - \Pi_0)\|_{\calL(\H)} 
 =\|(is\id - \A)^{-1}(\id - \Pi_0)\|_{\calL(\H)}\\
& \leq  \|(is\id - \A)^{-1}\|_{\calL(\H)} + \|(is\id - \A)^{-1}\Pi_0\|_{\calL(\H)} ,
\end{align*}
together with 
\begin{align*}
\|(is\id - \A)^{-1}\|_{\calL(\H)} 
& = \|(is\id - \dot{\A})^{-1}(\id - \Pi_0) + (is\id - \A)^{-1} \Pi_0\|_{\calL(\H)} \\
& \leq \|(is\id - \dot{\A})^{-1}\|_{\calL(\H)} + \|(is\id - \A)^{-1}\Pi_0\|_{\calL(\H)} .
\end{align*}
Moreover, for $|s|\geq 1$, we have
$$
\|(is\id - \A)^{-1}\Pi_0\|_{\calL(\H)} 
= \|(is)^{-1}\Pi_0\|_{\calL(\H)}
= \frac{1}{|s|} \|\Pi_0\|_{\calL(\H)} = \frac{C}{|s|},
$$
which concludes the proof of \eqref{eq: estimate A estimate dotA}.

\bigskip
Let us now prove~\eqref{eq: estimate P equiv A}. For concision, we set $H_1 = D(A^\frac12)$ endowed with the graph norm $\|u\|_{H_1} = \|(A+ \id)^\frac12 u\|_{H}$ and denote by $H_{-1} = D(A^\frac12)'$ its dual space. The operator $A$ can be uniquely extended as an operator $\calL(H_1 ;H_{-1})$, still denoted $A$ fo simplicity. With this notation, the space $H_{-1}$ can be equipped with the natural norm $\|u\|_{H_{-1}} = \|(A+ \id)^{-\frac12} u\|_{H}$.

As a consequence of Formula \eqref{eq: resolvent A wrt P}, and using the fact that $\Sp(\A) \cap i\R \subset \{0\}$, there exist constants $C>1$ and $s_0>0$ such that for all $s \in \R$, $|s|\geq s_0$, 
\begin{align}
\label{eq: M(s) equiv resA}
C^{-1} M(s)\leq \|(is\id - \A)^{-1}\|_{\calL(\H)} \leq C M(s)
\end{align}
with 
\begin{align}
\label{eq: expression resolvent P(is)}
M(s) =& \Big( \|P(is)^{-1}(BB^* + is\id)\|_{\calL(H_1)} 
+ \|P(is)^{-1}\|_{\calL(H; H_1)}  \nonumber \\
& + \|P(is)^{-1}(isBB^* - s^2\id) -\id \|_{\calL(H_1;H)}
+ \|s P(is)^{-1}\|_{\calL(H)}
\Big)
\end{align}
On the one hand, this direcly yields for $s \in \R$, $|s|\geq s_0$,
$$
|s| \|P(is)^{-1}\|_{\calL(H)} \leq C \|(is\id - \A)^{-1}\|_{\calL(\H)} .
$$
This proves that~\eqref{eq: BorTom resolvent} implies~\eqref{eq: CNS poly stable resolvent 0}.

On the other hand, we have to estimate each term of \eqref{eq: expression resolvent P(is)}. First, using $A u = P(is) u + s^2u - is BB^*u$, we have
\begin{align*}
\|u\|_{H_1}^2 
& = \|A^{\frac12} u\|_H^2 + \|u\|_H^2 
= \big( P(is) u + s^2u - is BB^*u , u \big)_H + \|u\|_H^2 \\
&= \Re \big( P(is) u, u \big)_H  + (s^2+1) \|u\|_H^2 
\leq C \left( \|P(is) u\|_H^2 + (s^2+1) \|u\|_H^2 \right) \\
& \leq C \left( 1 + (s^2+1)\|P(is)^{-1}\|_{\calL(H)}^2 \right) \|P(is)u\|_H^2 ,
\end{align*}
so that 
\begin{align}
\label{eq: estimate first}
\|P(is)^{-1}\|_{\calL(H;H_1)} \leq C \left( 1 + (|s|+1)\|P(is)^{-1}\|_{\calL(H)} \right) .
\end{align}

Second, the same computation for $(P(is)^{-1})^* = (A -s^2 \id - isBB^*)^{-1}$ (the adjoint of $P(is)^{-1}$ in the space $H$) in place of $P(is)^{-1}$ leads to $(P(is)^{-1})^* \in \calL(H;H_1)$, together with the estimate
\begin{align*}
\|(P(is)^{-1})^*\|_{\calL(H;H_1)} \leq C \left( 1 + (|s|+1)\|P(is)^{-1}\|_{\calL(H)} \right) .
\end{align*}
By transposition, we have $\transp(P(is)^{-1})^* \in \calL(H_{-1};H)$, together with the estimate
\begin{align}
\label{eq: transp Pis}
\|\transp(P(is)^{-1})^*\|_{\calL(H_{-1};H)}
\leq \|(P(is)^{-1})^*\|_{\calL(H;H_1)}
\leq C \left( 1 + (|s|+1)\|P(is)^{-1}\|_{\calL(H)} \right) .
\end{align}
Moreover, $\transp(P(is)^{-1})^*$ is defined, for every $u \in H, v \in H_{-1}$, by
\begin{align*}
\left(\transp(P(is)^{-1})^*v, u\right)_H
=\left< v, (P(is)^{-1})^* u\right>_{H_{-1}, H_1}
= \left( (A +\id)^{-\frac12}v, (A +\id)^{\frac12}(P(is)^{-1})^* u\right)_{H} .  
\end{align*}
In particular, taking $v \in H$ gives
\begin{align*}
\left(\transp(P(is)^{-1})^*v, u\right)_H
= \left( P(is)^{-1}v,  u\right)_{H} ,
\end{align*}
which implies that the restriction of the operator $\transp(P(is)^{-1})^*$ to $H$ coincides with $P(is)^{-1}$. For simplicity, we will denote $P(is)^{-1}$ for  $\transp(P(is)^{-1})^*$.

Equation \eqref{eq: transp Pis} can thus be rewritten
\begin{align}
\label{eq: P_1 in H-1 H}
\| P(is)^{-1}\|_{\calL(H_{-1};H)}
\leq C \left( 1 + (|s|+1)\|P(is)^{-1}\|_{\calL(H)} \right) .
\end{align}
Then, we have $P(is)^{-1}(isBB^* - s^2\id) -\id  =P(is)^{-1} A$, so that 
\begin{align}
\label{eq: estimate second}
 \|P(is)^{-1}(isBB^* - s^2\id) -\id \|_{\calL(H_1;H)}
 &= \|P(is)^{-1}A\|_{\calL(H_1;H)} \leq \|P(is)^{-1}\|_{\calL(H_{-1};H)} \|A\|_{\calL(H_1;H_{-1})}
 \nonumber \\
 & \leq \left( 1 + (|s|+1)\|P(is)^{-1}\|_{\calL(H)} \right)
\end{align}

Third, for $|s|\geq 1$ we write 
\begin{align}
\label{eq: def third term}
P(is)^{-1}(BB^* + is\id) = \frac{i}{s}\left(P(is)^{-1} A -\id \right), 
\end{align}
and it remains to estimate the term $\|P(is)^{-1} A\|_{\calL(H_1)}$ in \eqref{eq: expression resolvent P(is)}. For $f \in H_1$, we set $u = P(is)^{-1}Af$. We have $u \in H_1$, together with 
$$
(A - s^2\id + is BB^* )u = A f.
$$
Taking the real part of the inner product of this identity with $u$, we find
$$
\|A^\frac12 u\|_H^2 - s^2 \|u\|_H^2 =\Re (Af, u)_H 
\leq \|Af\|_{H_{-1}}\|u\|_{H_1}
\leq C \|f\|_{H_1}\|u\|_{H_1},
$$
as $A \in \calL(H_1, H_{-1})$. Hence 
$$
\|u\|_{H_1}^2 \leq C(1+s^2)\|u\|_H^2 + C \|f\|_{H_1}^2
$$
Using \eqref{eq: P_1 in H-1 H}, this gives
\begin{align*}
\|u\|_{H_1}^2 & \leq C(1+s^2)\|P(is)^{-1}A\|_{\calL(H_1 ; H)}^2 \|f\|_{H_1}^2 + C \|f\|_{H_1}^2 \\
& \leq C(1+s^2)\|P(is)^{-1}\|_{\calL(H_{-1} ; H)}^2 \|f\|_{H_1}^2 + C \|f\|_{H_1}^2  \\
& \leq C(1+s^2) \left( 1 + (|s|+1)\|P(is)^{-1}\|_{\calL(H)} \right)^2 \|f\|_{H_1}^2 ,
\end{align*}
and finally $\|P(is)^{-1} A\|_{\calL(H_1)} \leq C(1+|s|) \left( 1 + (|s|+1)\|P(is)^{-1}\|_{\calL(H)} \right)$. Coming back to \eqref{eq: def third term}, we have, for $|s|\geq 1$,
\begin{align}
\label{eq: estimate third}
\|P(is)^{-1}(BB^* + is\id)\|_{\calL(H_1)} 
\leq C \left( 1 + |s|\|P(is)^{-1}\|_{\calL(H)} \right) .
\end{align}

Finally, combining \eqref{eq: estimate first}, \eqref{eq: estimate second} and \eqref{eq: estimate third}, together with \eqref{eq: M(s) equiv resA}-\eqref{eq: expression resolvent P(is)}, we obtain for $|s|\geq 1$,
$$
 \|(is\id - \A)^{-1}\|_{\calL(\H)} 
 \leq C \left( 1 + |s|\|P(is)^{-1}\|_{\calL(H)} \right) .
$$
This concludes the proof of Lemma~\ref{lemma: conditions stable resolvent}.

\end{proof}

\part{Proof of Theorem~\ref{th: stabilization torus}: smooth damping coefficients on the torus}
To prove Theorem~\ref{th: stabilization torus}, we shall instead prove Estimate~\eqref{eq: CNS poly stable resolvent 0} with $\alpha = \frac{1}{1+\delta}$ (which, according to Proposition~\ref{prop: CNS poly sable}, is equivalent to the statement of Theorem~\ref{th: stabilization torus}).
Let us first recast \eqref{eq: CNS poly stable resolvent 0} with $\alpha = \frac{1}{1+\delta}$ in the semiclassical setting~: taking $h=s^{-1}$, we are left to prove that there exist $C>1$ and $h_0 >0$ such that for all $h\leq h_0$, for all $u \in H^2(\T^2)$, we have
\begin{align} 
\label{eq: upper bound}
 \|u\|_{L^2(\T^2)} \leq C h^{-\delta}  \|P(i/h)u\|_{L^2(\T^2)}
 \end{align}

We prove this inequality by contradiction, using the notion of semiclassical measures. The idea of developing such a strategy for proving energy estimates, together with the associate technology, originates from Lebeau~\cite{Leb:96}.

We assume that~\eqref{eq: upper bound} is not satisfied, and will obtain a contradiction at the end of Section~\ref{section: end of the proof}. Hence, for all $n \in \N$, there exists $0< h_n\leq \frac{1}{n}$ and $u_n \in H^2(\T^2)$ such that 
\begin{align*} 
 \|u_n\|_{L^2(\T^2)} > \frac{n}{h_n^{\delta}} \|P(i/h_n)u_n\|_{L^2(\T^2)} .
 \end{align*}
Setting $v_n = u_n/ \|u_n\|_{L^2(\T^2)}$, and 
$$
P_b^{h_n} = - h_n^2 \Delta - 1 + i h_n b(x) = h_n^2 P(i/h_n),
$$
we then have, as $n \to \infty$,
\begin{align*} \begin{cases}
 & h_n \to 0^+ ,\\
 & \|v_n\|_{L^2(\T^2)} = 1 ,\\
 & h_n^{-2-\delta}\|P_b^{h_n}v_n\|_{L^2(\T^2)} \to 0 .
 \end{cases}
\end{align*}
Our goal is now to associate to the sequence $(u_n , h_n)$ a semiclassical measure on the cotangent bundle $\mu$ on $T^*\T^2 = \T^2 \times (\R^2)^*$ (where $(\R^2)^*$ is the dual space of $\R^2$). To obtain a contradiction, we shall prove both that $\mu(T^*\T^2) = 1$, and that $\mu =0$ on $T^*\T^2$. 

From now on, we drop the subscript $n$ of the sequences above, and write $h$ in place of $h_n$ and $v_h$ in place of $v_n$. We study sequences $(h, v_h)$ such that $h \to 0^+$ and
\begin{align}
\label{eq: conditions sequence}
\begin{cases}
 & \|v_h\|_{L^2(\T^2)} = 1 \\
 &\|P_b^{h}v_h\|_{L^2(\T^2)} = o(h^{2+ \delta}), \quad \text{as } h \to 0^+ .
 \end{cases}
\end{align}
In particular, this last equation also yields the key information
$$
(b v_h, v_h)_{L^2(\T^2)} = 
h^{-1} \Im(P_b^{h}v_h ,v_h)_{L^2(\T^2)} = 
 o(h^{1+\delta}), \quad \text{as } h \to 0^+.
$$

In the following, it will be convenient to identify $(\R^2)^*$ and $\R^2$ through the usual inner product. In particular, the cotangent bundle $T^*\T^2 = \T^2 \times (\R^2)^*$ will be identified with $\T^2 \times \R^2$.

\section{Semiclassical measures}
\label{sub: semiclassical measures}

We denote by $\ovl{T^*\T^2}$ the compactification of $T^*\T^2$ obtained by adding a point at infinity to each fiber (\ie, the set $\T^2 \times (\R^2 \cup \{\infty\})$). A \nhd of $(x, \infty) \in \ovl{T^*\T^2}$ is a set $U\times \big(\{\infty\} \cup \R^2 \setminus K \big)$, where $U$ is a \nhd of $x$ in $\T^2$ and $K$ a compact set in $\R^2$. Endowed with this topology, the set $\ovl{T^*\T^2}$ is compact.

We denote by $S^0(T^*\T^2)$, $S^0$ for short, the space of 
functions $a(x,\xi)$ that satisfy the following properties: 
\begin{enumerate}
\item $a \in \Cinf ( T^*\T^2 )$.
\item There exists a compact set $K \subset \R^2$ and a constant $k_0 \in \C$ such that $a(x,\xi) = k_0$ for all $\xi \in \R^2 \setminus K$.
\end{enumerate}
Note that we have in particular $\Cinfc(T^*\T^2) \subset S^0(T^*\T^2) $.

To a symbol $a \in S^0(T^*\T^2)$, we associate its semiclassical Weyl quantization $\Op_h(a)$ by Formula~\eqref{eq: weyl quantif}, which, according to the Calder\'on-Vaillancourt Theorem (see Appendix~\ref{section: pseudo}) defines a uniformly bounded operator on $L^2(\T^2)$.

From the sequence $(v_h , h)$ (see for instance~\cite{GL:93}), we can define (using again the Calder\'on-Vaillancourt Theorem) the associated Wigner distribution $V^h \in (S^0)'$ by
\begin{align}
\label{eq: def Wigner}
\left< V^h , a \right>_{(S^0)',S^0} = \left( \Op_h(a) v_h , v_h \right)_{L^2(\T^2)}, 
\quad \text{for all } a \in S^0(T^*\T^2) .
\end{align} 
Decomposing $v_h$ and $a$ in Fourier series,
$$
\hat{v}_h (k)= \frac{1}{2 \pi}\int_{\T^2}e^{-i k \cdot x} v_h(x) dx , 
\quad \hat{a} (h, k ,\xi) = \frac{1}{2 \pi}\int_{\T^2}e^{-i k \cdot x} a(h,x,\xi) dx , 
$$
the expression~\eqref{eq: def Wigner} can be more explicitly rewritten as 
$$
\left< V^h , a \right>_{(S^0)',S^0} =
\frac{1}{2 \pi} \sum_{k,j \in \Z^2}\hat{a} \left(h, j-k ,\frac{h}{2}(k+j)\right) \hat{v}_h (k) \ovl{\hat{v}_h (j)} .
$$
 \begin{proposition}
\label{prop: existence mu} The family $(V^h)$ is bounded in $(S^0)'$.
Hence,
there exists a subsequence of the sequence $(h, v_h)$ and an element $\mu \in(S^0)'$, such that $V^h \rightharpoonup \mu$ weakly in $(S^0)'$, \ie
\begin{equation}
\label{eq: convergence semiclassical measure}
\left( \Op_h(a) v_h , v_h \right)_{L^2(\T^2)} 
\to \left< \mu , a \right>_{(S^0)',S^0}
\quad \text{for all } a \in S^0(T^*\T^2).
\end{equation}
In addition, $\left< \mu , a \right>_{(S^0)',S^0}$ is nonnegative if $a$ is; in other words, $\mu$ may be identified with a nonnegative Radon measure on $\ovl{T^*\T^2}$.
\end{proposition}

{\bf Notation:} in what follows we shall denote by  $\M^+(\ovl{T^*\T^2})$ the set of nonnegative Radon measures on $\ovl{T^*\T^2}$.

\begin{proof}
The proof is an adaptation from the original proof of G\'erard~\cite{Gerard:91} (see also~\cite{GL:93} in the semiclassical setting).

The fact that the Wigner distributions $V^h$ are uniformly bounded in $(S^0)'$ follows from the Calder\'on-Vaillancourt theorem (see Appendix~\ref{section: pseudo}), and from the boundedness of $(v_h)$ in $L^2(\T^2)$. 

The sharp G{\aa}rding inequality 
gives the existence of $C>0$ such that, for all $a\geq 0$ and $h>0$,
$$
\left( \Op_h(a) v_h , v_h \right)_{L^2(\T^2)} \geq - C h \|v_h\|_{L^2(\T^2)}^2 ,
$$
so that the distribution $\mu$ is nonnegative (and is hence a measure).

\end{proof}

\section{Zero-th and first order informations on $\mu$}
To simplify the notation, we set 
$$
P_b^h = P_0^h + ih b(x), \quad \text{with} \quad P_0^h = -h^2 \Delta - 1 = \Op_h (|\xi|^2-1).
$$

The geodesic flow on the torus $\phi_\tau : T^*\T^2 \to  T^*\T^2$ for $\tau \in \R$ is the flow generated by the Hamiltonian vector field associated to the symbol $\frac12(|\xi|^2-1)$, \ie by the vector field 
$ \xi \cdot \d_x  $ on $T^*\T^2$.
Explicitely, we have
$$
\phi_\tau(x,\xi) = (x + \tau \xi,\xi) , \quad \tau \in \R, \quad (x,\xi) \in  T^*\T^2 .
$$
Note that $\phi_\tau$ preserves the $\xi$-component, and, in particular every energy layer $\{|\xi|^2 = C>0\} \subset T^*\T^2$.

Now, we describe the first properties of the measure $\mu$ implied by~\eqref{eq: conditions sequence}.

We recall that for $\nu \in \Dist'(T^*\T^2)$, $(\phi_\tau)_*\nu \in \Dist'(T^*\T^2)$ is defined by $\left<(\phi_\tau)_* \nu , a \right> = \left<\nu , a \circ \phi_\tau \right>$ for all $a \in \Cinfc(T^*\T^2)$. In particular, $(\phi_\tau)_* \nu$ is a measure if $\nu$ is.
We shall say that $\nu$ is an {\em invariant measure} if it is invariant by the geodesic flow, \ie  $(\phi_\tau)_* \nu=\nu$ for all $\tau\in\R$.

\begin{proposition}
\label{proposition: zero first order info}
Let $\mu$ be as in Proposition \ref{prop: existence mu}.
We have 
\begin{enumerate}
\item \label{item: supp car} $\supp(\mu) \subset \{|\xi|^2 = 1\}$ (hence is compact in $T^*\T^2$),
\item \label{item: mu tot =1} $\mu(T^* \T^2) = 1$,
\item \label{item: propagation} $\mu$ is invariant by the geodesic flow, \ie $(\phi_\tau)_*\mu = \mu$,
\item \label{item: mu = 0 on b} $\left< \mu , b \right>_{\M_c(T^*\T^2) , \Con^0(T^*\T^2)} = 0$, where $\M_c(T^*\T^2)$ denotes the space of compactly supported measures on $T^*\T^2$.
\end{enumerate}
In other words, $\mu$ is an invariant probability measure on $T^* \T^2$ vanishing on $\{b>0\}$.
\end{proposition}
These are standard arguments, that we reproduce here for the reader's comfort.
In particular, we recover all informations required to prove the Bardos-Lebeau-Rauch-Taylor uniform stabilization theorem under GCC. But we do not use here the second order informations of \eqref{eq: conditions sequence}; this will be the key point to prove Theorem~\ref{th: stabilization torus}.

\begin{proof} 
First, we take $\chi \in \Cinf({T^*\T^2})$ depending only on the $\xi$ variable, such that $\chi\geq0$, $\chi(\xi) = 0$ for $|\xi| \leq 2$, and $\chi(\xi) = 1$ for $|\xi| \geq 3$. Hence, $\frac{\chi(\xi)}{|\xi|^2-1} \in \Cinf(T^*\T^2)$ and we have the exact composition formula
$$
\Op_h(\chi) = \Op_h\left(\frac{\chi(\xi)}{|\xi|^2-1} \right) P_0^h ,
$$
since both operators are Fourier multipliers. Moreover, $\Op_h\left(\frac{\chi(\xi)}{|\xi|^2-1} \right)$ is a bounded operator on $L^2(\T^2)$. As a consequence, we have
\begin{align*}
\left< V^h ,  \chi \right>_{(S^0)',S^0} 
\to \left< \mu , \chi \right>_{\M(\ovl{T^*\T^2}),\Con^0(\ovl{T^*\T^2})} , 
\end{align*}
together with 
\begin{align*}
\left< V^h ,  \chi \right>_{(S^0)',S^0} 
&= \left( \Op_h\left(\frac{\chi(\xi)}{|\xi|^2-1} \right) P_0^h v_h , v_h \right)_{L^2(\T^2)} \\
& = \left( \Op_h\left(\frac{\chi(\xi)}{|\xi|^2-1} \right) P_b^{h} v_h , v_h \right)_{L^2(\T^2)} 
- i h \left(\Op_h\left(\frac{\chi(\xi)}{|\xi|^2-1} \right) b v_h , v_h \right)_{L^2(\T^2)} .
\end{align*}
Since $\|P_b^{h} v_h\|_{L^2(\T^2)} = o(1)$ and $\| v_h \|_{L^2(\T^2)} =1$, both terms in this expression vanish in the limit $h \to 0^+$. This implies that $\left< \mu , \chi \right>_{\M(\ovl{T^*\T^2}),\Con^0(\ovl{T^*\T^2})} = 0$. Since this holds for all $\chi$ as above, we have $\supp(\mu) \subset  \{|\xi|^2 = 1\}$, which proves Item~\ref{item: supp car}.


In particular, this implies that $\mu\left(\ovl{T^*\T^2}\setminus T^*\T^2\right)=0$.
Now, Item~\ref{item: mu tot =1} is a direct consequence of $1 = \| v_h \|_{L^2(\T^2)}^2 \to \left< \mu , 1\right>_{\M(\ovl{T^*\T^2}), \Con^0(\ovl{T^*\T^2})}$ and Item~\ref{item: supp car}. Item~\ref{item: mu = 0 on b} is a direct consequence of $\left(b v_h , v_h \right)_{L^2(\T^2)} =o(1)$.

\bigskip
Finally, for $a \in \Cinfc(T^*\T^2)$, we recall that
$$
\left[  P_0^h , \Op_h(a) \right] 
= \frac{h}{i} \Op_h(\{|\xi|^2 - 1 , a\})
= \frac{2 h}{i} \Op_h( \xi \cdot \d_x a) ,
$$
is a consequence of the Weyl quantization (any other quantization would have left an error term of order $\O(h^2)$). Hence, \eqref{eq: def Wigner} yields 
\begin{align}
\label{eq: propagation}
\left< V^h ,  \xi \cdot \d_x a \right>_{\Dist'(T^*\T^2),\Cinfc(T^*\T^2)} 
\to \left< \mu , \xi \cdot \d_x a \right>_{\M(T^*\T^2),\Conc(T^*\T^2)} , 
\end{align}
together with
\begin{align}
\label{eq: propagation simple}
\left< V^h ,  \xi \cdot \d_x a  \right>_{\Dist'(T^*\T^2),\Cinfc(T^*\T^2)} 
& = \frac{i}{2h}\left( \left[  P_0^h , \Op_h(a) \right] v_h , v_h \right)_{L^2(\T^2)} \nonumber\\
& = \frac{i}{2h}\left( \Op_h(a)  v_h ,  P_0^h v_h \right)_{L^2(\T^2)} 
- \frac{i}{2h}\left( \Op_h(a)  P_0^h v_h , v_h \right)_{L^2(\T^2)} \nonumber\\
& = \frac{i}{2h}\left( \Op_h(a)  v_h , P_b^{h}v_h \right)_{L^2(\T^2)} 
- \frac{i}{2h}\left( \Op_h(a) P_b^{h} v_h , v_h \right)_{L^2(\T^2)}\nonumber\\
& \quad \quad - \frac{1}{2} \left( \Op_h(a)  v_h , b v_h \right)_{L^2(\T^2)}
- \frac{1}{2} \left( \Op_h(a) b v_h , v_h \right)_{L^2(\T^2)}.
\end{align}
In this expression, we have $\frac{1}{h}\left( \Op_h(a)  v_h , P_b^{h}v_h \right)_{L^2(\T^2)} \to 0$ and $\frac{1}{h}\left( \Op_h(a) P_b^{h} v_h , v_h \right)_{L^2(\T^2)} \to 0$ since $\|P_b^{h}v_h\|_{L^2(\T^2)} = o(h)$. Moreover, the last two terms can be estimated by 
\begin{equation}
\label{eq: b to zero}
|\left( \Op_h(a) b v_h , v_h \right)_{L^2(\T^2)}|  \leq \|\sqrt{b} v_h \|_{L^2(\T^2)} \| \sqrt{b} \Op_h(a) v_h \|_{L^2(\T^2)} = o(1) ,
\end{equation}
since $\left(b v_h , v_h \right)_{L^2(\T^2)} =o(1)$. This yields $\left< V^h ,  \xi \cdot \d_x a  \right>_{\Dist'(T^*\T^2),\Cinfc(T^*\T^2)} \to 0$, so that, using \eqref{eq: propagation}, $\left< \mu , \xi \cdot \d_x a \right>_{\M(T^*\T^2),\Conc(T^*\T^2)} = 0$ for all $a \in \Cinfc(T^*\T^2)$. Replacing $a$ by $a \circ \phi_\tau$ and integrating with respect to the parameter $\tau$ gives $(\phi_\tau)_*\mu = \mu$, which concludes the proof of Item~\ref{item: propagation}.

\end{proof}

\section{Geometry on the torus and decomposition of invariant measures}

\subsection{Resonant and non-resonant vectors on the torus}
\label{subsub: geometry on torus}
In this section, we collect several facts concerning the geometry of $T^*\T^2$ and its resonant subspaces. Most of the setting and the notation comes from~\cite[Section 2]{AM:11}. 

We shall say that a submodule $\Lambda \subset \Z^2$ is primitive if $\left< \Lambda \right> \cap \Z^2 = \Lambda$, where $\left< \Lambda \right>$ denotes the linear subspace of $\R^2$ spanned by $\Lambda$. The family of all primitive submodules will be denoted by $\P$. 

Let us denote by $\Omega_j \subset \R^2$, for $j =0,1,2$, the set of resonant vectors of order exactly $j$, \ie,
$$
\Omega_j : = \{ \xi \in \R^2 \text{ such that }\rk(\Lambda_\xi) = 2-j \}, 
\quad \text{with} \quad
\Lambda_\xi:= \left\{ k \in \Z^2 \ \text{such that} \ \xi \cdot k = 0 \right\} =\xi^\perp\cap\Z^2 .
$$
Note that the sets $\Omega_j$ form a partition of $\R^2$, and that we have
\begin{itemize}
\item $\Omega_0 = \{0\}$;
\item $\xi \in \Omega_1$ if and only if the geodesic issued from any $x \in \T^2$ in the direction $\xi$ is periodic;
\item $\xi \in \Omega_2$ if and only if the geodesic issued from any $x \in \T^2$ in the direction $\xi$ is dense in $\T^2$.
\end{itemize}
For each $\Lambda \in \P$ such that $\rk(\Lambda)= 1$, we define 
\begin{equation*}
\begin{array}{c}
\Lambda^\perp:= \left\{ \xi \in \R^2 \ \text{such that} \ \xi \cdot k = 0 \ \text{for all} \ k \in \Lambda\right\}, \\
\T_\Lambda : = \left< \Lambda \right> / 2 \pi \Lambda , \\
\T_{\Lambda^\perp} : = \Lambda^\perp / (2 \pi \Z^2 \cap \Lambda^\perp) .
\end{array}
\end{equation*}
Note that $\T_\Lambda$ and $\T_{\Lambda^\perp}$ are two submanifolds of $\T^2$ diffeomorphic to one-dimensional tori. 
Their cotangent bundles admit the global trivialisations $T^*\T_\Lambda = \T_\Lambda \times \left< \Lambda \right>$ and $T^*\T_{\Lambda^\perp} = \T_{\Lambda^\perp} \times \Lambda^\perp$. 

For a function $f$ on $\T^2$ with Fourier coefficients $(\hat{f}(k))_{k \in \Z^2}$, and $\Lambda\in \P$, we shall say that $f$ has only Fourier modes in $\Lambda$ if $\hat{f}(k) = 0$ for $k \notin \Lambda$. This means that $f$ is constant in the direction $\Lambda^\perp$, or, equivalently, that $\sigma \cdot \d_x f = 0$ for all $\sigma \in \Lambda^\perp$. We denote by $L^p_\Lambda(\T^2)$ the subspace of $L^p(\T^2)$ consisting of functions having only Fourier modes in $\Lambda$. For a function $f \in L^2(\T^2)$ (\resp a symbol $a \in S^0(T^*\T^2)$), we denote by $\left< f\right>_\Lambda$ its orthogonal projection on $L^2_\Lambda(\T^2)$, \ie the average of $f$ along $\Lambda^\perp$:
$$
\left< f\right>_\Lambda (x) : = \sum_{k \in \Lambda} \frac{e^{ik \cdot x}}{2 \pi} \hat{f}(k)  
\qquad \left( \text{\resp} \ 
\left< a\right>_\Lambda (x,\xi) : = \sum_{k \in \Lambda} \frac{e^{ik \cdot x}}{2 \pi} \hat{a} (k,\xi)
\right).
$$
If $\rk(\Lambda) = 1$ and $v $ is a vector in $\Lambda^\perp\setminus\{0\}$, we also have
\begin{equation}
\label{eq: moyennisation lambda}
\left< f\right>_\Lambda (x) = \lim_{T\to \infty} 
\frac{1}{T} \int_0^T f(x + t v) dt .
\end{equation}
In particular, note that $\left< f\right>_\Lambda$ (\resp $\left< a\right>_\Lambda$) is nonnegative if $f$ (\resp $a$) is, and that $\left< f\right>_\Lambda \in \Cinf(\T^2)$ (\resp $\left< a\right>_\Lambda \in S^0(T^*\T^2)$) if $f \in \Cinf(\T^2)$ (\resp $a \in S^0(T^*\T^2)$).

Finally, given $f \in L^\infty_\Lambda(\T^2)$, we denote by $m_f$ the bounded operator on $L^2_\Lambda(\T^2)$, consisting in the multiplication by $f$.

\subsection{Decomposition of invariant measures}
We denote by $\M^+(T^*\T^2)$ the set of finite, nonnegative measures on $T^*\T^2$.
With the definitions above, we have the following decomposition Lemmata, proved in~\cite{Macia:10} or~\cite[Section~2]{AM:11}.  
These properties are given for general measures $\mu \in \M^+(T^*\T^2)$. Of course, they apply in particular to the measure $\mu$ defined by Proposition~\ref{prop: existence mu}.
\begin{lemma}
\label{lemma: decompose mu}
Let $\mu \in \M^+(T^*\T^2)$. Then $\mu$ decomposes as a sum of nonnegative measures
\begin{equation}
\label{eq: decomposition mu}
\mu =  \mu|_{\T^2 \times \{0\}} + \mu|_{\T^2 \times \Omega_2} + \sum_{\Lambda \in \P, \rk(\Lambda)=1} \mu|_{\T^2 \times (\Lambda^\perp\setminus \{0\})}
\end{equation}
\end{lemma}

Given $\mu \in \M^+(T^*\T^2)$, we define its Fourier coefficients by the complex measures on $\R^2$:
$$
 \hat{\mu}(k , \cdot) 
 : = \int_{\T^2} \frac{e^{-ik \cdot x}}{2 \pi} \mu(dx, \cdot), \quad k \in \Z .
$$
One has, in the sense of distributions, the following Fourier inversion formula:
$$
\mu (x, \xi)= \sum_{k \in \Z^2} \frac{e^{ik \cdot x}}{2 \pi} \hat{\mu}(k , \xi) .
$$
\begin{lemma}
\label{lemma: mu only Lambda FM}
Let $\mu \in \M^+(T^*\T^2)$ and $\Lambda \in \P$. Then, the distribution
$$
\left< \mu \right>_\Lambda (x, \xi):= \sum_{k \in \Lambda} \frac{e^{ik \cdot x}}{2 \pi} \hat{\mu}(k , \xi) ,
$$
is in $\M^+(T^*\T^2)$ and satisfies, for all $a \in \Cinfc(T^*\T^2)$,
$$
\left< \left< \mu \right>_\Lambda , a \right>_{\M(T^*\T^2) , \Conc(T^*\T^2)}
= \left< \mu , \left< a \right>_\Lambda \right>_{\M(T^*\T^2) , \Conc(T^*\T^2)} .
$$
\end{lemma}

\begin{lemma}
\label{lemma: mu direction independent}
Let $\mu \in \M^+(T^*\T^2)$ be an {\em invariant} measure. Then, for all $\Lambda \in \P$, $\mu|_{\T^2 \times \left(\Lambda^\perp\setminus \{0\}\right)}$ is also a nonnegative {\em invariant} measure and
$$
\mu|_{\T^2 \times \left(\Lambda^\perp\setminus \{0\}\right)} = \left< \mu \right>_\Lambda|_{\T^2 \times  \left(\Lambda^\perp\setminus \{0\}\right)} .
$$
\end{lemma}

Let us now come back to the measure $\mu$ given by Proposition~\ref{prop: existence mu}, which satisfies all properties listed in Proposition~\ref{proposition: zero first order info}. In particular, this measure vanishes on the non-empty open subset of $\T^2$ given by $\{b>0\}$ (see Item~\ref{item: mu = 0 on b} in Proposition~\ref{proposition: zero first order info}). As a consequence of Proposition~\ref{proposition: zero first order info}, and of the three lemmata above, this yields the following lemma.

\begin{lemma}
\label{lemma: mu = 0 some directions}
We have
$
\mu = \sum_{\Lambda \in \P, \rk(\Lambda)=1} \mu|_{\T^2 \times (\Lambda^\perp\setminus \{0\})} .
$
\end{lemma}
As a consequence of Proposition~\ref{proposition: zero first order info}, we have indeed that the measure $\mu$ is supported in $\{|\xi| = 1\}$, which implies $\mu|_{\T^2 \times \{0\}} = 0$. In addition,  Lemma \ref{lemma: mu direction independent} applied with $\Lambda=\{0\}$ implies that
$ \mu|_{\T^2 \times \Omega_2} $ is constant in $x$ -- and thus vanishes everywhere since it vanishes on $\{b>0\}$.

\begin{remark}
\label{rem: mu lambda 0}
Since the measure $\mu$ is supported in $\{|\xi| = 1\}$ (Proposition~\ref{proposition: zero first order info}, Item~\ref{item: supp car}), we have
$$
\mu|_{\T^2 \times \Lambda^\perp} = \mu|_{\T^2 \times \left(\Lambda^\perp\setminus \{0\}\right)} 
$$
(which simplifies the notation).
\end{remark}

As a consequence of these lemmata and the last remark, the study of the measure $\mu$ is now reduced to that of all nonnegative invariant measures $\mu|_{\T^2 \times \Lambda^\perp}$ with $\rk(\Lambda)=1$. 

The aim of the next sections is to prove that the measure $\mu|_{\T^2 \times \Lambda^\perp}$ vanishes identically, for each periodic direction $\Lambda^\perp$.

\subsection{Geometry of the subtori $\T_\Lambda$ and $\T_{\Lambda^\perp}$\label{s:subtori}}

To study the measure  $\mu|_{\T^2 \times \left(\Lambda^\perp\setminus \{0\}\right)} $, we need to describe briefly the geometry of the subtori $\T_\Lambda$ and $\T_{\Lambda^\perp}$ of $\T^2$, and introduce adapted coordinates.

We define $\chi_\Lambda$ the linear isomorphism
$$
\chi_\Lambda : \Lambda^\perp \times \left< \Lambda \right> \to \R^2 : \quad
(s,y) \mapsto s+y , 
$$
and denote by $\tilde{\chi}_\Lambda : T^* \Lambda^\perp \times T^* \left< \Lambda \right>\to T^*\R^2$ its extension to the cotangent bundle. This application can be defined as follows: for $(s, \sigma) \in T^* \Lambda^\perp = \Lambda^\perp \times (\Lambda^\perp)^*$ and $(y, \eta) \in T^* \left< \Lambda \right> = \left< \Lambda \right> \times \left< \Lambda \right>^*$, we can extend $\sigma$ to a covector of $\R^2$ vanishing on $\left< \Lambda \right>$ and $\eta$ to a covector of $\R^2$ vanishing on $\Lambda^\perp$. Remember that we identify $(\R^2)^*$ with $\R^2$ through the usual inner product; thus we can also see $\sigma$ as an element of $ \Lambda^\perp$ and $\eta $ as an element of $\left< \Lambda \right>$.
Then, we have
$$
\tilde{\chi}_\Lambda(s, \sigma, y, \eta) = (s + y, \sigma +\eta) \in T^* \R^2 = \R^2 \times (\R^2)^*.
$$

Conversely, any $\xi\in  (\R^2)^*$ can be decomposed into $\xi=\sigma +\eta$ where $\sigma\in \Lambda^\perp$ and $\eta\in\left<\Lambda \right>$. We denote by $P_\Lambda$ the orthogonal projection of $\R^2$ onto $\left< \Lambda \right>$, \ie $P_\Lambda \xi = \eta$.

Next, the map $\chi_\Lambda$ goes to the quotient, giving a smooth Riemannian covering of $\T^2$ by
$$
\pi_\Lambda : \T_{\Lambda^\perp} \times \T_\Lambda \to \T^2 : \quad
(s,y) \mapsto s+y .
$$
We shall denote by $\tilde{\pi}_\Lambda$ its extension to cotangent bundles:
$$
\tilde{\pi}_\Lambda : T^*\T_{\Lambda^\perp} \times T^*\T_\Lambda \to T^*\T^2 .
$$
As the map $\pi_\Lambda$ is not an injection (because the torus $\T_{\Lambda^\perp} \times \T_\Lambda$ contains several copies of $\T^2$), 
we introduce its degree $p_\Lambda$, which is also equal to $\frac{\Vol(\T_{\Lambda^\perp} \times \T_\Lambda)}{\Vol(\T^2)}$.

Then, the application
$$
T_\Lambda u : = \frac{1}{\sqrt{p_\Lambda}} u \circ \chi_\Lambda ,
$$
defines a linear isomorphism $L^2_{\loc} (\R^2) \to L^2_{\loc}(\Lambda^\perp \times \left< \Lambda \right>)$. Note that because of the factor $\frac{1}{\sqrt{p_\Lambda}}$, $T_\Lambda$ maps $L^2(\T^2)$ isometrically into a subspace of $L^2(\T_{\Lambda^\perp} \times \T_\Lambda)$. Moreover, $T_\Lambda$ maps $L^2_\Lambda(\T^2)$ into $L^2(\T_\Lambda) \subset L^2(\T_{\Lambda^\perp} \times \T_\Lambda)$, since the nonvanishing Fourier modes of $u \in L^2_\Lambda(\T^2)$ correspond only to frequencies $k \in \Lambda$. This reads
\begin{equation}
\label{eq: TLambda}
T_\Lambda u(s,y) =  \frac{1}{\sqrt{p_\Lambda}} u (y) \ \text{for} \ (s,y) \in \T_{\Lambda^\perp} \times \T_\Lambda .
\end{equation}
Since $\tilde{\chi}_\Lambda$ is linear, we have, for any $a \in \Cinf(T^*\R^2)$
\begin{equation}
\label{eq: TLambda Op}
T_\Lambda \Op_h(a) = \Op_h(a \circ \tilde{\chi}_\Lambda) T_\Lambda,  
\end{equation}
where on the left $\Op_h$ is the Weyl quantization on $\R^2$ \eqref{eq: weyl quantif}, and on the right $\Op_h$ is the Weyl quantization on $\Lambda^\perp \times \left< \Lambda \right>$. Next, we denote by $\Op_h^{\Lambda^\perp}$ and $\Op_h^\Lambda$ the Weyl quantization operators defined on smooth test functions on $T^*\Lambda^\perp \times T^*\left< \Lambda \right>$ and acting only on the variables in $T^*\Lambda^\perp$ and $T^*\left< \Lambda \right>$ respectively, leaving the other frozen. For any $a \in \Cinfc(T^*\Lambda^\perp \times T^*\left< \Lambda \right>)$, we have~:
\begin{equation}
\label{eq: Op Lambda}
\Op_h(a) = \Op_h^{\Lambda^\perp}\circ \Op_h^\Lambda (a) 
= \Op_h^\Lambda \circ \Op_h^{\Lambda^\perp} (a).
\end{equation}

Now, if the symbol $a \in \Cinfc(T^*\T^2)$ has only Fourier modes in $\Lambda$, we remark, in view of~\eqref{eq: TLambda}, that $a \circ \tilde{\pi}_\Lambda$ does not depend on $s \in \T_{\Lambda^\perp}$. Therefore, we sometimes write $a \circ \tilde{\pi}_\Lambda(\sigma ,y , \eta)$ for $a \circ \tilde{\pi}_\Lambda(s, \sigma ,y , \eta)$ and \eqref{eq: TLambda Op}-\eqref{eq: Op Lambda} give
\begin{equation}
\label{eq: Op Lambda bis}
T_\Lambda \Op_h(a) = \Op_h^\Lambda \circ \Op_h^{\Lambda^\perp} (a \circ \tilde{\pi}_\Lambda) T_\Lambda 
= \Op_h^\Lambda (a \circ \tilde{\pi}_\Lambda(h D_s ,\cdot, \cdot )) T_\Lambda .
\end{equation}
Note that for every $\sigma \in \Lambda^\perp$, the operator $\Op_h^\Lambda (a \circ \tilde{\pi}_\Lambda( \sigma ,\cdot, \cdot ))$ maps $L^2(\T_\Lambda)$ into itself. More precisely, it maps the subspace $T_\Lambda (L^2_\Lambda(\T^2))$ into itself.


\section{Change of quasimode and construction of an invariant cutoff function}
\label{sub: cutoff function}


In this section, we first construct from the quasimode $v_h$ another quasimode $w_h$, that will be easier to handle when studying the measure $\mu|_{\T^2 \times \Lambda^\perp}$. Indeed $w_h$ is basically a microlocalization of $v_h$ in the direction $\Lambda^\perp$ at a precise concentration rate. 

Moreover, we introduce a cutoff function $\chi_h^\Lambda(x) =\chi_h^\Lambda(y,s)$, well-adapted to the damping coefficient $b$ and to the invariance of the measure $\mu|_{\T^2 \times \Lambda^\perp}$ in the direction $\Lambda^\perp$ (this cutoff function plays the role of the function $\chi(b/h)$ used in~\cite{BH:07} in the case where $b$ is itself invariant in the direction $\Lambda^\perp$). Its construction is a key point in the proof of Theorem~\ref{th: stabilization torus}.

\medskip
Let $\chi \in \Cinfc(\R)$ be a nonnegative function such that $\chi = 1$ in a \nhd of the origin. 
We first define 
\begin{align*}
w_h : = \Op_h \left( \chi \left(\frac{|P_\Lambda \xi|}{h^\alpha} \right) \right) v_h ,
 \end{align*}
which implicitely depends on $\alpha \in (0,1)$. 
The following lemma implies that, for $\delta$ and $\alpha$ sufficiently small, $w_h$ is as well a $o(h^{2+\delta})$-quasimode for $P_b^h$.

\begin{lemma}
\label{lemma: w_h quasimode}
For any $\alpha>0$ such that 
\begin{align}
\label{eq: condition alpha}
\delta + \frac{\eps}{2} + \alpha \leq \frac12 , \qquad 3\alpha + 2 \delta < 1 ,
\end{align}
we have
$$
\| P_b^h w_h \|_{L^2(\T^2)} = o(h^{2+\delta}) .
$$
\end{lemma}
As a consequence of this lemma, the semiclassical measures associated to $w_h$ satisfy in particular the conclusions of Proposition~\ref{proposition: zero first order info}.
Moreover, the following proposition implies that the sequence $w_h$ contains all the information in the direction $\Lambda^\perp$.

\begin{proposition}
\label{prop: mu Lambda w_h}
For any $a \in \Cinfc(T^*\T^2)$ and any $\alpha \in (0, 3/4)$ satisfying the assumptions of Lemma~\ref{lemma: w_h quasimode}, we have
$$
\left< \mu|_{\T^2 \times \Lambda^\perp} , a \right>_{\M(T^*\T^2) , \Conc(T^*\T^2)} 
= \lim_{h \to 0} \left( \Op_h(a)  w_h , w_h \right)_{L^2(\T^2)} .
$$
\end{proposition}

 Next, we state the desired properties of the cutoff function $\chi_h^\Lambda$. The proof of its existence is a crucial point in the proof of Theorem~\ref{th: stabilization torus}.

\begin{proposition}
\label{prop: existence chi}
For $\delta = 8\eps$, and $\eps < \frac{1}{76}$, there exists $\alpha$ satisfying \eqref{eq: condition alpha}, such that for any constant $c_0 >0$, there exists a cutoff function $\chi_h^\Lambda \in \Cinf(\T^2)$ valued in $[0, 1]$, such that
\begin{enumerate}
\item \label{item: chi(y)} $\chi_h^\Lambda = \chi_h^\Lambda(y)$ does not depend on the variable $s$ (\ie $\chi_h^\Lambda $ is $\Lambda^\perp$-invariant),
\item \label{item: 1-chi} $\|(1 -\chi_h^\Lambda ) w_h\|_{L^2(\T^2)} = o(1)$, 
\item \label{item: b leq c_0 h} $b \leq c_0 h$ on $\supp(\chi_h^\Lambda )$, 
\item \label{item: d_y chi} $\|\d_{y}\chi_h^\Lambda w_h\|_{L^2(\T^2)} = o(1)$,
\item \label{item: d_yy chi} $\|\d_{y}^2 \chi_h^\Lambda w_h\|_{L^2(\T^2)} = o(1)$.
\end{enumerate}
\end{proposition}
Note that the function $\chi_h^\Lambda$ implicitely depends on the constant $c_0$, that will be taken arbitrarily small in Section~\ref{section: Propagation laws}.

In the particular case where the damping function $b$ is invariant in one direction, this  proposition is not needed. In this case, one can take as in~\cite{BH:07} $\chi_h^\Lambda = \chi(\frac{b}{c_0 h})$. In the $d$-dimensional torus, this cutoff functions works as well if $b$ is invariant in $d - 1$ directions, and an analogue of Theorem~\ref{th: stabilization torus} can be stated in this setting. Unfortunately, our construction of the function $\chi_h^\Lambda$ (see the proof of Proposition~\ref{prop: existence chi} in Section~\ref{section: existence chi}) strongly relies on the fact that all trapped directions are periodic, and fails in higher dimensions.


\medskip
We give here a proof of Lemma~\ref{lemma: w_h quasimode}. Because of their technicality, we postpone the proofs of Propositions~\ref{prop: mu Lambda w_h} and \ref{prop: existence chi} to Sections~\ref{section: proof nu alpha} and~\ref{section: existence chi} respectively.

\begin{proof}[Proof of Lemma~\ref{lemma: w_h quasimode}]
First, we develop
\begin{align}
\label{eq: P_b^h w_h}
P_b^h w_h = P_b^h \Op_h \left( \chi \left(\frac{|P_\Lambda \xi|}{h^\alpha} \right) \right) v_h
= \Op_h \left( \chi \left(\frac{|P_\Lambda \xi|}{h^\alpha} \right) \right) P_b^h v_h + ih \left[b , \Op_h \left( \chi \left(\frac{|P_\Lambda \xi|}{h^\alpha} \right) \right) \right] v_h ,
\end{align}
since $P_0^h$ and $\Op_h \left( \chi \left(\frac{|P_\Lambda \xi|}{h^\alpha} \right) \right)$ are both Fourier mutipliers. We know that
\begin{align*}
\left\| \Op_h \left( \chi \left(\frac{|P_\Lambda \xi|}{h^\alpha} \right) \right) P_b^h v_h \right\|_{L^2(\T^2)} \leq \|P_b^h v_h \|_{L^2(\T^2)} = o (h^{2 + \delta}) .
\end{align*}
It only remains to study the operator
\begin{align}
\label{eq: estimate bracket}
\left[b , \Op_h \left( \chi \left(\frac{|P_\Lambda \xi|}{h^\alpha} \right) \right) \right] 
=ih^{1-\alpha} \Op_h \left( \d_y b \ \chi' \left(\frac{|P_\Lambda \xi|}{h^\alpha} \right) \right) + \O_{\calL(L^2)}(h^{2 (1- \alpha)}) 
\end{align}
according to the symbolic calculus. Moreover, using Assumption~\eqref{eq: assumption b}, we have 
$$
\left| \d_y b \ \chi' \left(\frac{|P_\Lambda \xi|}{h^\alpha} \right) \right| \leq C b^{1-\eps} .
$$
The sharp G{\aa}rding inequality applied to the nonnegative symbol
$$
C^2 b^{2(1-\eps)} - \left| \d_y b \ \chi' \left(\frac{|P_\Lambda \xi|}{h^\alpha} \right) \right|^2 ,
$$
then yields 
$$
\left(\Op_h\left( C^2 b^{2(1-\eps)} - \left| \d_y b \ \chi' \left(\frac{|P_\Lambda \xi|}{h^\alpha} \right) \right|^2 \right) v_h , v_h \right)_{L^2(\T^2)} \geq - C h^{1-\alpha} , 
$$
and hence
\begin{align}
\label{eq: estimate garding bracket}
 \left\|\Op_h \left(  \d_y b \ \chi' \left(\frac{|P_\Lambda \xi|}{h^\alpha} \right) \right) v_h \right\|_{L^2(\T^2)}^2
 \leq  C^2 ( b^{2(1-\eps)}v_h , v_h )_{L^2(\T^2)} + O(h^{1-\alpha})  .
\end{align}
When using the inequality $\int f^{1-\eps} d \nu \leq \left( \int f d \nu \right)^{1-\eps}$ for nonnegative functions (with $d \nu = |v_h(x)|^2 dx$), we obtain
$$
(b^{2(1-\eps)} v_h, v_h)_{L^2(\T^2)} \leq (b^2 v_h, v_h)_{L^2(\T^2)}^{(1-\eps)} \leq C 
\|{b} v_h\|_{L^2(\T^2)}^{2(1-\eps)} = o(h^{1-\eps}) .
$$
Combining this estimate together with \eqref{eq: estimate bracket} and \eqref{eq: estimate garding bracket} gives
\begin{align*}
\left\| ih \left[b , \Op_h \left( \chi \left(\frac{|P_\Lambda \xi|}{h^\alpha} \right) \right) \right] v_h \right\|_{L^2(\T^2)} = o(h^{\frac52 -\alpha- \frac{\eps}{2}}) + O(h^{\frac{5-3\alpha}{2}}) .
\end{align*}
Coming back to the expression of $P_b^h w_h$ given in~\eqref{eq: P_b^h w_h}, this concludes the proof of Lemma~\ref{lemma: w_h quasimode}.
\end{proof}

\section{Second microlocalization on a resonant affine subspace}
\label{sub: second microlocalization}

We want to analyse precisely the structure of the restriction $\mu|_{\T^2 \times \left(\Lambda^\perp\setminus \{0\}\right)} $, 
using the full information contained in $o(h^{2+ \delta})$-quasimodes like $v_h$ and $w_h$.

From now on, we want to take advantage of the family $w_h$ of $o(h^{2+\delta})$-quasimodes constructed in Section~\ref{sub: cutoff function}, which are microlocalised in the direction $\Lambda^\perp$. Hence, we define the Wigner distribution $W^h \in \Dist'(T^*\T^2)$ associated to the functions $w_h$ and the scale $h$, by 
\begin{align*}
\left< W^h , a \right>_{(S^0)',S^0} = \left( \Op_h(a) w_h , w_h \right)_{L^2(\T^2)} 
\quad \text{for all } a \in S^0(T^*\T^2) .
\end{align*} 
According to Proposition~\ref{prop: mu Lambda w_h}, we recover in the limit $h\to 0$,
$$
\left< W^h , a \right>_{(S^0)',S^0} \to \left< \mu|_{\T^2 \times \Lambda^\perp} , a \right>_{\M(T^*\T^2) , \Conc(T^*\T^2)} ,
$$
for any $a \in \Cinfc(T^*\T^2)$ (and $\alpha$ satisfying~\eqref{eq: condition alpha}).  

To provide a precise study of $\mu|_{\T^2 \times \Lambda^\perp}$, we shall introduce as in~\cite{Macia:10,AM:11} two-microlocal semiclassical measures, describing at a finer level the concentration of the sequence $v_h$ on the resonant subspace
$$
\Lambda^\perp = \{\xi \in \R^2  \text{ such that } P_\Lambda \xi = 0\}.
$$
These objects have been introduced in the local Euclidean case by Nier~\cite{Nier:96} and Fermanian-Kammerer~\cite{Ferm:00heat, Ferm:00}. A specific concentration scale may also be chosen in the in the two-microlocal variable, giving rise to the two-scales semiclassical measures studied by Miller~\cite{Miller:96, Miller:97} and Fermanian-Kammerer and G\'erard~\cite{FKG:02}.

We first have to describe the adapted symbol class (inspired by \cite{Ferm:00} and used in \cite{AM:11}). According to Lemma~\ref{lemma: mu direction independent} (see also Remark~\ref{rem: mu lambda 0}), it suffices to test the measure $\mu|_{\T^2 \times \Lambda^\perp} $ with functions constant in the direction $\Lambda^\perp$ (or equivalently, having only $x$-Fourier modes in $\Lambda$, in the sense of the following definition).

\begin{definition}
Given $\Lambda \in \P$, we shall say that $a \in S^1_\Lambda$ if $a = a(x,\xi, \eta) \in \Cinf(T^*\T^2 \times \left< \Lambda \right>)$ and
\begin{enumerate}
\item there exists a compact set $K_a \subset T^*\T^2$ such that, for all $\eta \in \left< \Lambda \right>$, the function $(x, \xi) \mapsto a(x, \xi , \eta)$ is compactly supported in $K_a$;
\item $a$ is homogeneous of order zero at infinity in the variable $\eta \in \left< \Lambda \right>$; \ie, if we denote by $\bS_{\Lambda}:= \bS^{1} \cap \left< \Lambda \right>$ the unit sphere in $\left< \Lambda \right>$, there exists $R_0 >0$ (depending on $a$) and $a_{\hom} \in \Cinfc (T^*\T^2 \times \bS_{\Lambda})$ such that
$$
a(x,\xi, \eta) 
= a_{\hom}\left(x, \xi, \frac{\eta}{|\eta|}\right) ,
\quad \text{for} \quad |\eta| \geq R_0 \ \text{and} \ (x, \xi) \in T^*\T^2 ;$$
for $\eta\not=0$, we will also use the notation $a(x,\xi, \infty \eta):=  a_{\hom}\left(x, \xi, \frac{\eta}{|\eta|}\right).$
\item $a$ has only $x$-Fourier modes in $\Lambda$, \ie
$$
a (x, \xi, \eta)= \sum_{k \in \Lambda} \frac{e^{ik \cdot x}}{2 \pi} \hat{a}(k , \xi, \eta) .
$$
\end{enumerate}
Note that this last assumption is equivalent to saying that $\sigma \cdot \d_x a = 0$ for any $\sigma \in \Lambda^\perp$. We denote by ${S^1_\Lambda}'$ the topological dual space of $S^1_\Lambda$.
\end{definition}
 
Let $\chi \in \Cinfc(\R)$ be a nonnegative function such that $\chi = 1$ in a \nhd of the origin. Let $R>0$. The previous remark allows us to define, for $a \in S^1_\Lambda$ the two following elements of ${S^1_\Lambda}'$:
\begin{align}
\label{eq: def W^h_R^Lambda}
&\left< W^{h,\Lambda}_R , a \right>_{{S^1_\Lambda} ', S^1_\Lambda}
:= \left< W^h ,  \left( 1 - \chi \left(\frac{|P_\Lambda \xi|}{R h}\right)\right) a \left(x, \xi, \frac{P_\Lambda \xi}{h} \right) \right>_{\Dist'(T^*\T^2), \Cinfc(T^*\T^2)} , \\
\label{eq: def W^h_R,Lambda}
&\left< W^{h}_{R,\Lambda} , a \right>_{{S^1_\Lambda} ', S^1_\Lambda}
:= \left< W^h , \chi \left(\frac{|P_\Lambda \xi|}{R h}\right) a \left(x, \xi, \frac{P_\Lambda \xi}{h} \right) \right>_{\Dist'(T^*\T^2), \Cinfc(T^*\T^2)} . 
\end{align}
In particular, for any $R>0$ and $a \in S^1_\Lambda$, we have 
\begin{equation}
\label{eq: W^h reconstruction}
\left< W^h , a \left(x, \xi, \frac{P_\Lambda \xi}{h} \right) \right>_{\Dist'(T^*\T^2), \Cinfc(T^*\T^2)}
= \left< W^{h,\Lambda}_R , a \right>_{{S^1_\Lambda} ', S^1_\Lambda} 
+ \left< W^{h}_{R,\Lambda} , a \right>_{{S^1_\Lambda} ', S^1_\Lambda} .
\end{equation}

\bigskip
The following two propositions are the analogues of \cite{Ferm:00} in our context. They state the existence of the two-microlocal semiclassical measures, as the limit objects of $W^{h,\Lambda}_R$ and $W^{h}_{R,\Lambda}$.

\begin{proposition}
\label{proposition: nu Lambda}
There exists a subsequence $(h,w_h)$ and a nonnegative measure $\nu^\Lambda \in \M^+(T^*\T^2 \times \bS_\Lambda)$ such that, for all $a \in S^1_\Lambda$, we have
$$
\lim_{R \to \infty} \lim_{h \to 0} \left< W^{h,\Lambda}_R , a \right>_{{S^1_\Lambda} ', S^1_\Lambda}
= \left< \nu^\Lambda , a_{\hom}\left(x, \xi , \frac{\eta}{|\eta|}\right) \right>_{\M(T^*\T^2 \times \bS_\Lambda) , \Conc(T^*\T^2 \times \bS_\Lambda)} .
$$
\end{proposition}

To define the limit of the distributions $W^{h}_{R,\Lambda}$, we need first to introduce operator spaces and operator-valued measures, following~\cite{Gerard:91}. Given a Hilbert space $H$ (in the following, we shall use $H = L^2(\T_\Lambda)$), we denote respectively by $\K(H)$, $\calL^1(H)$ the spaces of compact and trace class operators on $H$. We recall that they are both two-sided ideals of the ring $\calL(H)$ of bounded operators on $H$. We refer for instance to~\cite[Chapter VI.6]{RS:book1} for a description of the space $\calL^1(H)$ and its basic properties. Given a Polish space $T$ (in the following, we shall use $T = T^*\T_{\Lambda^\perp}$), we denote by $\M^+(T;\calL^1(H))$ the space of nonnegative measures on $T$, taking values in $\calL^1(H)$. More precisely, we have $\rho \in \M^+(T;\calL^1(H))$ if $\rho$ is a bounded linear form on $\Conc(T)$ such that, for every nonnegative function $a \in \Conc(T)$, $\left< \rho , a \right>_{\M(T),\Conc(T)} \in \calL^1(H)$ is 
 a nonnegative hermitian operator. 
 As a consequence of~\cite[Theorem VI.26]{RS:book1}, these measures can be identified in a natural way to nonnegative linear functionals on $\Conc(T; \K(H))$.

\begin{proposition}
\label{proposition: rho Lambda}
There exists a subsequence $(h,w_h)$ and a nonnegative measure 
$$
\rho_\Lambda \in \M^+(T^*\T_{\Lambda^\perp} ; \calL^1(L^2(\T_\Lambda))),
$$
such that, for all $K \in \Cinfc(T^*\T_{\Lambda^\perp} ; \K(L^2(\T_\Lambda)))$, 
\begin{equation}
\label{eq: definition rho K}
\lim_{h\to 0}\left( K(s,h D_s) T_\Lambda w_h , T_\Lambda w_h \right)_{L^2(\T_{\Lambda^\perp} ;L^2(\T_\Lambda))}
= \tr \left\{ \int_{T^*\T_{\Lambda^\perp}} K(s, \sigma )\rho_\Lambda(ds, d \sigma) \right\} .
\end{equation}
Moreover (for the same subsequence), for all $a \in S^1_\Lambda$, we have
\begin{equation}
\label{eq: definition rho a}
\lim_{R \to \infty} \lim_{h \to 0} \left< W^{h}_{R,\Lambda} , a \right>_{{S^1_\Lambda} ', S^1_\Lambda}
= \tr \left\{ \int_{T^*\T_{\Lambda^\perp}} \Op_1^\Lambda \big(a(\tilde{\pi}_\Lambda(\sigma, y, 0), \eta) \big)\rho_\Lambda(ds, d \sigma) \right\} .
\end{equation}
\end{proposition}
In the left hand-side of \eqref{eq: definition rho K}, the inner product actually means
\begin{multline*}
\left( K(s,h D_s) T_\Lambda w_h , T_\Lambda w_h \right)_{L^2(\T_{\Lambda^\perp}L^2(\T_\Lambda))}
\\
= \int_{s \in \T_{\Lambda^\perp}, s'\in \Lambda^\perp, \sigma \in \Lambda^\perp} 
e^{\frac{i}{h}(s-s') \cdot \sigma}
\left(  K \big(\frac{s+s'}{2}, \sigma \big) T_\Lambda w_h (s',y) , T_\Lambda w_h (s,y)
\right)_{L^2_y(\T_\Lambda)}
ds \ ds' \ d\sigma .
\end{multline*}

In the expression~\eqref{eq: definition rho a}, remark that for each $\sigma \in \Lambda^\perp$, the operator $\Op_1^\Lambda \big(a(\tilde{\pi}_\Lambda(\sigma, y, 0), \eta)$ is in $\calL(L^2(\T_\Lambda))$.
Hence, its product with the operator $\rho_\Lambda(d s ,d \sigma)$ defines a trace-class operator.

\bigskip
Before proving Propositions~\ref{proposition: nu Lambda} and~\ref{proposition: rho Lambda}, we explain how to reconstruct the measure $\mu|_{\T^2 \times \Lambda^\perp}$ from the two-microlocal measures $\nu^\Lambda$ and $\rho_\Lambda$. 
This reduces the study of the measure $\mu$ to that of all two-microlocal measures $\nu^\Lambda$ and $\rho_\Lambda$, for $\Lambda \in \P$.


We denote by $\M^+_c(T)$ the set of compactly supported measures on $T$, and by $\left< \cdot , \cdot \right>_{\M_c(T), \Con^0(T)}$ the associated duality bracket.

\begin{proposition}
\label{proposition: reconstruction mu Lambda}
For all $a \in \Cinfc(T^*\T^2)$ having only $x$-Fourier modes in $\Lambda$ (\ie for all $a \in S^1_\Lambda$ independent of the third variable $\eta \in \left< \Lambda \right>$), we have
\begin{equation}
\label{eq: reconstruction interm}
 \left< \mu , a \right>_{\M(T^*\T^2), \Conc(T^*\T^2)}
= \left< \nu^\Lambda , a \right>_{\M(T^*\T^2 \times \bS_\Lambda), \Conc(T^*\T^2 \times \bS_\Lambda)}
+ \tr \left\{ \int_{T^*\T_{\Lambda^\perp}} m_{a \circ \tilde{\pi}_\Lambda}(\sigma)
\rho_\Lambda(ds, d \sigma) \right\} ,
\end{equation}
and 
\begin{align}
\label{eq: reconstruction mu Lambda}
 \left< \mu |_{\T^2 \times \Lambda^\perp}  , a \right>_{\M(T^*\T^2), \Conc(T^*\T^2)}
& = \left< \nu^\Lambda |_{\T^2 \times \Lambda^\perp \times \bS_\Lambda} , a \right>_{\M(T^*\T^2 \times \bS_\Lambda), \Conc(T^*\T^2 \times \bS_\Lambda)} \nonumber \\
& \quad + \tr \left\{ \int_{T^*\T_{\Lambda^\perp}} m_{a \circ \tilde{\pi}_\Lambda}(\sigma)
\rho_\Lambda(ds, d \sigma) \right\} ,
\end{align}
where for $\sigma \in \Lambda^\perp$, $m_{a \circ \tilde{\pi}_\Lambda}(\sigma)$ denotes the multiplication  in $L^2(\T_\Lambda)$ by the function $y \mapsto a \circ \tilde{\pi}_\Lambda (\sigma, y)$.

Moreover, we have $\nu^\Lambda \in \M^+_c(T^*\T^2 \times \bS_\Lambda)$ and $\rho_\Lambda \in \M^+_c(T^*\T_{\Lambda^\perp}; L^2(\T_\Lambda))$ (\ie both measures are compactly supported).
\end{proposition}

Formula \eqref{eq: reconstruction mu Lambda} follows immediately from \eqref{eq: reconstruction interm} by restriction. By the definition of the measure $\rho_\Lambda$, we see that it is already supported on $\T^2 \times \Lambda^\perp$ (see expression~\eqref{eq: def W^h_R,Lambda}).

The end of this section is devoted to the proofs of the three propositions, inspired by~\cite{Ferm:00, AM:11}. 

\begin{proof}[Proof of Proposition~\ref{proposition: nu Lambda}]
The Calder\'on-Vaillancourt theorem implies that the operators $$\Op_h\left(\left( 1 - \chi \left(\frac{|P_\Lambda \xi|}{R h}\right)\right) a \left(x, \xi, \frac{P_\Lambda \xi}{h} \right)\right)= \Op_1\left(\left( 1 - \chi \left(\frac{|P_\Lambda \xi|}{R }\right)\right) a \left(x, h\xi, P_\Lambda \xi\right)\right)$$
are uniformly bounded as $h\to 0$ and $R\to +\infty$. It follows that the family  $W^{h,\Lambda}_R$ is bounded in ${S^1_\Lambda} ' $, and thus there exists a subsequence $(h, w_h)$ and a distribution $\tilde{\mu}^\Lambda$ such that 
$$
\lim_{R \to \infty} \lim_{h \to 0} \left< W^{h,\Lambda}_R , a \right>_{{S^1_\Lambda} ', S^1_\Lambda}
= \left< \tilde{\mu}^\Lambda , a \left(x, \xi , \eta \right) \right>_{{S^1_\Lambda} ', S^1_\Lambda}.
$$
Because of the support properties of the function $\chi$, we notice that $\left< \tilde{\mu}^\Lambda , a \right>_{{S^1_\Lambda} ', S^1_\Lambda} = 0$ as soon as the support of $a$ is compact in the variable $\eta$. Hence, there exists a distribution $\nu^\Lambda \in \Dist'(T^*\T^2 \times \bS_{\Lambda})$ such that
$$
\left< \tilde{\mu}^\Lambda , a (x, \xi , \eta ) \right>_{{S^1_\Lambda} ', S^1_\Lambda} 
= \left< \nu^\Lambda , a_{\hom} \left(x, \xi , \frac{\eta}{|\eta|} \right) \right>_{\Dist'(T^*\T^2 \times \bS_{\Lambda}), \Cinfc(T^*\T^2 \times \bS_{\Lambda})}.
$$
Next, suppose that $a > 0$ (and that $\sqrt{1-\chi}$ is smooth). Then, using \cite[Corollary 35]{AM:11}, and setting 
$$
b^R(x,\xi) = \left( \left(1 - \chi \left(\frac{|P_\Lambda \xi|}{R h}\right)\right) a \left(x, \xi, \frac{P_\Lambda \xi}{h} \right)\right)^\frac12 ,
$$
there exists $C>0$ such that for all $h \leq h_0$ and $R \geq 1$, we have
$$
\left\|\Op_h\left( \left(1 - \chi \left(\frac{|P_\Lambda \xi|}{R h}\right)\right) a \left(x, \xi, \frac{P_\Lambda \xi}{h} \right)\right) - \Op_h(b^R)^2 \right\|_{\calL(L^2(\T^2))} \leq \frac{C}{R} .
$$
As a consequence, we have,
$$
 \left< W^{h,\Lambda}_R , a \right>_{{S^1_\Lambda} ', S^1_\Lambda} \geq \| \Op_h(b^R) w_h \|_{L^2(\T^2)}^2 - \frac{C}{R}\| w_h \|_{L^2(\T^2)}^2 ,
$$
so that the limit $\left< \nu^\Lambda , a_{\hom} \left(x, \xi , \frac{\eta}{|\eta|} \right) \right>_{\Dist'(T^*\T^2 \times \bS_{\Lambda}), \Cinfc(T^*\T^2 \times \bS_{\Lambda})}$ is nonnegative.
The distribution $\nu^\Lambda$ is nonnegative, and is hence a measure. This concludes the proof of Proposition~\ref{proposition: nu Lambda}.
\end{proof}

\begin{proof}[Proof of Proposition~\ref{proposition: rho Lambda}]
First, the proof of the existence of a subsequence $(h, w_h)$ and the measure $\rho_\Lambda$ satisfying \eqref{eq: definition rho K} is the analogue of Proposition~\ref{prop: existence mu} in the context of operator valued measures, viewing the sequence $w_h$ as a bounded sequence of $L^2( \T_{\Lambda^\perp}; L^2(\T_\Lambda))$. It follows the lines of this result, after the adaptation of the symbolic calculus to operator valued symbols (or more precisely, of \cite{Gerard:91} in the semiclassical setting).

\bigskip
Second, using the definition \eqref{eq: def W^h_R,Lambda} together with \eqref{eq: Op Lambda bis}, we have
\begin{align*}
\left< W^{h}_{R,\Lambda} , a \right>_{{S^1_\Lambda} ', S^1_\Lambda}
&= \left( \Op_h\left( \chi \left(\frac{|P_\Lambda \xi|}{R h}\right) a \left(x, \xi, \frac{P_\Lambda \xi}{h} \right) \right) w_h ,w_h \right)_{L^2(\T^2)} \\
&=  \left( \Op_h^{\Lambda^\perp} \circ \Op_h^{\Lambda} \left( \chi \left(\frac{|\eta|}{R h}\right) a \left(\tilde{\pi}_\Lambda(\sigma , y , \eta), \frac{\eta}{h} \right) \right) T_\Lambda w_h , T_\Lambda w_h \right)_{L^2(\T_{\Lambda^\perp}\times\T_{\Lambda})} .
\end{align*}
Hence, setting 
$$
a^h_{R, \Lambda}(\sigma, y, \eta) 
= \chi \left(\frac{|\eta|}{R}\right) a \left(\tilde{\pi}_\Lambda(\sigma , y , h\eta), \eta \right),
$$
we obtain
\begin{align*}
\left< W^{h}_{R,\Lambda} , a \right>_{{S^1_\Lambda} ', S^1_\Lambda}
=  \left( \Op_h^{\Lambda^\perp} \circ \Op_1^{\Lambda} \left( a^h_{R, \Lambda}(\sigma, y, \eta)  \right) T_\Lambda w_h , T_\Lambda w_h \right)_{L^2(\T_{\Lambda^\perp}\times\T_{\Lambda})} .
\end{align*}
We also notice that $\Op_1^{\Lambda} \left( a^h_{R, \Lambda} \right) \in \K(L^2(\T_\Lambda))$, for any $\sigma \in \Lambda^\perp$ since $a^h_{R, \Lambda}$ has compact support with respect to $\eta$. Moreover, for any $R>0$ fixed and $a \in S^1_\Lambda$, the Calder\'on-Vaillancourt theorem yields
$$
\Op_1^{\Lambda} \left( a^h_{R, \Lambda} \right) 
= \Op_1^{\Lambda} \left( a^0_{R, \Lambda} \right) + h B
$$
for some $B \in \calL(L^2(\T_\Lambda))$, uniformly bounded with respect to $h$. Using~\eqref{eq: definition rho K}, this implies that for any $R>0$ fixed and $a \in S^1_\Lambda$, we have 
\begin{equation*}
\lim_{h \to 0} \left< W^{h}_{R,\Lambda} , a \right>_{{S^1_\Lambda} ', S^1_\Lambda}
= \tr \left\{ \int_{T^*\T_{\Lambda^\perp}} \Op_1^\Lambda \big( a^0_{R, \Lambda}\big)\rho_\Lambda(ds, d \sigma) \right\} .
\end{equation*}
Moreover, we have 
$$
\lim_{R \to +\infty} \Op_1^\Lambda \big( a^0_{R, \Lambda}\big) =  \Op_1^\Lambda \big( a^0_{\infty, \Lambda}\big)
= \Op_1^\Lambda \big(a(\tilde{\pi}_\Lambda(\sigma, y, 0), \eta) \big),
$$
in the strong topology of $\Cinfc(T^*\T_{\Lambda^\perp} ; \calL(L^2(\T_\Lambda)))$. This proves~\eqref{eq: definition rho a} and concludes the proof of Proposition~\ref{proposition: rho Lambda}.
\end{proof}

\begin{proof}[Proof of Proposition~\ref{proposition: reconstruction mu Lambda}]
Taking $a \in S^1_\Lambda$, independent of the third variable $\eta \in \left< \Lambda \right>$ gives
\begin{equation*}
\left< W^h , a \left(x, \xi \right) \right>_{\Dist'(T^*\T^2), \Cinfc(T^*\T^2)}
\to \left< \mu |_{\T^2 \times \Lambda^\perp}  , a \right>_{\M(T^*\T^2), \Conc(T^*\T^2)}, 
\end{equation*}
together with 
\begin{equation*}
\left< W^{h,\Lambda}_R , a \right>_{{S^1_\Lambda} ', S^1_\Lambda} \to 
 \left< \nu^\Lambda , a \right>_{\M(T^*\T^2 \times \bS_\Lambda), \Conc(T^*\T^2 \times \bS_\Lambda)}, 
\end{equation*}
(according to Proposition~\ref{proposition: nu Lambda}) and 
\begin{equation*}
\left< W^{h}_{R,\Lambda} , a \right>_{{S^1_\Lambda} ', S^1_\Lambda} \to 
\tr \left\{ \int_{T^*\T_{\Lambda^\perp}} \Op_1^\Lambda \big(a(\tilde{\pi}_\Lambda(\sigma, y, 0)) \big)\rho_\Lambda(ds, d \sigma) \right\} 
= \tr \left\{ \int_{T^*\T_{\Lambda^\perp}} m_{a \circ \tilde{\pi}_\Lambda}(\sigma)
\rho_\Lambda(ds, d \sigma) \right\} ,
\end{equation*}
(according to Proposition~\ref{proposition: rho Lambda}).
Now, using the last three equations together with Equation~\eqref{eq: W^h reconstruction} directly gives \eqref{eq: reconstruction mu Lambda}.

\bigskip
As both terms in the right hand-side of~\eqref{eq: reconstruction mu Lambda} are nonnegative measures and the left-hand side is a compactly supported nonnegative measure, this implies that $\nu^\Lambda$ and $\rho_\Lambda$ are both compactly supported.
\end{proof}

\section{Propagation laws for the two-microlocal measures $\nu^\Lambda$ and $\rho_\Lambda$}
\label{section: Propagation laws}
In this section, we study the propagation properties of $\nu^\Lambda$ and $\rho_\Lambda$. The key point here is the use of the cutoff function introduced in Proposition~\ref{prop: existence chi}. 


We will use repeatedly the following fact, which follows from Item~\ref{item: 1-chi} in Proposition~\ref{prop: existence chi}: if $A$ is a bounded operator on $L^2(\T^2)$, we have
\begin{equation}
\label{e:cutoff}
(A w_h, w_h)_{L^2(\T^2)}= \left(A \chi_h^\Lambda w_h, \chi_h^\Lambda w_h\right)_{L^2(\T^2)} +\| A\|_{\calL(L^2)} \ o(1) .
\end{equation} 
To simplify the notation, we shall write $A_{c_0 ,  h}$ for $\chi_h^\Lambda A \chi_h^\Lambda$.

\subsection{Propagation of $\nu^\Lambda$}
\label{section: propagation nu lambda}
We define for $(x, \xi, \eta) \in T^* \T^2 \times \left< \Lambda \right>$ and $\tau \in \R$ the flows
$$
\phi^0_\tau(x, \xi, \eta) : = (x +\tau \xi , \xi , \eta) ,
$$
generated by the vector field $\xi \cdot \d_x$ and, for $\eta \neq 0$,
$$
\phi^1_\tau(x, \xi, \eta) : = \left(x +\tau \frac{\eta}{|\eta|} , \xi , \eta \right) ,
$$
generated by the vector field $\frac{\eta}{|\eta|} \cdot \d_x$.
With these definitions, we have the following propagation laws for the two-microlocal measure $\nu^\Lambda$.
\begin{proposition}
\label{proposition: propagation nu lambda}
The measure $\nu^\Lambda$ is $\phi^0_\tau$- and $\phi^1_\tau$-invariant, \ie
$$
(\phi^0_\tau)_*\nu^\Lambda = \nu^\Lambda 
\quad \text{and} \quad
(\phi^1_\tau)_*\nu^\Lambda = \nu^\Lambda , \quad \text{for every } \tau \in \R .
$$
\end{proposition}
The key result here is the additional ``transverse propagation law'' given by the flow $\phi^1_\tau$. The measure $\nu^\Lambda$ not only propagates along the geodesic flow $\phi^0_\tau$, but also along directions transverse to $\Lambda^\perp$.

\begin{proof}
Fix $a \in S^1_\Lambda$.
The computation done in \eqref{eq: propagation simple} is still valid 
replacing $a$ by $\left( 1 - \chi \left(\frac{|P_\Lambda \xi|}{R h}\right)\right) a \left(x, \xi, \frac{P_\Lambda \xi}{h} \right)$, since it only uses the fact that
$\Op_h\left(\left( 1 - \chi \left(\frac{|P_\Lambda \xi|}{R h}\right)\right) a \left(x, \xi, \frac{P_\Lambda \xi}{h} \right)\right)$ is bounded and that
$\|P_b^{h}w_h\|_{L^2(\T^2)} = o(h)$ and $\left(b w_h , w_h \right)_{L^2(\T^2)} =o(1)$. This yields
 \begin{multline*}
\lim_{h \to 0}\left< W^{h,\Lambda}_R , \xi \cdot \d_x  a \right>_{{S^1_\Lambda} ', S^1_\Lambda} \\
= \lim_{h \to 0} \left< W^h ,  \xi \cdot \d_x  \left\{ \left( 1 - \chi \left(\frac{|P_\Lambda \xi|}{R h}\right)\right)  a \left(x, \xi, \frac{P_\Lambda \xi}{h} \right) \right\} \right>_{\Dist'(T^*\T^2), \Cinfc(T^*\T^2)}  = 0 ,
\end{multline*}
and hence, in the limit $R \to + \infty$, we obtain
$$
\left< \nu^\Lambda , \xi \cdot \d_x a_{\hom}\left(x, \xi , \frac{\eta}{|\eta|}\right) \right>_{\M(T^*\T^2 \times \bS_\Lambda) , \Conc(T^*\T^2 \times \bS_\Lambda)}  = 0 .
$$
Replacing $a_{\hom}$ by $a_{\hom} \circ \phi^0_\tau$ and integrating with respect to the parameter $\tau$ gives $(\phi^0_\tau)_*\nu^\Lambda = \nu^\Lambda$, which concludes the first part of the proof.

\bigskip
Second, to prove the $\phi^1_\tau$-invariance of $\nu^\Lambda$ we compute
\begin{equation}
\label{eq: propagation nulambda proof}
\left< \nu^\Lambda , \frac{\eta}{|\eta|} \cdot \d_x a_{\hom}\left(x, \xi , \frac{\eta}{|\eta|}\right) \right>_{\M(T^*\T^2 \times \bS_\Lambda) , \Conc(T^*\T^2 \times \bS_\Lambda)}
= \lim_{R \to \infty}\lim_{h \to 0}\left< W^{h,\Lambda}_R , \frac{\eta}{|\eta|} \cdot \d_x  a \right>_{{S^1_\Lambda} ', S^1_\Lambda} .
\end{equation}
Setting 
$$
a^R (x, \xi , \eta) = \frac{1}{|\eta|}\left(1 - \chi \left(\frac{|\eta|}{R} \right)\right) 
a (x, \xi , \eta), 
$$
and
\begin{equation}A^R: = \Op_h \left( a^R\Big(x, \xi , \frac{P_\Lambda \xi}{h}\Big) \right)
\label{e:AR}\end{equation}
we have the relation
$$
\left< W^{h,\Lambda}_R , \frac{\eta}{|\eta|} \cdot \d_x  a \right>_{{S^1_\Lambda} ', S^1_\Lambda}=
-\frac{i}2([\Delta_\Lambda , A^R] w_h, w_h)_{L^2(\T^2)}
$$ 
where $\Delta_\Lambda = \d_{y}^2$ is the laplacian in the direction $\Lambda$.

\begin{lemma} \label{l:cutoff}For any given $c_0 >0$ and $R>0$, we have
$$([\Delta_\Lambda, A^R] w_h, w_h)_{L^2(\T^2)}=([\Delta_\Lambda, A_{c_0 ,  h}^R] w_h, w_h)_{L^2(\T^2)} +o(1).$$
\end{lemma}
We postpone the proof of Lemma \ref{l:cutoff} and first indicate how it allows to prove 
Proposition \ref{proposition: propagation nu lambda}. We now know that
$$\left< \nu^\Lambda , \frac{\eta}{|\eta|} \cdot \d_x a_{\hom}\left(x, \xi , \frac{\eta}{|\eta|}\right) \right>_{\M(T^*\T^2 \times \bS_\Lambda) , \Conc(T^*\T^2 \times \bS_\Lambda)}
= \lim_{R \to \infty}\lim_{h \to 0}-\frac{i}2([\Delta_\Lambda, A_{c_0 ,  h}^R] w_h, w_h)_{L^2(\T^2)}.$$

Recall that $a \in S^1_\Lambda$ implies that $a$ has only $x$-Fourier modes in $\Lambda$, \ie
$P_\Lambda \xi \cdot \d_x a = \xi \cdot \d_x a$. We have also assumed in this section that $b$ has 
only $x$-Fourier modes in $\Lambda$. As a consequence, we have
\begin{align}
\label{eq: propagation nulambda}
-\frac{i}2([\Delta_\Lambda, A_{c_0 ,  h}^R] w_h, w_h)_{L^2(\T^2)}&=
-\frac{i}2([\Delta, A_{c_0 ,  h}^R] w_h, w_h)_{L^2(\T^2)}
  \nonumber\\
& = \frac{i}{2h^2} \left(\left[ P_0^h , A_{c_0 ,  h}^R\right]w_h , w_h \right)_{L^2(\T^2)} .
\end{align}
 
Developing the last expression of~\eqref{eq: propagation nulambda}, we obtain
\begin{align}
\label{e:commut}
\frac{i}{2h^2} \left(\left[ P_0^h , A_{c_0 ,  h}^R\right]w_h , w_h \right)_{L^2(\T^2)}
&= \frac{i}{2h^2}\left( A_{c_0 ,  h}^R  w_h , P_b^{h}w_h \right)_{L^2(\T^2)} 
- \frac{i}{2h^2}\left( A_{c_0 ,  h}^R P_b^{h} w_h , w_h \right)_{L^2(\T^2)} \nonumber \\
& \quad \quad - \frac{1}{2h} \left( A_{c_0 ,  h}^R  w_h , b w_h \right)_{L^2(\T^2)}
- \frac{1}{2h} \left(A_{c_0 ,  h}^R b w_h , w_h \right)_{L^2(\T^2)}.
\end{align}
Since $A_{c_0 ,  h}^R$ is bounded in $ \calL(L^2(\T^2))$, its adjoint $A_{c_0 ,  h}^R$ is also bounded so that the first two terms in the last expression vanish in the limit $h \to 0$, using $\| P_b^{h}w_h \|_{L^2(\T^2)} = o(h^2)$. To estimate the last two terms, we use again the boundedness of $A^R$ and $(A^R)^*$
and write
$$|\left( A_{c_0 ,  h}^R  w_h , b w_h \right)_{L^2(\T^2)}|\leq \|A^R\|\, \|\chi_h^\Lambda b w_h\|_{L^2(\T^2)}\leq 2 c_0 h \ \|A^R\|,$$
according to Item~\ref{item: b leq c_0 h} in Proposition~\ref{prop: existence chi}.
It follows that 
$$\limsup_{h\to 0 } \left| \frac{1}{2h} \left( A_{c_0 ,  h}^R  w_h , b w_h \right)_{L^2(\T^2)}
+ \frac{1}{2h} \left(A_{c_0 ,  h}^R b w_h , w_h \right)_{L^2(\T^2)}\right|\leq 2 c_0 \sup  \|A^R\|.$$

Coming back to the expression~\eqref{eq: propagation nulambda proof}, we obtain
$$
\left|\left< \nu^\Lambda , \frac{\eta}{|\eta|} \cdot \d_x a_{\hom}\left(x, \xi , \frac{\eta}{|\eta|}\right) \right>_{\M(T^*\T^2 \times \bS_\Lambda) , \Conc(T^*\T^2 \times \bS_\Lambda)} \right|\leq 2 c_0 \sup  \|A^R\|
$$
and since $c_0$ was arbitrary, 
$$
\left< \nu^\Lambda , \frac{\eta}{|\eta|} \cdot \d_x a_{\hom}\left(x, \xi , \frac{\eta}{|\eta|}\right) \right>_{\M(T^*\T^2 \times \bS_\Lambda) , \Conc(T^*\T^2 \times \bS_\Lambda)}=0$$
Replacing $a_{\hom}$ by $a_{\hom} \circ \phi^1_\tau$ and integrating with respect to the parameter $\tau$ gives $(\phi^1_\tau)_*\nu^\Lambda = \nu^\Lambda$, which concludes the proof of Proposition~\ref{proposition: propagation nu lambda}.
\end{proof}

\begin{proof}[Proof of Lemma \ref{l:cutoff}]
We are going to show that
\begin{equation}
\label{e:firststep}
([\Delta_\Lambda, A_{c_0 ,  h}^R] w_h, w_h)_{L^2(\T^2)}=([\Delta_\Lambda, A^R]_{c_0 ,  h} w_h, w_h)_{L^2(\T^2)} +o(1) .
\end{equation}
Then, using the fact that $[\Delta_\Lambda, A^R]$ is a bounded operator (its symbol is  $\left(1 - \chi \left(\frac{|\eta|}{R} \right)\right) \frac{\eta}{|\eta|}\cdot
\partial_x a (x, \xi , \eta)$), together with~\eqref{e:cutoff}, this is also
$([\Delta_\Lambda, A^R] w_h, w_h)_{L^2(\T^2)} +o(1)$.

To prove \eqref{e:firststep}, we develop the difference $[\Delta_\Lambda, A_{c_0 ,  h}^R]-
[\Delta_\Lambda, A^R]_{c_0 ,  h}$ as
\begin{align}
\label{eq: bracket decomposition}
[\Delta_\Lambda, A_{c_0 ,  h}^R]-
[\Delta_\Lambda, A^R]_{c_0 ,  h}=
\left[ \d_y^2 , \chi_h^\Lambda \right] A^R  \chi_h^\Lambda
+\chi_h^\Lambda A^R \left[ \d_y^2 , \chi_h^\Lambda \right].
\end{align}
Then, writing
$$
\left[ \d_y^2 , \chi_h^\Lambda \right] = \d_y^2 \chi_h^\Lambda + 2\d_y \chi_h^\Lambda\,\, \d_y ,
$$
we have 
$$
\left(\left[ \d_y^2 , \chi_h^\Lambda \right] A^R  \chi_h^\Lambda w_h , w_h \right)_{L^2(\T^2)}
= \left( A^R  \chi_h^\Lambda \,w_h , \d_y^2\chi_h^\Lambda \,w_h \right)_{L^2(\T^2)}
+ \left(  \d_y \circ A^R  \,\chi_h^\Lambda \,w_h , 2 \d_y \chi_h^\Lambda \,w_h \right)_{L^2(\T^2)} .
$$
Recalling that the operator $\partial_y \circ A^R$ is bounded, and using Items~\ref{item: d_y chi} and~\ref{item: d_yy chi} in Proposition~\ref{prop: existence chi}, we obtain
$$
\left|\left(\left[ \d_y^2 , \chi_h^\Lambda \right] A^R  \chi_h^\Lambda w_h , w_h \right)_{L^2(\T^2)} \right|
\leq C \|\d_y^2\chi_h^\Lambda \, w_h \|_{L^2(\T^2)} + 
+ C  \| \d_y \chi_h^\Lambda\, w_h  \|_{L^2(\T^2)}  = o(1).
$$
The last term in~\eqref{eq: bracket decomposition} is handled similarly.
This finally implies~\eqref{e:firststep} and concludes the proof of Lemma~\ref{l:cutoff}.
\end{proof}

\subsection{Propagation of $\rho_\Lambda$}
\label{section: propagation rho lambda}
We denote by $(\omega^j_\Lambda , e^j_\Lambda)_{j\in \N}$ the eigenvalues and associated eigenfunctions of the operator $-\Delta_\Lambda = - \d_y^2$ forming a Hilbert basis of $L^2(\T_\Lambda)$. We shall use the projector onto low frequencies of $-\Delta_\Lambda$, \ie, for any $\omega \in \R_+$, the operator
$$
\Pi_\Lambda^\omega 
:= \sum_{\omega^j_\Lambda \leq \omega} ( \cdot , e^j_\Lambda )_{L^2(\T_\Lambda)} e^j_\Lambda ,
$$
which has finite rank.

We have the following propagation laws for the two-microlocal measure $\rho_\Lambda$.

\begin{proposition}
\label{proposition: commutation M Delta}
\begin{enumerate}
\item \label{item 2 rho} For any $K \in \Cinfc \big(T^*\T_{\Lambda^\perp};\K(L^2(\T_\Lambda))\big)$, independent of $s$ (\ie $K(s, \sigma) = K(\sigma)$), and any $\omega >0$, we have
$$
 \tr \left\{ \int_{T^* \T_{\Lambda^\perp}} [ \Delta_\Lambda , \Pi_\Lambda^\omega  K(\sigma) \Pi_\Lambda^\omega  ] \ \rho_\Lambda(ds, d \sigma) \right\}
 = 0 .
$$
\item \label{item 3 rho} Moreover, defining 
$$
M_\Lambda := \int_{\T_{\Lambda^\perp}\times \Lambda^\perp} \rho_\Lambda(ds, d \sigma) \in \calL^1(L^2(\T_\Lambda)),
$$
we have 
$$
[\Delta_\Lambda , M_\Lambda] = 0 .
$$
\end{enumerate}
\end{proposition}
 
Remark that for any $\sigma \in \Lambda^\perp$, the operator 
$$
[\Delta_\Lambda , \Pi_\Lambda^\omega K(\sigma) \Pi_\Lambda^\omega ]
= \Pi_\Lambda^\omega [\Delta_\Lambda ,  K(\sigma) ] \Pi_\Lambda^\omega,
$$
has finite rank, so the  right hand-side of Item~\ref{item 2 rho} is well-defined.
Note that the definition of $M_\Lambda$ has a signification since $\rho_\Lambda$ has a compact support, according to Proposition~\ref{proposition: reconstruction mu Lambda}.

The commutation relations of Items~\ref{item 2 rho} and \ref{item 3 rho} in this proposition correspond to propagation laws at the operator level. They are formulated here in a ``derivated form'', which, for Item~\ref{item 3 rho} for instance, is equivalent to 
$$
e^{i \tau \Delta_\Lambda} M_\Lambda e^{- i \tau \Delta_\Lambda} = M_\Lambda , \quad \text{for all } \tau \in \R ,
$$
in the ``integrated form''.

\begin{proof}
For $K \in \Cinfc(\Lambda^\perp ;\K(L^2(\T_\Lambda)))$ (in other words $K \in \Cinfc(T^*\T_{\Lambda^\perp };\K(L^2(\T_\Lambda)))$ independent of $s\in \T_{\Lambda^\perp }$), we denote
$$
K^\omega(\sigma) := \Pi_\Lambda^\omega K(\sigma) \Pi_\Lambda^\omega 
$$
and we note that  $K^\omega $ is also in $ \Cinfc(\Lambda^\perp ;\K(L^2(\T_\Lambda)))$.
Hence, we have
\begin{align*}
 \tr \left\{ \int_{T^* \T_{\Lambda^\perp}} [\Delta_\Lambda ,\Pi_\Lambda^\omega K (\sigma) \Pi_\Lambda^\omega ] \ \rho_\Lambda(ds, d \sigma) \right\}
 &= - \lim_{h \to 0}\left( [- \Delta_\Lambda , K^\omega(h D_s)] T_\Lambda w_h , T_\Lambda w_h \right)_{L^2(\T_{\Lambda^\perp} ;L^2(\T_\Lambda))} \end{align*}
 To show that this limit vanishes, we proceed as in lines \eqref{eq: propagation nulambda}, \eqref{e:commut} and in the subsequent calculation, replacing the operator $A^R$ by $K^\omega(h D_s)$.
 
With the notation $\Delta_\Lambda= \d_{y}^2$ and $\Delta_{\Lambda^\perp}=\d_{s}^2$, we first note that 
$$ 
\left( [- \Delta_\Lambda , K^\omega(h D_s)] T_\Lambda w_h , T_\Lambda w_h \right)_{L^2(\T_{\Lambda^\perp} ;L^2(\T_\Lambda))}=\left( [- \Delta , K^\omega(h D_s)] T_\Lambda w_h , T_\Lambda w_h \right)_{L^2(\T_{\Lambda^\perp} ;L^2(\T_\Lambda))},
$$
since $\Delta=\Delta_\Lambda+\Delta_{\Lambda^\perp}$ and since
$
[ \Delta_{\Lambda^\perp} , K^\omega(h D_s)] = 0.
$
As a matter of fact,  $K^\omega(h D_s) = \Op_h^\Lambda(K^\omega(\sigma))$ and $ \Delta_{\Lambda^\perp} =-h^{-2}\Op_h^\Lambda(|\sigma|^2)$ are both Fourier multipliers.

The following lemma is proved the same way as Lemma \ref{l:cutoff}
\begin{lemma} \label{l:cutoff1}For any given $c_0 >0$, we have
$$([\Delta_\Lambda, K^\omega(\sigma)]T_\Lambda w_h, T_\Lambda w_h)_{L^2(\T^2)}=([\Delta_\Lambda, K_{c_0 ,  h}^\omega(h D_s)] T_\Lambda w_h, T_\Lambda w_h)_{L^2(\T^2)} +o(1).$$
\end{lemma}
Here $K_{c_0 ,  h}^\omega(h D_s)$ means 
$\chi_h^\Lambda K^\omega(h D_s)\chi_h^\Lambda$.

Writing
$$
- h^2\Delta = T_\Lambda P^h_b T_\Lambda^*  - ih b\circ \pi_\Lambda ,
$$
we have
\begin{multline*}
\left( [- \Delta , K_{c_0 ,  h}^\omega(h D_s)] T_\Lambda w_h , T_\Lambda w_h \right)_{L^2(\T_{\Lambda^\perp} ;L^2(\T_\Lambda))} \\
= \frac{1}{h^2} \left( K_{c_0 ,  h}^\omega(h D_s) T_\Lambda w_h ,  T_\Lambda P^h_b w_h \right)_{L^2(\T_{\Lambda^\perp} ;L^2(\T_\Lambda))} 
- \frac{1}{h^2} \left( K_{c_0 ,  h}^\omega(h D_s) T_\Lambda P^h_b w_h , T_\Lambda w_h \right)_{L^2(\T_{\Lambda^\perp} ;L^2(\T_\Lambda))}\\
 \qquad  + \frac{i}{h} \left( K_{c_0 ,  h}^\omega(h D_s) T_\Lambda w_h ,  T_\Lambda (b w_h) \right)_{L^2(\T_{\Lambda^\perp} ;L^2(\T_\Lambda))} 
+ \frac{i}{h} \left(K_{c_0 ,  h}^\omega(h D_s) T_\Lambda (b w_h ) , T_\Lambda w_h \right)_{L^2(\T_{\Lambda^\perp} ;L^2(\T_\Lambda))}.
\end{multline*}
It follows, as in \eqref{e:commut}, that
$$\limsup_{h\to 0}|\left( [- \Delta , K_{c_0 ,  h}^\omega(h D_s)] T_\Lambda w_h , T_\Lambda w_h \right)_{L^2(\T_{\Lambda^\perp} ;L^2(\T_\Lambda))}|\leq 2c_0 \| K\|$$
and since $c_0$ was arbitrary, we can conclude that
$$\lim_{h\to 0}([\Delta_\Lambda, K^\omega(\sigma)]T_\Lambda w_h, T_\Lambda w_h)_{L^2(\T^2)}=0,$$
which concludes the proof of Item~\ref{item 2 rho}.

\bigskip
Item~\ref{item 2 rho} gives, for all $K \in \K(L^2(\T_\Lambda))$ constant (which is possible since $\rho_\Lambda(ds, d \sigma)$ has compact support),
$$
0 = \tr \left\{ \int_{T^* \T_{\Lambda^\perp}} [\Delta_\Lambda , K^\omega ]  \rho_\Lambda(ds, d \sigma) \right\}
 = \tr \left\{ [\Delta_\Lambda , K^\omega ] \int_{T^* \T_{\Lambda^\perp}} \rho_\Lambda(ds, d \sigma) \right\}
 = \tr \left\{ [\Delta_\Lambda , K^\omega ] M_\Lambda \right\}.
$$
Using that $\tr \left( A B\right) = \tr \left( B A\right)$ for all $A \in \calL^1$ and $B \in \calL$ together with the linearity of the trace (see~\cite[Theorem VI.25]{RS:book1}), we now obtain, for all $K \in \K(L^2(\T_\Lambda))$, and all $\omega >0$,
$$
0 = \tr \left\{ [\Delta_\Lambda ,\Pi_\Lambda^\omega  K \Pi_\Lambda^\omega ] M_\Lambda \right\}
= \tr \left\{ K \Pi_\Lambda^\omega [\Delta_\Lambda , M_\Lambda ] \Pi_\Lambda^\omega \right\}.
$$
Consequently, we have for all $\omega>0$, $\Pi_\Lambda^\omega  [\Delta_\Lambda , M_\Lambda ] \Pi_\Lambda^\omega =0$ (see~\cite[Theorem VI.26]{RS:book1}). Letting $\omega$ go to $+ \infty$, this yields $[\Delta_\Lambda , M_\Lambda ] =0$ and concludes the proof of Item~\ref{item 3 rho}.
\end{proof}

\section{The measures $ \nu^\Lambda$ and $\rho_\Lambda$ vanish identically. End of the proof of Theorem~\ref{th: stabilization torus}}
\label{section: end of the proof}
In this section, we prove that both measures $\nu^\Lambda$ and $\rho_\Lambda$ vanish when paired with the function $\left< b \right>_\Lambda$. Then, we deduce that these two measures vanish identically. In turn, this implies that $\mu|_{\T^2 \times \Lambda^\perp} = 0$, and finally that $\mu =0$, which will conclude the proof of Theorem~\ref{th: stabilization torus}.


\begin{proposition}
\label{proposition: nu and rho vanish on b}
We have 
\begin{equation*}
\left< \nu^\Lambda|_{\T^2 \times \Lambda^\perp \times \bS_\Lambda} , \left< b \right>_\Lambda \right>_{\M_c(T^*\T^2 \times \bS_\Lambda), \Con^0(T^*\T^2 \times \bS_\Lambda)} = 0 
, \quad \text{and} \quad 
 \tr \{ m_{\left< b\right>_\Lambda} M_\Lambda \}=0 .
\end{equation*}
\end{proposition}

As a consequence, we prove that $\rho_\Lambda$ and $\nu^\Lambda|_{\T^2 \times \Lambda^\perp \times \bS_\Lambda}$ vanish.

\begin{proposition}
\label{proposition: nu = M = 0}
We have $\rho_\Lambda = 0$ and $\nu^\Lambda|_{\T^2 \times \Lambda^\perp \times \bS_\Lambda}=0$. Hence $\mu|_{\T^2 \times \Lambda^\perp}=0$.
\end{proposition}

This allows to conclude the proof of Theorem~\ref{th: stabilization torus}. 
Indeed, as a consequence of the decomposition formula of Proposition~\ref{proposition: reconstruction mu Lambda}, we obtain, for all $\Lambda \in \P$, such that $\rk(\Lambda)=1$, $\mu|_{\T^2\times \Lambda^\perp}= 0$. Using the decomposition of the measure $\mu$ given in Lemma~\ref{lemma: decompose mu} together with Lemma~\ref{lemma: mu = 0 some directions}, this yields $\mu = 0$ on $\T^2$. This is in contradiction with $\mu(T^*\T^2) = 1$ (Proposition~\ref{proposition: zero first order info}), and this contradiction proves Theorem~\ref{th: stabilization torus}.


\bigskip
We now prove Propositions~\ref{proposition: nu and rho vanish on b} and \ref{proposition: nu = M = 0}

\begin{proof}[Proof of Proposition~\ref{proposition: nu and rho vanish on b}]
First, \eqref{eq: conditions sequence} implies that $(b v_h, v_h)_{L^2(\T^2)} \to 0$, and hence
$$
\left<\mu, b \right>_{\M_c(T^*\T^2), \Con^0 (T^*\T^2)} =0 .
$$
Then the decomposition given in Lemma \ref{lemma: decompose mu} into a sum of nonnegative measures yields that, for all $\Lambda \in \P$, 
\begin{equation}
\label{eq: mu b lambda = 0}
\left<\mu|_{\T^2 \times \Lambda^\perp}, b \right>_{\M_c(T^*\T^2), \Con^0 (T^*\T^2)} =0 ,
\end{equation}
since $b$ is also nonnegative.
Lemmata~\ref{lemma: mu only Lambda FM}, \ref{lemma: mu direction independent} and \ref{lemma: mu = 0 some directions} (see also Remark~\ref{rem: mu lambda 0}), then give
\begin{align}
\label{eq: decomp mu b lambda}
\left<\mu|_{\T^2 \times \Lambda^\perp} , \left< b \right>_\Lambda \right>_{\M_c(T^*\T^2), \Con^0 (T^*\T^2)} & = \left<\mu|_{\T^2 \times\left( \Lambda^\perp\setminus \{0\}\right)} , \left< b \right>_\Lambda \right>_{\M_c(T^*\T^2), \Con^0 (T^*\T^2)}  \nonumber \\
& = \left<\mu|_{\T^2 \times \Lambda^\perp} , b \right>_{\M_c(T^*\T^2), \Con^0 (T^*\T^2)} =0 ,
\end{align}
where the function $\left< b \right>_\Lambda$ is also nonnegative. The decomposition formula of Proposition~\ref{proposition: reconstruction mu Lambda} into the two-microlocal semiclassical measures then yields
\begin{align*}
\left< \mu|_{\T^2 \times \Lambda^\perp} , \left< b \right>_\Lambda \right>_{\M_c(T^*\T^2), \Con^0(T^*\T^2)}
& = \left< \nu^\Lambda|_{\T^2 \times \Lambda^\perp \times \bS_\Lambda} , \left< b \right>_\Lambda \right>_{\M_c(T^*\T^2 \times \bS_\Lambda), \Con^0(T^*\T^2 \times \bS_\Lambda)} \\
& \quad + \tr \left\{ \int_{T^*\T_{\Lambda^\perp}} m_{\left< b \right>_\Lambda}
\rho_\Lambda(ds, d \sigma) \right\} .
\end{align*}
Besides, the measure $\nu^\Lambda|_{\T^2 \times \Lambda^\perp \times \bS_\Lambda}$ is nonnegative, hence $\left< \nu^\Lambda|_{\T^2 \times \Lambda^\perp \times \bS_\Lambda} , \left< b \right>_\Lambda \right>_{\M_c(T^*\T^2 \times \bS_\Lambda), \Con^0(T^*\T^2 \times \bS_\Lambda)} \geq 0$. Similarly, $\rho_\Lambda \in \M^+_c(T^*\T_{\Lambda^\perp};\calL^1(\T_\Lambda))$ and the operator $m_{\left< b \right>_\Lambda} \in \calL(L^2(\T_\Lambda))$ is selfadjoint and nonnegative, which gives $\tr \left\{ \int_{T^*\T_{\Lambda^\perp}} m_{\left< b \right>_\Lambda} \rho_\Lambda(ds, d \sigma) \right\} \geq 0$. Using~\eqref{eq: mu b lambda = 0} and \eqref{eq: decomp mu b lambda}, this yields 
$$
\left< \nu^\Lambda|_{\T^2 \times \Lambda^\perp \times \bS_\Lambda} , \left< b \right>_\Lambda \right>_{\M_c(T^*\T^2 \times \bS_\Lambda), \Con^0(T^*\T^2 \times \bS_\Lambda)} = 0 ,
$$
and 
$$
\tr \left\{ \int_{T^*\T_{\Lambda^\perp}} m_{\left< b \right>_\Lambda} \rho_\Lambda(ds, d \sigma) \right\} = 0 .
$$
In this expression, the operator $m_{\left< b \right>_\Lambda}$ does not depend on $(s, \sigma)$, so that
$$
0 = \tr \left\{  m_{\left< b \right>_\Lambda} \int_{T^*\T_{\Lambda^\perp}} \rho_\Lambda(ds, d \sigma) \right\} 
=  \tr \{ m_{\left< b\right>_\Lambda} M_\Lambda \}, 
$$
which concludes the proof of Proposition~\ref{proposition: nu and rho vanish on b}.
\end{proof}

\begin{proof}[Proof of Proposition~\ref{proposition: nu = M = 0}]
Let us first prove that $\rho_\Lambda = 0$. We recall that the operator $M_\Lambda$ is a selfadjoint nonnegative trace-class operator.
Moreover, Proposition~\ref{proposition: commutation M Delta} implies that the operators $M_\Lambda$ and $\Delta_\Lambda$ commute. As a consequence, there exists a Hilbert basis $(\tilde{e}^j_\Lambda)_{j\in \N}$ of $L^2(\T_\Lambda)$ in which $M_\Lambda$ and $\Delta_\Lambda$ are simultaneously diagonal, \ie such that
$$
- \Delta_\Lambda \tilde{e}^j_\Lambda = \omega^j_\Lambda \tilde{e}^j_\Lambda , 
\quad \text{and} \quad 
 M_\Lambda \tilde{e}^j_\Lambda = \gamma^j_\Lambda \tilde{e}^j_\Lambda,
$$
where $(\gamma^j_\Lambda)_{j \in \N}$ are the associated eigenvalues of $M_\Lambda$. In particular, we have $\gamma^j_\Lambda \geq 0$ for all $j \in \N$ (and $\gamma^j_\Lambda \in \ell^1$).
Note that the basis  $(\tilde{e}^j_\Lambda)_{j\in \N}$ is not necessarily the same as the basis $(e^j_\Lambda)_{j\in \N}$ introduced in Section~\ref{section: propagation rho lambda}.

Using Proposition~\ref{proposition: nu and rho vanish on b}, together with the definition of the trace (see for instance~\cite[Theorem VI.18]{RS:book1}) we have 
$$
0 = \tr \{ m_{\left< b\right>_\Lambda} M_\Lambda \}
= \sum_{j \in \N} \left( m_{\left< b\right>_\Lambda} M_\Lambda \tilde{e}^j_\Lambda , \tilde{e}^j_\Lambda \right)_{L^2(\T_\Lambda)} 
= \sum_{j \in \N} \gamma^j_\Lambda \left( \left< b\right>_\Lambda \tilde{e}^j_\Lambda , \tilde{e}^j_\Lambda \right)_{L^2(\T_\Lambda)} .
$$
Since all terms in this sum are nonnegative (because both $\gamma^j_\Lambda$ and $\left< b\right>_\Lambda$ are), we deduce that for all $j \in \N$, 
$$
\gamma^j_\Lambda \left( \left< b\right>_\Lambda \tilde{e}^j_\Lambda , \tilde{e}^j_\Lambda \right)_{L^2(\T_\Lambda)} = 0.
$$
Suppose that $\gamma^j_\Lambda \neq 0$ for some $j \in \N$. Then, $\left( \left< b\right>_\Lambda \tilde{e}^j_\Lambda , \tilde{e}^j_\Lambda \right)_{L^2(\T_\Lambda)} = 0$ where $\left< b\right>_\Lambda$ is nonnegative and not identically zero on $\T_\Lambda$. This yields $\tilde{e}^j_\Lambda = 0$ on the nonempty open set $\{\left< b\right>_\Lambda>0\}$. Using a unique continuation property for eigenfunctions of the Laplace operator on $\T_\Lambda$, we finally obtain that the eigenfunction $\tilde{e}^j_\Lambda$ vanishes identically on $\T_\Lambda$. This is absurd, and thus we must have $\gamma^j_\Lambda = 0$ for all $j \in \N$, so that $M_\Lambda = 0$. Since $\rho_\Lambda \in \M^+(T^*\T_{\Lambda^\perp};\calL^1(\T_\Lambda))$, this directly gives $\rho_\Lambda = 0$.  

\bigskip
Next, we prove that $\nu^\Lambda = 0$. This is a consequence of the additional propagation law of $\nu^\Lambda$ with respect to the flow $\phi^1_\tau$ (see Section~\ref{section: propagation nu lambda}). Indeed the torus $\T_\Lambda$ has dimension one, $(\phi^1_\tau)_*\nu^\Lambda = \nu^\Lambda$ (according to Proposition~\ref{proposition: propagation nu lambda}) and, using Proposition~\ref{proposition: nu and rho vanish on b}, $\nu^\Lambda$ vanishes on the (nonempty) set $\{ \left< b \right>_\Lambda > 0\} \times \R^2 \times \bS_\Lambda$ (with $\{ \left< b \right>_\Lambda > 0\}$ clearly satisfying GCC on $\T_\Lambda$). Hence, $\nu^\Lambda = 0$.

\bigskip
To conclude the proof of Proposition~\ref{proposition: nu = M = 0}, it only remains to use the decomposition formula~\eqref{eq: reconstruction mu Lambda} which directly yields $\mu|_{\T^2 \times \Lambda^\perp} = 0$.
\end{proof}

\section{Proof of Proposition~\ref{prop: mu Lambda w_h}}
\label{section: proof nu alpha}
In this section, we prove Proposition~\ref{prop: mu Lambda w_h}. For this, we consider two-microlocal semiclassical measures at the scale $h^\alpha$. The setting is close to that of~\cite{Ferm:05}.

We shall see that the concentration rate of the sequence $v_h$ towards the direction $\Lambda^\perp$ is of the form $h^\alpha$ for all $\alpha \leq \frac{3+\delta}{4}$.

First, Lemma~\ref{lemma: mu direction independent} yields $\mu|_{\T^2 \times \Lambda^\perp} = \left< \mu \right>_\Lambda|_{\T^2 \times \Lambda^\perp}$ (see also Remark~\ref{rem: mu lambda 0}), \ie
$$
\left< \mu |_{\T^2 \times \Lambda^\perp} , a \right>_{\M(T^*\T^2) , \Conc(T^*\T^2)}
= \left< \mu |_{\T^2 \times \Lambda^\perp}  , \left< a \right>_\Lambda \right>_{\M(T^*\T^2) , \Conc(T^*\T^2)} ,
$$
and it suffices to characterize the action of $\mu |_{\T^2 \times \Lambda^\perp}$ on $\Lambda^\perp$-invariant symbols. Recall that, for all $a\in \Cinfc(T^*\T^2)$,
$$
\left< \mu , a \right>_{\M(T^*\T^2) , \Conc(T^*\T^2)}
 = \lim_{h \to 0} \left( \Op_h(a)  v_h , v_h \right)_{L^2(\T^2)} .
$$

In this section, the assumption $\sqrt{b} \in \Cinf(\T^2)$ is used in an essential way for the propagation result of Lemma~\ref{lemma: propagation nu lambda beta} below. Like in \eqref{eq: def W^h_R^Lambda} and \eqref{eq: def W^h_R,Lambda}, let us define~:
 
\begin{align}
 &
\left< V^{h,\Lambda}_R , a \right>_{{S^1_\Lambda} ', S^1_\Lambda}
:= \left< V^h ,  \left( 1 - \chi \left(\frac{|P_\Lambda \xi|}{R h}\right)\right) a \left(x, \xi, \frac{P_\Lambda \xi}{h} \right) \right>_{\Dist'(T^*\T^2), \Cinfc(T^*\T^2)},\\
&\left< V^{h}_{R,\Lambda} , a \right>_{{S^1_\Lambda} ', S^1_\Lambda}
:= \left< V^h , \chi \left(\frac{|P_\Lambda \xi|}{R h}\right) a \left(x, \xi, \frac{P_\Lambda \xi}{h} \right) \right>_{\Dist'(T^*\T^2), \Cinfc(T^*\T^2)} ,
\end{align}
for $a\in S^1_\Lambda$.

We take $R=R(h)=h^{-(1-\alpha)}$ for some $\alpha \in (0, 1)$, so that $Rh=h^\alpha$. The proof of Proposition~\ref{proposition: nu Lambda} applies verbatim and shows the existence of
a subsequence $(h,v_h)$ and a nonnegative measure $\nu_\alpha^\Lambda \in \M^+(T^*\T^2 \times \bS_\Lambda)$ such that, for all $a \in S^1_\Lambda$, we have
$$
  \lim_{h \to 0} \left< V^{h,\Lambda}_{R(h)} , a \right>_{{S^1_\Lambda} ', S^1_\Lambda}
= \left< \nu_\alpha^\Lambda , a_{\hom}\left(x, \xi , \frac{\eta}{|\eta|}\right) \right>_{\M(T^*\T^2 \times \bS_\Lambda) , \Conc(T^*\T^2 \times \bS_\Lambda)} .
$$

\begin{proposition}
\label{p:h14} 
Let $R(h)=h^{-(1-\alpha)}$ with $\alpha \leq \frac{3+\delta}{4}$. Then
$$\nu_\alpha^\Lambda |_{\T^2\times \left(\Lambda^\perp\setminus \{0\}\right)\times  \bS_\Lambda}=0
$$
\end{proposition}
 
The proof of Proposition~\ref{p:h14} relies on the following propagation result. \begin{lemma} 
\label{lemma: propagation nu lambda beta}
For $\alpha\leq\frac{3+\delta}{4}$ the measure $\nu_\alpha^\Lambda$ is $\phi^0_\tau$- and $\phi^1_\tau$-invariant, \ie
$$
(\phi^0_\tau)_*\nu_\alpha^\Lambda = \nu_\alpha^\Lambda 
\quad \text{and} \quad
(\phi^1_\tau)_*\nu_\alpha^\Lambda = \nu_\alpha^\Lambda , \quad \text{for every } \tau \in \R .
$$
\end{lemma}
The proof is very similar to that of Proposition \ref{proposition: propagation nu lambda} but does not use Assumption~\eqref{eq: assumption b}.
\begin{proof} The proof of $ \phi_\tau^0$-invariance is strictly identical to what has been done for Proposition \ref{proposition: propagation nu lambda} and thus we focus on the $\phi_\tau^1$-invariance. Equation \eqref{e:commut} still holds with $R(h)=h^{-(1-\alpha)}$, now reading
\begin{align*}
\left< V^{h,\Lambda}_{R(h)} , \frac{\eta}{|\eta|} \cdot \d_x  a \right>_{{S^1_\Lambda} ', S^1_\Lambda}
&= \frac{i}{2h^2}\left( A^{R(h)}  v_h , P_b^{h}v_h \right)_{L^2(\T^2)} 
- \frac{i}{2h^2}\left( A^{R(h)} P_b^{h} v_h , v_h \right)_{L^2(\T^2)} \\
& \quad \quad - \frac{1}{2h} \left( A^{R(h)}  v_h , b v_h \right)_{L^2(\T^2)}
- \frac{1}{2h} \left( A^{R(h)} b v_h , v_h \right)_{L^2(\T^2)}
\end{align*}
where $A^R$ was defined in \eqref{e:AR}.
Using $\|P_b^h v_h\|_{L^2(\T^2)} = o(h^{1+\delta})$ together with the boundedness of $ A^{R(h)}$, it follows that
$$\lim_{h\to 0}\left< V^{h,\Lambda}_{R(h)} , \frac{\eta}{|\eta|} \cdot \d_x  a \right>_{{S^1_\Lambda} ', S^1_\Lambda}=
\lim_{h\to 0} \Big( - \frac{1}{2h} \big( A^{R(h)}  v_h , b v_h \big)_{L^2(\T^2)}
- \frac{1}{2h} \big( A^{R(h)} b v_h , v_h \big)_{L^2(\T^2)} \Big).
$$
Recall from \eqref{eq: conditions sequence} that $\|\sqrt{b}v_h\|_{L^2(\T^2)}= o(h^{\frac{1+\delta}{2}}).$
In addition, it follows from standard microlocal calculus that
$$
[ A^{R(h)}, \sqrt{b}]= \O_{\calL(L^2)}(R(h)^{-2}).
$$ 
%
We can thus write
\begin{align*}
\left< V^{h,\Lambda}_{R(h)} , \frac{\eta}{|\eta|} \cdot \d_x  a \right>_{{S^1_\Lambda} ', S^1_\Lambda}
& = o(1)
- \frac{1}{h} \big( A^{R(h)} \sqrt{b} v_h ,  \sqrt{b}v_h \big)_{L^2(\T^2)} 
+ \frac{1}{2h}\big(\sqrt{b} [A^{R(h)}, \sqrt{b}]  v_h ,  v_h \big)_{L^2(\T^2)}\\ 
& \quad + \frac{1}{2h} \big( [ \sqrt{b}, A^{R(h)}] \sqrt{b} v_h ,  v_h \big)_{L^2(\T^2)} \\
& = o(1)+o(R(h)^{-2} h^{\frac{-1 + \delta}{2}})  = o(1)+o(h^{\frac32 +\frac{\delta}{2}-2 \alpha}) ,
\end{align*}
which vanishes if we take $\alpha\leq\frac{3+\delta}{4}$.
\end{proof}

\begin{proof}[Proof of Proposition~\ref{p:h14}]
To prove Proposition \ref{p:h14}, we first note that 
$$
\left< \nu^\Lambda_\alpha|_{\T^2\times \left(\Lambda^\perp\setminus \{0\}\right)\times  \bS_\Lambda} , \left< b \right>_\Lambda \right>_{\M_c(T^*\T^2 \times \bS_\Lambda), \Con^0(T^*\T^2 \times \bS_\Lambda)} = 0 ,
$$
since $\nu^\Lambda_\alpha$ is $(\phi^0_\tau)$-invariant and $
\left< \nu^\Lambda_\alpha, b\right>_{\M_c(T^*\T^2 \times \bS_\Lambda), \Con^0(T^*\T^2 \times \bS_\Lambda)} =0$.
Then, the $\phi_\tau^1$-invariance of $\nu_\alpha^\Lambda$ implies that $\nu_\alpha^\Lambda |_{\T^2\times \left(\Lambda^\perp\setminus \{0\}\right)\times  \bS_\Lambda}$ vanishes. 
\end{proof}

\begin{proof}[Proof of Proposition~\ref{prop: mu Lambda w_h}]
Proposition~\ref{p:h14} implies that
$$ \left< \mu|_{\T^2 \times \Lambda^\perp} , a \right>_{\M(T^*\T^2) , \Conc(T^*\T^2)} =
\lim_{h \to 0} \left( \Op_h\left( \chi\left(\frac{|P_\Lambda \xi|}{h^\alpha} \right)  a(x, \xi) \right)  v_h , v_h \right)_{L^2(\T^2)} 
$$
for all $\alpha \leq \frac{3+\delta}{4}$ and $a \in \Cinfc(T^*\T^2)$.
The same holds if we replace $\chi$ by $\chi^2$~:
$$ \left< \mu|_{\T^2 \times \Lambda^\perp} , a \right>_{\M(T^*\T^2) , \Conc(T^*\T^2)} =
\lim_{h \to 0} \left( \Op_h\left( \chi^2\left(\frac{|P_\Lambda \xi|}{h^\alpha} \right)  a(x, \xi) \right)  v_h , v_h \right)_{L^2(\T^2)}. 
$$

Since 
\begin{align}
\label{eq: calcul psido}
 \Op_h\left( \chi^2\left(\frac{|P_\Lambda \xi|}{h^\alpha} \right)  a(x, \xi) \right) 
=  \Op_h\left( \chi\left(\frac{|P_\Lambda \xi|}{h^\alpha}\right) \right) \Op_h(a) \Op_h\left( \chi\left(\frac{|P_\Lambda \xi|}{h^\alpha}\right) \right) + \O(h^{1-\alpha}) ,
\end{align}
we obtain  
\begin{equation*}
\left< \mu|_{\T^2 \times \Lambda^\perp} , a \right>_{\M(T^*\T^2) , \Conc(T^*\T^2)} 
= \lim_{h \to 0} \left( \Op_h(a) \Op_h\left( \chi\left(\frac{|P_\Lambda \xi|}{h^\alpha}\right) \right) v_h ,
 \Op_h\left( \chi\left(\frac{|P_\Lambda \xi|}{h^\alpha}\right) \right) v_h \right)_{L^2(\T^2)}, 
\end{equation*}
for all $\alpha \leq \frac{3+\delta}{4}$ and $a \in \Cinfc(T^*\T^2)$.\end{proof}

\section{Proof of Proposition~\ref{prop: existence chi}: existence of the cutoff function}
\label{section: existence chi}

Given a constant $c_0 >0$, we define the following subsets of $\T^2$: 
$$
\E_h  = \left< \{b > c_0 h\} \right>_{\Lambda} ,
 \qquad \F_h = \left< \bigcup_{x \in \{b > c_0 h\} } B(x, (c_0 h)^{2\eps})\right>_\Lambda
= \bigcup_{x \in \E_h} B(x, (c_0 h)^{2\eps}) , \qquad 
 \G_h =  \F_h \setminus \E_h ,
$$
where for $U \subset \T^2$, we denote $\left< U \right>_{\Lambda} : = \bigcup_{\tau \in \R}\{U + \tau \sigma\}$ for some $\sigma \in \Lambda^\perp \setminus \{0\}$. Remark that $\E_h \subset \F_h$ and that $\T^2 = \E_h \cup \G_h \cup (\T^2 \setminus \F_h)$. Note also that the sets $\E_h , \F_h$ are non-empty for $h$ small enough, and that $\G_h$ is non empty (for $h$ small enough) as soon as $b$ vanishes somewhere on $\T^2$ (this condition is assumed here since otherwise, GCC is satisfied). 

%
%

In this section, we construct the cutoff function $\chi_h^\Lambda$ needed to prove the propagation results of Section~\ref{section: Propagation laws}. In particular, this function will be $\Lambda^\perp$-invariant and will satisfy $\chi_h^\Lambda = 0$ on $\E_h$ and $\chi_h^\Lambda = 1$ on $\T^2 \setminus \F_h$.

\bigskip
The proof of Proposition~\ref{prop: existence chi} relies on three key lemmata.
The first key lemma is a precised version of Proposition~\ref{proposition: zero first order info} concerning the localization in $T^*\T^2$ of the semiclassical measure $\mu$. It is an intermediate step towards the propagation result stated in Lemma~\ref{lemma: transport psi}.

\begin{lemma}
\label{lemma: propagation precised}
For any $\chi \in \Cinfc(\R)$, such that $\chi = 1$ in a \nhd of the origin, for all $a \in \Cinfc(T^*\T^2)$, and $\gamma \leq \frac{3+\delta}{2}$, we have 
\begin{align}
\label{eq: loc xi =1}
\left(\Op_h(a) w_h , w_h \right)_{L^2(\T^2)} =
\left(\Op_h(a) \Op_h \left(\chi\left(\frac{|\xi|^2-1}{h^\gamma} \right) \right) w_h , w_h \right)_{L^2(\T^2)} 
+ o(h^{\frac{3+\delta}{2}-\gamma})\|\Op_h(a)\|_{\calL(L^2)} ,
\end{align}
For all $a \in \Cinfc(T^*\T^2)$ and all $\tau \in \R$,
\begin{align*}
\left(\Op_h(a \circ \phi_\tau) w_h , w_h \right)_{L^2(\T^2)} 
= \left(\Op_h(a) w_h , w_h \right)_{L^2(\T^2)} +o(\tau h^{\frac{1+\delta}{2}})\|\Op_h(a \circ \phi_t)\|_{L^\infty(0,\tau ;\calL(L^2(\T^2)))}
\end{align*}
\end{lemma}
In this statement, we used the notation 
$$
\|\Op_h(a \circ \phi_t)\|_{L^\infty(0,\tau ;\calL(L^2(\T^2)))} := 
\sup_{t\in (0,\tau)}\|\Op_h(a \circ \phi_t)\|_{\calL(L^2(\T^2))} .
$$
In turn, this lemma implies the following transport property.

\begin{lemma}
\label{lemma: transport psi}
Suppose that the coefficients $\alpha, \eps$ satisfy 
\begin{align}
\label{eq: condition coeffs}
0< 10\eps \leq \alpha , \quad \text{and} \quad \alpha + 2\eps \leq 1.
\end{align}
Then, for any time $\tau \in \R$ uniformly bounded with respect to $h$, and any $h$-family of functions $\psi = \psi_h \in \Cinfc(\T^2)$ satisfying 
\begin{align}
\label{eq: derivate psi}
\|\d_x^k \psi\|_{L^\infty(\T^2)} \leq C_k h^{-2 \eps |k|},\quad \text{for all } k\in \N^2,
\end{align}
we have, 
\begin{align}
\label{eq: propagation cutoff}
\left(\psi(s,y) w_h , w_h \right)_{L^2(\T^2)}
& = \left( \psi (s + \tau, y ) w_h , w_h \right)_{L^2(\T^2)} 
+  \left( \psi ( s - \tau , y) w_h , w_h \right)_{L^2(\T^2)}  
 \nonumber \\
& \quad  
 + \O (h^{\alpha - 10 \eps}) 
 + \O(h^{1-\alpha-2 \eps})  
 + o( h^{\frac{1+\delta}{2}}), 
\end{align}
where the coordinates $(s,y)$ are the ones introduced in Section~\ref{s:subtori}.
\end{lemma}

In view of Proposition~\ref{prop: existence chi}, this lemma will allow us to propagate the smallness of the sequence $w_h$ above the set $\{b > c_0 h\}$ to all $\E_h$.

The third key lemma states a property of the damping function $b$, as a consequence of Assumption~\ref{eq: assumption b}.
\begin{lemma}
\label{lemma: property b}
There exists $b_0 = b_0(\eps) >0$ such that for all $x \in \T^2$ satisfying $0<b(x)<b_0$ and for all $z\in B(x,b(x)^{2\eps})$, we have $b(z) \geq \frac{b(x)}{2}$.
\end{lemma}

With these three lemmata, we are now able to prove Proposition~\ref{prop: existence chi}.

\begin{proof}[Proof of Proposition~\ref{prop: existence chi}]
In the coordinates $(s, y)$ of Section~\ref{s:subtori}, we can write
$$
\E_h = \T_{\Lambda^\perp} \times E_h , \quad 
\F_h = \T_{\Lambda^\perp} \times F_h , \quad 
\text{with } E_h \subset F_h \subset \T_\Lambda .
$$
Here, $F_h$ is a union of intervals and has uniformly bounded total length. We can hence cover $F_h$ with $C_1 h^{-2\eps}$ subsets of length of order $(c_0 h)^{2\eps}/2$, overlapping on intervals of length of order $(c_0 h)^{2\eps}/10$. Associated to this covering, we denote by $(\psi_j)_{j \in \{1,\dots , J\}}$, $J=J(h)$, a  smooth partition of unity on $E_h$, satisfying moreover
\begin{itemize}
\item $\psi_j \in \Cinfc(F_h)$;
\item $\sum_{j=1}^J \psi_j (y) = 1$ for $y \in E_h$;
\item $\|\d_y^m \psi_j \|_{L^\infty(\T_\Lambda)} \leq C_m h^{-2 \eps m}$, for all $m\in \N$;
\item $J = J(h) \leq C h^{-2\eps}$.
\end{itemize}
Similarly, we cover $\T_{\Lambda^\perp}$ with $C_2 h^{-2\eps}$ subsets of length of order $(c_0 h)^{2\eps}/2$, overlapping on intervals of length of order $(c_0 h)^{2\eps}/10$, and define $(\psi_k)_{k \in \{1,\dots , K\}}$ an associated partition of unity on $\T_{\Lambda^\perp}$ satisfying
\begin{itemize}
\item $\psi_k \in \Cinfc(\T_{\Lambda^\perp})$;
\item $\sum_{k=1}^K \psi_k (s) = 1$ for $s \in \T_{\Lambda^\perp}$;
\item $\|\d_s^m \psi_k \|_{L^\infty(\T_{\Lambda^\perp})} \leq C_m h^{-2 \eps m}$, for all $m\in \N$;
\item $K = K(h) \leq C h^{-2\eps}$,
\item for any $k, k_0 \in \{1 , \dots, K\}^2$, there exists $\tau_k$ satisfying $|\tau_k| \leq \text{Length}(\T_{\Lambda^\perp})\leq C$ and 
$\psi_{k}(s+ \tau_k) = \psi_{k_0}(s)$.
\end{itemize}

We set 
$$
\psi_{kj}(s,y) := \psi_k(s)\psi_j(y) , \qquad \text{and} \qquad
\chi_h^\Lambda(s,y) = 1 - \sum_{j=1}^J \sum_{k=1}^K \psi_{kj}(s,y) \in \Cinf(\T^2), 
$$
 which satisfies $\d_s \chi_h^\Lambda(s,y) = 0$, \ie $\chi_h^\Lambda$ is $\Lambda^\perp$-invariant, together with
\begin{itemize}
\item $\chi_h^\Lambda = 0$ on $\E_h$ and hence $b \leq c_0 h$ on $\supp(\chi_h^\Lambda )$;
\item $\chi_h^\Lambda = 1$ on $\T^2 \setminus \F_h$;
\item $\chi_h^\Lambda \in [0,1]$ on $\G_h$, with $|\d_y \chi_h^\Lambda| \leq C h^{-2\eps}$ and $|\d_y^2 \chi_h^\Lambda| \leq C h^{-4\eps}$.
\end{itemize}
To conclude the proof of Proposition~\ref{prop: existence chi}, it remains to check Item~\ref{item: 1-chi} ($\|(1 -\chi_h^\Lambda ) w_h\|_{L^2(\T^2)} = o(1)$), Item~\ref{item: d_y chi} ($\|\d_{y}\chi_h^\Lambda w_h\|_{L^2(\T^2)} = o(1)$) and Item~\ref{item: d_yy chi} ($\|\d_{y}^2 \chi_h^\Lambda w_h\|_{L^2(\T^2)} = o(1)$).

Now, let us fix $j_0 \in \{1 , \dots, J\}$. Because of the definition of the set $\E_h$, there exists $k_0 \in \{1 , \dots, K\}$ and $x_0 \in \{b> c_0 h\}$ such that $\supp(\psi_{k_0 j_0}) \subset B(x_0 , (c_0 h)^{2\eps})$. According to Lemma~\ref{lemma: property b}, we have $B(x_0 , (c_0 h)^{2\eps}) \subset \{b> \frac{c_0 h }{2}\}$, so that $\supp(\psi_{k_0 j_0}) \subset \{b> \frac{c_0 h }{2}\}$. This yields
$$
 \frac{c_0 h }{2} (\psi_{k_0 j_0} w_h,w_h)_{L^2(\T^2)}
 \leq (b \psi_{k_0 j_0} w_h,w_h)_{L^2(\T^2)}
 = o(h^{1+\delta}) ,
$$
and hence $(\psi_{k_0 j_0} w_h,w_h)_{L^2(\T^2)} = o(h^{\delta})$. Moreover, for any $k \in \{1 , \dots, K\}$, there exists $\tau_k$ satisfying $|\tau_k| \leq C_2$ with
$$
\psi_{k j_0}(s+ \tau_k,y) = \psi_{k_0 j_0}(s,y) .
$$
Hence, using~\eqref{eq: propagation cutoff}, we obtain
\begin{align}
\label{eq: conclusion psijk}
o(h^\delta) &= ( \psi_{k_0 j_0}(s,y) w_h,w_h)_{L^2(\T^2)}
= (\psi_{k j_0}(s+ \tau_k,y) w_h,w_h)_{L^2(\T^2)} \nonumber\\
&= (\psi_{k j_0}(s +2 \tau_k,y) w_h,w_h)_{L^2(\T^2)} 
+ (\psi_{k j_0}(s,y) w_h,w_h)_{L^2(\T^2)} \nonumber\\
& \quad +  \O (h^{\alpha - 10 \eps})
+ \O(h^{1-\alpha-2 \eps})
 +o(h^{\frac{1+\delta}{2}}) .
\end{align}
Since both terms on the right hand-side are nonnegative, this implies $(\psi_{k j_0}(s,y) w_h,w_h)_{L^2(\T^2)} = o(h^\delta)$ as soon as 
\begin{equation*}
\left\{
\begin{array}{l}
\dsp \alpha - 10 \eps > \delta , \\
1-\alpha-2 \eps> \delta , \\
\dsp \frac{1+\delta}{2} \geq \delta , 
\end{array}
\right.
\end{equation*}
(which implies \eqref{eq: condition coeffs}).
From now on we will take $\delta = 8 \eps$ (this choice is explained in the following lines). The existence of $\alpha$ satisfying this condition together with \eqref{eq: condition alpha} and $\alpha < 3/4$, is equivalent to having $\eps < \frac{1}{76}$.

To conclude the proof of Proposition~\ref{prop: existence chi}, we first compute
$$
((1- \chi_h^\Lambda ) w_h,w_h)_{L^2(\T^2)}  
=  \sum_{j=1}^J \sum_{k=1}^K (\psi_{kj} w_h,w_h)_{L^2(\T^2)}  
= C h^{-4\eps} o(h^\delta) = o(1),
$$
since $\delta \geq 4\eps$. This proves Item~\ref{item: 1-chi}. Next, we have by construction $\supp(\d_{y}^2 \chi_h^\Lambda) \subset \supp(\d_{y}\chi_h^\Lambda) \subset \G_h$ with $\|\d_{y}\chi_h^\Lambda\|_{L^\infty(\T^2)} = \O(h^{-2\eps})$, $\|\d_{y}^2 \chi_h^\Lambda\|_{L^\infty(\T^2)} = \O(h^{-4\eps})$. Hence, covering $\supp(\d_{y}\chi_h^\Lambda))$ by balls of radius $(c_0 h)^{2\eps}$ and using a propagation argument similar to~\eqref{eq: conclusion psijk} shows that we have $\|w_h\|_{L^2(\supp(\d_{y}\chi_h^\Lambda))} = o(h^{\frac{\delta}{2}})$. We thus obtain 
$$
\|\d_{y}\chi_h^\Lambda w_h\|_{L^2(\T^2)} = o(h^{\frac{\delta}{2}- 2\eps}) = o(1) ,
\qquad 
\|\d_{y}^2\chi_h^\Lambda w_h\|_{L^2(\T^2)} = o(h^{\frac{\delta}{2}- 4\eps}) = o(1) ,
$$
(since $\delta \geq 8\eps$) which concludes the proof of Items~\ref{item: d_y chi} and~\ref{item: d_yy chi}, and that of Proposition~\ref{prop: existence chi}.
\end{proof}

To conclude this section, it remains to prove Lemmata~\ref{lemma: transport psi},~\ref{lemma: propagation precised} and~\ref{lemma: property b}. In the following proofs, we shall systematically write $\eta$ in place of $P_\Lambda \xi$ and $\sigma$ in place of $(1 - P_\Lambda) \xi$ to lighten the notation. Hence, $\xi \in \R^2$ is decomposed as $\xi = \eta + \sigma$ with $\eta \in \left< \Lambda \right>$ and $\sigma \in \Lambda^\perp$, in accordance to Section~\ref{s:subtori}.

\begin{proof}[Proof of Lemma~\ref{lemma: transport psi} from Lemma~\ref{lemma: propagation precised}]
First, given a function $\psi \in \Cinfc(\T^2)$ satisfying~\eqref{eq: derivate psi}, we have, 
\begin{align*}
\left(\psi w_h , w_h \right)_{L^2(\T^2)}
& = \left(\Op_h(\psi \circ \phi_\tau) w_h , w_h \right)_{L^2(\T^2)} 
 +o(\tau h^{\frac{1+\delta}{2}})\|\Op_h(\psi \circ \phi_t)\|_{L^\infty(0,\tau ;\calL(L^2))} \\
& = \left(\Op_h(\psi \circ \phi_\tau) \Op_h \left(\chi\left(\frac{|\xi|^2-1}{h^\gamma} \right) \right) 
 \Op_h \left(\chi\left(\frac{\eta}{2h^\alpha} \right) \right)  w_h , w_h \right)_{L^2(\T^2)}  \\
& \quad  
 +\big( o(\tau h^{\frac{1+\delta}{2}}) +o(\tau h^{\frac{3+\delta}{2} -\gamma}) \big)
 \|\Op_h(\psi \circ \phi_t)\|_{L^\infty(0,\tau ;\calL(L^2))}, 
\end{align*}
when using Lemma~\ref{lemma: propagation precised} together with $ \Op_h \left(\chi\left(\frac{\eta}{2h^\alpha} \right) \right)  w_h = w_h$. Next, the pseudodifferential calculus yields
\begin{align}
\label{eq: propagation precised interm}
\left(\psi w_h , w_h \right)_{L^2(\T^2)}
& = \left(\Op_h \left( \psi \circ \phi_\tau \ \chi\left(\frac{|\xi|^2-1}{h^\gamma} \right) \chi\left(\frac{\eta}{2h^\alpha} \right) \right)  w_h , w_h \right)_{L^2(\T^2)} 
+ \O(h^{2-\gamma-2 \eps}) + \O(h^{1-\alpha-2 \eps}) \nonumber \\
& \quad  
 +\big( o(\tau h^{\frac{1+\delta}{2}}) +o(\tau h^{\frac{3+\delta}{2} -\gamma}) \big)
 \|\Op_h(\psi \circ \phi_t)\|_{L^\infty(0,\tau ;\calL(L^2))}  .
\end{align}
A particular feature of the Weyl quantization in the Euclidean setting is that the Egorov theorem provides an exact formula (see for instance~\cite{DS:99}): $\Op_h(\psi \circ \phi_t) = e^{-ith\frac{\Delta}{2}} \Op_h(\psi) e^{ith\frac{\Delta}{2}} $, so that $\|\Op_h(\psi \circ \phi_t)\|_{L^\infty(0,\tau ;\calL(L^2))} \leq C_0$ uniformly with respect to $h$.
Now, remark that the cutoff function $ \chi\left(\frac{\eta}{2h^\alpha} \right) \chi\left(\frac{|\xi|^2-1}{h^\gamma} \right)$ can be decomposed (for $h$ small enough) as
\begin{align*}
 \chi\left(\frac{\eta}{2h^\alpha} \right) 
 \chi\left(\frac{|\xi|^2-1}{h^\gamma} \right) 
 = \chi\left(\frac{\eta}{2h^\alpha} \right) 
 \left( \tilde{\chi}_\eta^h (\sigma) + \tilde{\chi}_\eta^h (- \sigma)  \right)
\end{align*}
for some nonnegative function $\tilde{\chi}_\eta^h$ such that $(\sigma,\eta)\mapsto \tilde{\chi}_\eta^h(\sigma) \in \Cinfc(\R^2)$, such that $\tilde{\chi}_\eta^h(\sigma) = \chi\left(\frac{|\xi|^2-1}{h^\gamma} \right)$ for $\eta \in \supp \chi\left(\frac{\cdot}{2h^\alpha} \right)$ and $\sigma >0$, and $\tilde{\chi}_\eta^h(\sigma) = 0$ for $\eta \notin \supp \chi\left(\frac{\cdot}{2h^\alpha} \right)$ or $\sigma \leq 0$.

Choosing $\gamma = \alpha$, we have in particular
$$
|\sigma - 1| \leq C h^{\alpha} \quad \text{on } \supp \left( \chi\left(\frac{\eta}{2h^\alpha} \right)  \tilde{\chi}_\eta^h (\sigma) \right) .
$$
Next, we recall that $\psi \circ \phi_\tau (s,y,  \sigma, \eta) = \psi (s+\tau \sigma, y+ \tau \eta )$, and we focus on the first term (corresponding to $\sigma >0$) in the right-hand side of the identity
\begin{align}
\label{eq: decomposition chi}
\chi\left(\frac{|\xi|^2-1}{h^{\alpha}} \right) \chi\left(\frac{\eta}{2h^\alpha} \right) \psi \circ \phi_\tau   
=\chi\left(\frac{\eta}{2h^\alpha} \right) 
 \left( \tilde{\chi}_\eta^h (\sigma) + \tilde{\chi}_\eta^h (- \sigma)  \right)\psi \circ \phi_\tau   .
\end{align}
We set
\begin{align*}
\zeta_{\tau}^{(1)} (s,y,\sigma,\eta) = \chi\left(\frac{\eta}{2h^\alpha} \right) \tilde{\chi}_\eta^h (\sigma) \psi (s+\tau \sigma, y+ \tau \eta ) , \quad \text{and} \quad 
\zeta_{\tau}^{(2)} (s,y ,\sigma,\eta) = \chi\left(\frac{\eta}{2h^\alpha} \right) \tilde{\chi}_\eta^h (\sigma) \psi (s+\tau, y) ,
\end{align*}
and we want to compare $\Op_h(\zeta_{\tau}^{(1)})$ and $\Op_h(\zeta_{\tau}^{(2)})$. For this, let us estimate, for multiindices $\ell, m \in \N^2$,
\begin{align}
\label{eq: calcul 1}
&\left|\d_{(s,y)}^{\ell} \d_{(\sigma,\eta)}^m\left(\zeta_{\tau}^{(2)} - \zeta_{\tau}^{(1)}\right) (s,y,\sigma,\eta)\right|
\nonumber \\
& \qquad \quad 
\leq C_{m} \sum_{\nu \leq m}
\left|\d_{(\sigma,\eta)}^{m-\nu} \left(\chi\left(\frac{\eta}{2h^\alpha} \right) \tilde{\chi}_\eta^h (\sigma) \right) \d_{(s,y)}^{\ell}\d_{(\sigma,\eta)}^{\nu} \left( \psi (s+\tau \sigma, y+ \tau \eta )  - \psi (s+\tau, y) \right) \right|.
\end{align}
On the one hand, we have
\begin{align}
\label{eq: calcul 2}
\left| \d_{(\sigma,\eta)}^{m-\nu} \left(\chi\left(\frac{\eta}{2h^\alpha} \right) \tilde{\chi}_\eta^h (\sigma) \right) \right|
\leq C_{m ,\nu} h^{-\alpha |m-\nu|}.
\end{align}
On the other hand, for $|\nu|>0$ we can also write
\begin{align*}
\left|\d_{(s,y)}^{\ell}\d_{(\sigma,\eta)}^{\nu} \left( \psi (s+\tau \sigma, y+ \tau \eta )  - \psi (s+\tau, y) \right) \right| 
&=\left|\d_{(s,y)}^{\ell}\d_{(\sigma,\eta)}^{\nu}  \psi (s+\tau \sigma, y+ \tau \eta )   \right|\\
& \leq C_{\ell, \nu} |\tau|^{|\nu|} h^{-2\eps (|\ell|+|\nu|)}\leq C_{\ell, \nu}   h^{-2\eps (|\ell|+ |\nu|)}, 
\end{align*}
since $|\tau|\leq C$. 

Finally, for $|\nu|=0$, we apply the mean value theorem to the function
$$
(\sigma,\eta) \mapsto \d_{(s,y)}^{\ell}   \psi (s+\tau \sigma, y+ \tau \eta )
$$ and write
\begin{align*}
&\left|\d_{(s,y)}^{\ell}  \left( \psi (s+\tau \sigma, y+ \tau \eta )  - \psi (s+\tau, y) \right) \right| \\
& \qquad \quad \leq 
 (|\eta| + |\sigma - 1|)
\sup_{T^*\T^2} \left|\nabla_{(\sigma,\eta)} \d_{(s,y)}^{\ell} \left( \psi (s+\tau \sigma, y+ \tau \eta )\right) \right| .
\end{align*}
With~\eqref{eq: derivate psi}, this yields 
\begin{align}
\label{eq: calcul 3}
\left|\d_{(s,y)}^{\ell} \left( \psi (s+\tau \sigma, y+ \tau \eta )  - \psi (s+\tau, y) \right) \right| 
& \leq 
 (|\eta| + |\sigma - 1|) C_{\ell} h^{-2\eps |\ell|} |\tau|  h^{-2\eps } \nonumber \\
&\leq 
 (|\eta| + |\sigma - 1|) C_{\ell} h^{-2\eps( |\ell|+ 1)} ,
 \end{align}
for $|\tau|\leq C$.

Using now that $|\eta| \leq Ch^\alpha$ and $|\sigma - 1| \leq C h^{\alpha}$ on $\supp \left(\chi\left(\frac{\eta}{2h^\alpha} \right) \tilde{\chi}_\eta^h (\sigma) \right)$, and combining~\eqref{eq: calcul 1},~\eqref{eq: calcul 2} and~\eqref{eq: calcul 3}, we obtain, for all $m \in \N^2, \ell \in \N^2$ and $0< h \leq h_0$ sufficiently small, 
\begin{align*}
h^{|m|}\left|\d_{(s,y)}^{\ell} \d_{(\sigma,\eta)}^m\left(\zeta_{\tau}^{(2)} - \zeta_{\tau}^{(1)}\right) (s,y ,\sigma,\eta)\right| 
&\leq C_{\ell, m} h^{\alpha-2\eps( |\ell|+ 1)}h^{|m|} h^{- \alpha |m|} \\
& \quad + 
C_{\ell, m}  \sum_{0<\nu \leq m} 
 h^{|m|}  h^{-2\eps (|\ell|+ |\nu|)} h^{- \alpha |m-\nu|} \\
& \leq C_{\ell, m}
\left( h^{(1-\alpha)|m|} h^{\alpha-2\eps (|\ell| + 1)} + |m| h^{|m|(1-\alpha)} h^{-2\eps |\ell|} h^{\alpha -2\eps}\right) \\
& \leq C_{\ell, m} h^{\alpha-2\eps (|\ell| + 1)} .
\end{align*}
Using a precised version of the Calder\'on-Vaillancourt theorem, as presented in Theorem~\ref{th: CV precised} below 
(in which only $|\ell| = 4$ derivations are needed with respect to $x$ in dimension two), we obtain
$$
\Op_h(\zeta_{\tau}^{(2)})
= \Op_h(\zeta_{\tau}^{(1)}) + \O_{\calL(L^2)} (h^{\alpha - 10 \eps}) .
$$

Similarly, we have
$$
\Op_h \left(\chi\left(\frac{\eta}{2h^\alpha} \right) \tilde{\chi}_\eta^h (- \sigma) \psi (s+\tau \sigma,y+ \tau \eta)\right) 
= \Op_h \left(\chi\left(\frac{\eta}{2h^\alpha} \right) \tilde{\chi}_\eta^h (- \sigma) \psi (s -\tau,y )\right)  + \O_{\calL(L^2)} (h^{\alpha - 10 \eps}) .
$$
Coming back to~\eqref{eq: propagation precised interm} and using~\eqref{eq: decomposition chi}, we finally obtain, for all $|\tau|\leq C$, 
\begin{align*}
\left(\psi w_h , w_h \right)_{L^2(\T^2)}
& = 
 \left(\Op_h \left(\chi\left(\frac{\eta}{2h^\alpha} \right) \tilde{\chi}_\eta^h (\sigma) \psi (s + \tau,y)\right)  w_h , w_h \right)_{L^2(\T^2)}  \nonumber \\
 & \quad +
\left(\Op_h \left(\chi\left(\frac{\eta}{2h^\alpha} \right) \tilde{\chi}_\eta^h (- \sigma) \psi (s -\tau,y)\right)  w_h , w_h \right)_{L^2(\T^2)}  \nonumber \\
& \quad + \O (h^{\alpha - 10 \eps}) 
 + \O(h^{1-\alpha-2 \eps})  
 + o( h^{\frac{1+\delta}{2}}) +o(h^{\frac{3+\delta}{2} -\alpha}) .
\end{align*}
With the pseudodifferential calculus, this yields~\eqref{eq: propagation cutoff}, which concludes the proof of Lemma~\ref{lemma: transport psi}.
\end{proof}

\begin{proof}[Proof of Lemma~\ref{lemma: propagation precised}]
Here, we only have to make more precise some arguments in the proof of Lemma~\ref{proposition: zero first order info}. Recall that according to Lemma~\ref{lemma: w_h quasimode}, $w_h$ satisfies $P_b^h w_h = o(h^{2+\delta})$. 

First, we take $\chi \in \Cinfc(\R)$, such that $\chi = 1$ in a \nhd of the origin. Hence, $\frac{1- \chi(r)}{r} \in \Cinf(\R)$ and we have the exact composition formula
$$
\Op_h \left( 1- \chi\left(\frac{|\xi|^2-1}{h^\gamma} \right) \right) 
= \Op_h\left(\left( 1- \chi\left(\frac{|\xi|^2-1}{h^\gamma} \right) \right)
 \frac{h^\gamma}{|\xi|^2-1}\right) \frac{P_0^h}{h^\gamma} ,
$$
since both operators are Fourier multipliers. Moreover, $\Op_h\left(\left( 1- \chi\left(\frac{|\xi|^2-1}{h^\gamma} \right) \right)  \frac{h^\gamma}{|\xi|^2-1}\right)$ is uniformly bounded as an operator of $\calL(L^2(\T^2))$. As a consequence, we have
\begin{align*}
& \left(\Op_h(a)\Op_h \left( 1- \chi\left(\frac{|\xi|^2-1}{h^\gamma} \right) \right) w_h , w_h \right)_{L^2(\T^2)} \\
\qquad & = \left(\Op_h(a) \Op_h\left(\left( 1- \chi\left(\frac{|\xi|^2-1}{h^\gamma} \right) \right)
 \frac{h^\gamma}{|\xi|^2-1}\right) \frac{P_0^h}{h^\gamma}w_h , w_h \right)_{L^2(\T^2)} \\
 \qquad & = \left(A \frac{P_b^h}{h^\gamma}w_h , w_h \right)_{L^2(\T^2)} 
 - \left( A \frac{ih b}{h^\gamma}w_h , w_h \right)_{L^2(\T^2)} ,
\end{align*}
where $A = \Op_h(a) \Op_h\left(\left( 1- \chi\left(\frac{|\xi|^2-1}{h^\gamma} \right) \right)
 \frac{h^\gamma}{|\xi|^2-1}\right)$ is bounded on $L^2(\T^2)$.
 Using $P_b^h w_h = o(h^{2+\delta})$ and $(b w_h, w_h)_{L^2(\T^2)} = o(h^{1+\delta})$, this gives
\begin{align*}
\left(\Op_h(a)\Op_h \left( 1- \chi\left(\frac{|\xi|^2-1}{h^\gamma} \right) \right) w_h , w_h \right)_{L^2(\T^2)} = o(h^{\frac{3+\delta}{2}-\gamma})\|\Op_h(a)\|_{\calL(L^2)} ,
\end{align*}
which in turn implies~\eqref{eq: loc xi =1}.

Next, Identity~\eqref{eq: propagation simple} yields, for all $a \in \Cinfc(\T^2)$, 
\begin{align*}
\left( \Op_h(\xi \cdot \d_x a)  w_h , w_h \right)_{L^2(\T^2)}
& = \frac{i}{2h}\left( \Op_h(a)  w_h , P_b^{h}w_h \right)_{L^2(\T^2)} 
- \frac{i}{2h}\left( \Op_h(a) P_b^{h} w_h , w_h \right)_{L^2(\T^2)}\nonumber\\
& \quad \quad - \frac{1}{2} \left( \Op_h(a)  w_h , b w_h \right)_{L^2(\T^2)}
- \frac{1}{2} \left( \Op_h(a) b w_h , w_h \right)_{L^2(\T^2)} \\
& = o(h^{1+\delta})\|\Op_h(a)\|_{\calL(L^2)} +o(h^{\frac{1+\delta}{2}})\|\Op_h(a)\|_{\calL(L^2)}, 
\end{align*}
as a consequence of $P_b^h w_h = o(h^{2+\delta})$ and $(b w_h, w_h)_{L^2(\T^2)} = o(h^{1+\delta})$. Applying this identity to $a \circ \phi_t$ in place of $a$, and integrating on $t \in [0, \tau]$ finally gives
\begin{align*}
\left( \Op_h(a \circ \phi_\tau)  w_h , w_h \right)_{L^2(\T^2)}
& =\left( \Op_h(a)  w_h , w_h \right)_{L^2(\T^2)}
 +o( \tau h^{\frac{1+\delta}{2}})\|\Op_h(a \circ \phi_t)\|_{L^\infty(0,\tau ; \calL(L^2))}, 
\end{align*}
which concludes the proof of Lemma~\ref{lemma: propagation precised}.
\end{proof}

\begin{proof}[Proof of Lemma~\ref{lemma: property b}]
Here, $B := B(x,b(x)^{2\eps})$ denotes the euclidian ball in $\T^2$ centered at $x$ of radius $b(x)^{2\eps}$. Setting
\begin{align*}
M :=\sup_{z \in B} b(z) , \qquad m :=\inf_{z \in B} b(z) ,
\end{align*}
we have 
\begin{align*}
|\nabla b (z)| \leq C_\eps b^{1-\eps}(z) \leq C_\eps M^{1-\eps} , \quad  \text{for all } z \in B ,
\end{align*}
as a consequence of Assumption~\ref{eq: assumption b}.
Moreover, the mean value theorem yields
\begin{align*}
b(x) - C_\eps M^{1-\eps} b(x)^{2\eps} \leq b(z) 
\leq b(x) + C_\eps M^{1-\eps} b(x)^{2\eps}  , \quad  \text{for } z \in B ,
\end{align*}
and, in particular,
\begin{align}
\label{eq: estimates M m}
m \geq b(x) - C_\eps M^{1-\eps} b(x)^{2\eps},
\quad \text{and} \quad
M \leq b(x) + C_\eps M^{1-\eps} b(x)^{2\eps}  .
\end{align}
Now, defining $f(M):=b(x) + C_\eps M^{1-\eps} b(x)^{2\eps}$, we see that $f$ is a strictly concave function with $f(0)=b(x)>0$. There exists a unique $M_0 \in \R_+$ satisfying $f(M_0)=M_0$. Moreover, we have $M\leq f(M)$ if and only if $M\leq M_0$.
Taking $b_0$ sufficiently small so that $b_0 + C_\eps b_0^{1+\eps^2}\leq b_0^{1-\eps}$, we obtain $f(b(x)^{1-\eps}) \leq b(x)^{1-\eps}$. In particular, this gives $M_0 \leq b(x)^{1-\eps}$ and hence $M\leq b(x)^{1-\eps}$ according to the second estimate of \eqref{eq: estimates M m}. Coming back to the first estimate of~\eqref{eq: estimates M m}, this yields
$$
m \geq b(x) - C_\eps b(x)^{(1-\eps)^2} b(x)^{2\eps} = b(x) - C_\eps b(x)^{1+\eps^2} .
$$
Taking $b_0$ sufficiently small so that $b_0 - C_\eps b_0^{1+\eps^2} \geq \frac{b_0}{2}$, we obtain
$m \geq \frac{b(x)}{2}$, which concludes the proof of Lemma~\ref{lemma: property b}.
\end{proof}

\part{An \textit{a priori} lower bound for decay rates on the torus: proof of Theorem~\ref{th: lower bound torus}}

Under the assumption 
\begin{equation}
\label{eq: hyp geom lower bound}
\ovl{\{b>0\}} \cap \{x_0 + \tau \xi_0, \tau \in \R\} = \emptyset,
\end{equation} for some $(x_0, \xi_0) \in T^*\T^2$, $\xi_0 \neq 0$, we construct in this section a constant $\kappa_0>0$ and a sequence $(\varphi_n)_{n \in \N}$ of $\O(1)$-quasimodes in the limit $n \to + \infty$ for the family of operators $P(i n \kappa_0)$.

We use the notation introduced in Sections~\ref{subsub: geometry on torus} and~\ref{sub: second microlocalization}.
First, note that, as a consequence of~\eqref{eq: hyp geom lower bound}, $\xi_0$ is necessarily a rational direction, and the set $\{x_0 + \tau \xi_0, \tau \in \R\}$ is a one-dimensional subtorus of $\T^2$, given by
$$
\{x_0 + \tau \xi_0, \tau \in \R\} = \ovl{\{x_0 + \tau \xi_0, \tau \in \R\}} = x_0 + \T_{\Lambda_{\xi_0}^\perp}, 
\quad \text{with }\Lambda_{\xi_0} \in \P.
$$

Let $\chi \in \Cinfc(\T^2)$ such that $\chi$ has only $x$-Fourier modes in ${\Lambda_{\xi_0}}$, $\chi =0$ on a \nhd of $\ovl{\{b>0\}}$ and $\chi =1$ on $x_0 + \T_{\Lambda_{\xi_0}^\perp}$. 

From Assumption~\eqref{eq: hyp geom lower bound}, we have $\rk({\Lambda_{\xi_0}}) = 1$, so that one can find $k \in \Lambda_{\xi_0}^\perp \cap \Z^2 \setminus \{0\}$. Besides, for all $n\in \N$ we have $n k \in \Lambda_{\xi_0}^\perp \cap \Z^2 \setminus \{0\}$.

We then define the sequence of {\it quasimodes} $(\varphi_n)_{n \in \N}$ by 
$$
\varphi_n (x) = \chi(x) e^{i n k \cdot x} , \quad n \in \N ,
\quad x \in \T^2 .
$$
We have $\varphi_n \in \Cinf(\T^2)$, together with the decoupling
$$
\varphi_n \circ \pi_{\Lambda_{\xi_0}}(s,y) = \chi(y) e^{i n k \cdot s} , \quad n \in \N ,
\quad (s,y) \in \T_{\Lambda_{\xi_0}^\perp} \times \T_{\Lambda_{\xi_0}}.
$$
This yields
\begin{align*}
- \big( T_{\Lambda_{\xi_0}} \Delta  T_{\Lambda_{\xi_0}}^* \big)\varphi_n \circ \pi_{\Lambda_{\xi_0}}(s,y) & = - \big( \Delta_{\Lambda_{\xi_0}} + \Delta_{\Lambda_{\xi_0}^\perp} \big) \varphi_n \circ \pi_{\Lambda_{\xi_0}}(s,y)  \\
& = -  e^{i n k \cdot s} \Delta_{\Lambda_{\xi_0}} \chi(y) 
+  n^2 |k|^2 \chi(y) e^{i n k \cdot s} .
\end{align*}
Moerover, $b \varphi_n =0$, according to their respective supports.
Hence, recalling that $P(i n|k|) = -\Delta - n^2|k|^2 + i n|k|b(x)$, we have
$$
 \big( T_{\Lambda_{\xi_0}} P(i n|k|)  T_{\Lambda_{\xi_0}}^* \big)\varphi_n \circ \pi_{\Lambda_{\xi_0}}= -  e^{i n k \cdot s} \Delta_{\Lambda_{\xi_0}} \chi(y), 
$$
and 
$$
\| P(i n|k|) \varphi_n \|_{L^2(\T^2)} 
= \| \big( T_{\Lambda_{\xi_0}} P(i n|k|)  T_{\Lambda_{\xi_0}}^* \big)\varphi_n \circ \pi_{\Lambda_{\xi_0}}
\|_{L^2(\T_{\Lambda_{\xi_0}^\perp} \times\T_{\Lambda_{\xi_0}})}
= C_0 \|\Delta_{\Lambda_{\xi_0}} \chi \|_{L^2(\T_{\Lambda_{\xi_0}})} .
$$
Since we also have $\|\varphi_n \|_{L^2(\T^2)}= \| T_{\Lambda_{\xi_0}} \varphi_n 
\|_{L^2(\T_{\Lambda_{\xi_0}^\perp} \times\T_{\Lambda_{\xi_0}})}= C_0 \| \chi \|_{L^2(\T_{\Lambda_{\xi_0}})}$, we obtain, for all $n \in \N$,
$$
\|P^{-1}(i n|k|)\|_{\calL(L^2(\T^2))} \geq \frac{\|\varphi_n \|_{L^2(\T^2)}}{\| P(i n|k|)  \varphi_n \|_{L^2(\T^2)}}
= \frac{\| \chi \|_{L^2(\T_{\Lambda_{\xi_0}})}}{\|\Delta_{\Lambda_{\xi_0}} \chi \|_{L^2(\T_{\Lambda_{\xi_0}})}} = C >0,
$$
which concludes the proof of Theorem~\ref{th: lower bound torus}.
\hfill \qedsymbol \endproof

\bigskip
\paragraph{Acknowledgments.}
The authors would like to thank Nicolas Burq for having found a significant error in a previous version of this article, and for advice on how to fix it. 
The second author wishes to thank Luc~Robbiano 
for several interesting discussions on the subject of this article.


\appendix
\section{Pseudodifferential calculus}
\label{section: pseudo}
In the main part of the article, we use the semiclassical Weyl quantization, that associates to a function $a$ on $T^*\R^2$ an operator $\Op_h(a)$ defined by
\begin{equation}
\label{eq: weyl quantif}
\big(\Op_h(a) u \big)(x) : = \frac{1}{(2\pi h)^2} \int_{\R^2} \int_{\R^2}
 e^{\frac{i}{h}\xi \cdot (x-y)} a \left( \frac{x + y}{2} , \xi \right) u(y) dy \ d\xi .
\end{equation}
For smooth functions $a$ with uniformly bounded derivatives, $\Op_h(a)$ defines a continuous operator on $\scrS(\R^2)$, and also by duality on $\scrS'(\R^2)$. 
On a manifold, the quantization $\Op_h$ may be defined by working in local coordinates with a partition of unity.
On the torus, formula \eqref{eq: weyl quantif} still makes sense : taking $a \in \Cinf(T^* \T^2)$ is equivalent to taking $a \in \Cinf(\R^2 \times \R^2)$, $(2\pi \Z)^2$-periodic with respect to the $x$-variable. Then the operator defined by \eqref{eq: weyl quantif} preserves the space of $(2\pi \Z)^2$-periodic distributions on $\R^2$, and hence $\Dist'(\T^2)$.

We sometimes write, with $D : = \frac{1}{i} \d$, 
$$
a(x, h D) = \Op_h(a) .
$$
We also note that $\Op_1(a)$ is the classical Weyl quantization, and that we have the relation
$$
a(x, h D) = \Op_h(a(x,\xi)) = \Op_1(a(x,h\xi)).
$$

\begin{theorem}
\label{th: CV precised}
There exists a constant $C >0$ such that for any $a \in \Cinf(T^*\T^2)$ with uniformly bounded derivatives, we have
$$
\| \Op_1(a) \|_{\calL(L^2(\T^2))} \leq C \sum_{\alpha ,\beta \in \{0,1,2\}^2}
\|\d_x^\alpha \d_\xi^\beta a \|_{L^\infty(T^*\T^2)}   .
$$
\end{theorem}
Equivalently, this can be rewritten as
$$
\| \Op_h (a) \|_{\calL(L^2(\T^2))} \leq C \sum_{\alpha ,\beta \in \{0,1,2\}^2}
h^{|\beta|} \|\d_x^\alpha \d_\xi^\beta a  \|_{L^\infty(T^*\T^2)}  .
$$
This precised version of the Calder\'on-Vaillancourt theorem 
is needed in Section~\ref{section: existence chi}, and proved in~\cite[Theorem~$B_\rho$]{Cordes:75} or~\cite[Th\'eor\`eme~3]{CM:78}. Here in dimension two, this means that only $|\alpha| = 4$ derivations are needed with respect to the space variable~$x$.


\section{Spectrum of $P(z)$ for a piecewise constant damping\\
(by St\'ephane Nonnenmacher)}
\label{app: stephane}

In this Appendix we provide an explicit description of some part of
the {\it spectrum} of
the damped wave equation \eqref{eq: stabilization} on $\T^2$, for a damping function
proportional to the
characteristic function of a vertical strip. We identify the torus
$\T^2$ with the square $\{-1/2\leq x<1/2,\
0\leq y<1\}$. We choose some half-width $\sigma\in (0,1/2)$, and
consider a vertical strip of width $2\sigma$. Due to translation
symmetry of $\T^2$, we may center this strip on the axis
$\{x=0\}$. Choosing a damping strength $\tB>0$, we then get the damping function
\begin{equation}\label{e:b(x)}
b(x,y)=b(x)=\begin{cases}0,& |x|\leq \sigma,\\ \tB,&\sigma<|x|\leq 1/2\,.
\end{cases}
\end{equation}
The reason for centering the strip at $x=0$ is
the parity of the problem \wrt that axis, which greatly simplifies the
computations.

We are interested in the spectrum of the operator $\A$ generating the
equation \eqref{eq: stabilization}, which amounts to solving the
eigenvalue problem
\begin{equation}\label{e:P(z)=0}
P(z)u=0,\quad\text{for}\quad P(z) = -\Delta + z b(x) + z^2,\quad
z\in\C\,,\ u\in L^2(\T^2)\,,\ u\not\equiv 0\,.
\end{equation}
This spectrum consists in a discrete set $\{z_j\}$,
which is symmetric \wrt the horizontal axis: indeed, any solution
$(z,u)$ admits a ``sister'' solution $(\bar z,\bar u)$. Furthermore, any
solution with $\Im z\neq 0$ satisfies 
\begin{equation}\label{e:Re(z)}
\Re z =-\frac12 \frac{(u, bu)_{L^2(\T^2)}}{\|u\|_{L^2(\T^2)}^2},\quad \text{and thus}\quad -\tB/2 \leq \Re z\leq 0\,.
\end{equation}
We may thus restrict
ourselves to the half-strip $\{ -\tB/2 \leq \Re z\leq 0,\ \Im z >0 \}$.

Our aim is to find high frequency eigenvalues ($\Im z\gg 1$) which are
as close as possible to the imaginary axis. We will prove the following
\begin{proposition}\label{pro:small-Re(z)}
There exists $C_0>0$ such that the spectrum \eqref{e:Re(z)} for the damping function
\eqref{e:b(x)} contains an infinite subsequence $\{z_i\}$
such that $\Im z_i\to\infty$ and $|\Re z_i|\leq \frac{C_0}{(\Im
  z_i)^{3/2}}$.
\end{proposition}
The proof of the proposition will actually give an explicit value for
$C_0$, as a function of $\tilde B$, $\sigma$.

%
%

\begin{proof}
To study the high frequency limit $\Im z\to\infty$ we will change of
variables and take 
$$z=i(1/h + \tzeta)\,,$$ 
where $h\in (0,1]$ will be a small parameter, while $\tzeta\in\C$ is assumed to be uniformly bounded
when $h\to 0$. The eigenvalue equation then takes the form
\begin{equation}
(-h^2\Delta + i h ( 1 + h \tzeta ) b ) u = \big(1+2h\tzeta (1+  h\tzeta/2)\big)u\,.
\end{equation}
Having chosen $b$ independent of $y$, we may naturally Fourier
transform along this direction, that is look for
solutions of the form $u(x,y)=e^{2i\pi ny}v(x)$, $n\in \Z$. For each
$n$, we now have to
solve the 1-dimensional problem
\begin{equation}\label{eq:P_1}
(-h^2\d^2/\d_x^2 + i h ( 1 + h \tzeta ) b(x) ) v = \big(1 - (2\pi h n)^2
+ 2h\tzeta (1+  h\tzeta/2)\big)\,v\,.
\end{equation}
Let us call
$$
B\defi \tB( 1 + h \tzeta ) ,\quad \zeta\defi \tzeta (1+ h\tzeta/2)\,.
$$
In terms of these parameters, the above equation reads:
\begin{equation}\label{eq:P_0=E}
(-h^2\d^2/\d_x^2 + i h B\,\bbbone_{\{\sigma< |x|\leq 1/2\}}(x) ) v = E\,v\,,\qquad\text{with}\quad E =  1 - (2\pi h n)^2 +
2h\zeta\,.
\end{equation}
Since we will assume throughout that $\tzeta=\O(1)$, we will
have in the semiclassical limit
\begin{equation}\label{eq:B-tB}
B=\tB+\O(h),\quad \tzeta=\zeta (1 - h\zeta/2 +\O(h^2))\,.
\end{equation}
At leading order we may forget that
the variables $B,\zeta$ are not independent from one another, and
consider \eqref{eq:P_0=E} as a  {\it bona fide} linear eigenvalue problem.

Since the function $b(x)$ is even, we may separately search for even, resp. odd
solutions $v(x)$. Let us start with the even solutions.
Since $b(x)$ is piecewise constant, any even and periodic solution
$v(x)$ takes the following form on $[-1/2,1/2]$ (up to a global
normalization factor):
\begin{align}\label{e:v(x)-even}
v(x) &= \begin{cases}\cos(k\,x),& |x|\leq \sigma,\\
\beta \cos\big(k'\,(1/2-|x|)\big),& \sigma < |x| \leq 1/2,\end{cases},\\
k&=\frac{E^{1/2}}{h},\quad k' = \frac{(E-ihB)^{1/2}}{h}\,.\label{eq:k-k'}
\end{align}
We notice that $k,k'$ are defined modulo a change of sign, so we may
always assume that $\Re k\geq 0$, $\Re k'\geq 0$.
The factor $\beta$ is obtained by imposing the
continuity of $v$ and of its derivative $v'$ at the discontinuity point $x=\sigma$ (we
use the notation $\sigma'\defi 1/2-\sigma$):
\begin{align*}
\cos (k\sigma) &= \beta \cos(k' \sigma'), \\
-k\,\sin(k\sigma) &= \beta k'\,\sin (k' \sigma')\,.
\end{align*}
The ratio of these two equations provides the quantization
condition for the even solutions:
\begin{equation}\label{eq:even-quantiz}
\tan(k\sigma) = - \frac{k'}{k}\,\tan(k' \sigma' )\,.
\end{equation}
Similarly, any odd eigenfunction takes the form (modulo a global
normalization factor):
\begin{equation*}\label{e:v(x)-odd}
v(x) = \begin{cases} \sin(k\,x),& |x|\leq \sigma,\\
\beta\,{\rm sgn}(x)\,\sin (k'(1/2-|x|)),& \sigma < |x| \leq 1/2,\end{cases},
\end{equation*} 
so the associated
eigenvalues should satisfy the condition
\begin{equation}\label{eq:odd-quantiz}
\tan(k\sigma) = - \frac{k}{k'}\,\tan(k' \sigma' )\,.
\end{equation}
We will now study the solutions of the quantization conditions 
\eqref{eq:even-quantiz} and \eqref{eq:odd-quantiz}, taking into
account the relations \eqref{eq:k-k'} between the wavevectors $k,k'$
and the energy $E$. To describe the full spectrum (which we plan to
present in a separate publication), we would need to consider several
r\'egimes, depending on the relative scales of $E$ and $h$. 
However, since we are only interested here in proving
Proposition~\ref{pro:small-Re(z)}, we will focus on the r\'egime leading to the
smallest possible values of $|\Im \tzeta|=|\Re z|$.
What characterizes the corresponding eigenmodes $v(x)$~? 
From \eqref{e:Re(z)} we see that the mass of $v(x)$ in the damped region,
$2\int_{\sigma}^{1/2}|v(x)|^2\,dx$, should be small compared to its
full mass. 
Intuitively, if such a mode were carrying a large horizontal ``momentum'' $\Re (hk)$
in the undamped region, it would then strongly penetrate the damped
region, because the boundary at $x=\sigma$ is not reflecting. As a
result, the mass in the damped region would be of the same order of
magnitude as the one in the undamped one.
This hand-waving argument explains why we choose to investigate the eigenmodes for
which $hk$ is the smallest possible,
namely of order $\O(h)$. 
This implies that $E=(hk)^2=\O(h^2)$, which means
that almost all of the energy is carried by the vertical momentum:
$$
hn = (2\pi)^{-1}+ \O(h)\,.
$$
The study of the full spectrum actually
confirms that the smallest values of $\Im \tzeta$ are obtained in
this r\'egime. 

Eq.\eqref{eq:k-k'} implies that the wavevector $k'$ in the damped region
is then much larger than $k$:
$$
k' = \frac{(-ihB + (hk)^2)^{1/2}}{h} = e^{-i\pi/4}(B/h)^{1/2} +
\O(h^{1/2})\,.
$$
$\Im k'\sigma'\approx -\sigma'(B/2h)^{1/2}$ is negative and large, so that
$\tan(k'\sigma')= -i + \O(e^{2\Im(k'\sigma')})$, uniformly \wrt
$\Re (k'\sigma')$.  

\subsubsection*{Even eigenmodes}
In this situation the even quantization condition \eqref{eq:even-quantiz} reads
\begin{equation}\label{eq:even-quantiz-1}
\tan(k\sigma) = i \frac{k'}{k}\,\big(1 + \O(e^{-\sigma'(2B/h)^{1/2}})\big)\,.
\end{equation}
Since the \rhs is large, $k\sigma$ must be close to a pole of the
tangent function. Hence, for each integer $m$ in a bounded
interval\footnote{Recall that we only need to study values $\Re k\geq 0$.}
$0\leq m\leq M$ we look for a solution of the form 
$$
k_{m+1/2} = \frac{\pi(m+1/2)}{\sigma} +\delta k_{m+1/2},\quad \text{with}\quad
|\delta k_{m+1/2}|\ll 1\,.
$$ 
The quantization condition \eqref{eq:even-quantiz-1} then reads
\begin{align*}
\sigma\delta k_{m+1/2} +\O((\delta k_{m+1/2})^2) &= 
i\; \frac{k_{m+1/2}}{e^{-i\pi/4}(B/h)^{1/2} + \O(h^{1/2})}\,\big(1 + \O(e^{-\sigma'(2B/h)^{1/2}})\big)\\
\Longrightarrow k_{m+1/2} &=\frac{\pi(m+1/2)}{\sigma} \Big(1+h^{1/2}\frac{e^{i3\pi/4}}
{\sigma B^{1/2}}+\O(h)\Big)\,.
\end{align*}
Using \eqref{eq:P_0=E}, the corresponding spectral parameter $\zeta$ is then given by 
\begin{align*}
\zeta_{n,m+1/2}&=\frac{(hk_{m+1/2})^2 + (2\pi hn)^2 -1}{2h}\\
&=\frac{(2\pi hn)^2 -1}{2h} +
\frac{h}{2}\Big(\frac{\pi(m+1/2)}{\sigma}\Big)^2
+ h^{3/2}\,\Big(\frac{\pi(m+1/2)}{\sigma}\Big)^2\frac{e^{i3\pi/4}}
{\sigma B^{1/2}}+\O(h^2)\,.
\end{align*}
From the assumptions on the quantum numbers $n,m$, we check
that $\zeta_{n,m+1/2}=\O(1)$.
We may now go back to the original variables $\tzeta,\,\tB$, using the relations
\eqref{eq:B-tB}.
The spectral parameter $\tzeta$ has an imaginary part
\begin{equation}\label{e:zeta_nm+1/2}
\Im \tzeta_{n,m+1/2} = \Im \zeta_{n,m+1/2} (1-h\Re\zeta_{n,m+1/2})+\O(h^2) = 
h^{3/2}\frac{(\pi(m+1/2))^2}{\sigma^3(2\tB)^{1/2}}+\O(h^2)\,.
\end{equation}
Returning back to the spectral variable $z$, the above expression
gives a string of eigenvalues $\{z_{n,m+1/2}\}$ with 
$\Im z_{n,m+1/2}=h^{-1}+\O(1)$, $\Re z_{n,m+1/2}=-\Im \tzeta_{n,m+1/2}$. These
even-parity eigenvalues prove Proposition~\ref{pro:small-Re(z)},
and one can take for $C_0$ any value greater than $\frac{(\pi/2)^2}{\sigma^3(2\tB)^{1/2}}$.
\end{proof}

%

We remark that the leading order of $k_{m+1/2}$ corresponds to the even spectrum of
the operator $-h^2\d^2/\d_x^2$ on the undamped
interval $[-\sigma,\sigma]$, with Dirichlet boundary
conditions. The eigenmode
$v_{n,m+1/2}$ associated with $\tzeta_{n,m+1/2}$ is indeed essentially
supported on that interval, where it resembles the Dirichlet eigenmode
$\cos\big(x\pi(1/2+m)/\sigma\big)$. 
At the boundary of that interval, it takes the value 
$$
v_{n,m+1/2}(\sigma)=(-1)^{m+1}e^{i3\pi/4}\,h^{1/2}\frac{\pi(m+1/2)}
{\sigma \tB^{1/2}}+\O(h)\,,
$$ 
and decays exponentially fast inside the
damping region, with a ``penetration length'' $(\Im
k')^{-1}\approx (2h/\tB)^{1/2}$. From \eqref{e:Re(z)} we see that the
intensity $|v_{n,m+1/2}(\sigma)|^2\sim C\,h$ penetrating on a distance
$\sim h^{1/2}$
exactly accounts for
the size $\sim h^{3/2}=h\,h^{1/2}$ of the $\Re z_{n,m+1/2}$.

We notice that the smallest damping occurs for the state $v_{n,1/2}$
resembling the ground state of the Dirichlet Laplacian.


\subsubsection*{Odd eigenmodes}
For completeness we also investigate the 
odd-parity eigenmodes with $k=\O(1)$. The computations are very
similar as in the even-parity case. 
The odd quantization condition reads in this r\'egime
\begin{equation}\label{eq:odd-quantiz-1}
\tan(k\sigma) = i \frac{k}{k'}\,\big(1 + \O(e^{-(2B/h)^{1/2}})\big)\,.
\end{equation}
The \rhs is then very small, showing that $\sigma k$ is close to a zero
of the tangent, so we may take $k_m = \pi m/\sigma +
\delta k_m$ with $|\delta k_m|\ll 1$ and $0\leq m\leq
M$. We easily see that the case $m=0$ does not lead to a solution. For
the case $m>0$ we get
$$
\delta k_m = e^{3i\pi/4}\,h^{1/2}\,\frac{\pi m}{\sigma^2 B^{1/2}}+\O(h)\,,
$$
and thus
$$
k_m = \frac{\pi m}{\sigma} \Big(1 +
h^{1/2}\,\frac{e^{3i\pi/4}}{\sigma B^{1/2}}+\O(h) \Big),\quad 1\leq
m\leq M\,.
$$
These values $k_m$ approximately sit on the same
``line'' $\{s(1 +
h^{1/2}\,\frac{e^{3i\pi/4}}{\sigma B^{1/2}}),\ s\in\R \}$ as the values $k_{m+1/2}$
corresponding to the even eigenmodes, both types of eigenvalues
appearing successively. The corresponding energy parameter $\tzeta_{n,m}$ satisfies
\begin{equation}\label{e:zeta_nm}
\Im \tzeta_{n,m} = h^{3/2}\frac{(\pi m)^2}{\sigma^3(2\tB)^{1/2}}+\O(h^2)\,.
\end{equation}
As in the even parity case, the
eigenmodes $v_{n,m}$ are close to the odd eigenmodes $\sin\big(x\pi m/\sigma\big)$ of the semiclassical
Dirichlet Laplacian on $[-\sigma,\sigma]$, and penetrate on a length
$\sim h^{1/2}$ inside the damped region.

\subsubsection*{The case of the square}

If the torus is replaced by the square $[-1/2,1/2]\times [0,1]$ with
Dirichlet boundary conditions, with the same damping function \eqref{e:b(x)}, the eigenmodes
$P(z)$ can as well be factorized into 
$u(x,y)=\sin(2\pi n y)v(x)$, with $n\in \frac12 \N\setminus 0$, and $v(x)$ must
be an eigenmode of the operator \eqref{eq:P_0=E} vanishing at $x=\pm
1/2$. We notice that the odd-parity eigenstates \eqref{e:v(x)-odd}
satisfy this boundary conditions, so the eigenvalues $z_{n,m}$ (with
real parts given by \eqref{e:zeta_nm}) belong to the spectrum of the damped Dirichlet problem.

Similarly, in the case of Neumann boundary conditions the eigenmodes
factorize as $u(x,y)=\cos(2\pi n y)v(x)$, with $n\in \frac12 \N$. The
even-parity states \eqref{e:v(x)-even} satisfy the Neumann boundary
conditions at $x=\pm 1/2$, so that the eigenvalues $z_{n,m+1/2}$
described in \eqref{e:zeta_nm+1/2} belong to the Neumann spectrum.

As a result, the Dirichlet and Neumann spectra also satisfy
Proposition~\ref{pro:small-Re(z)}.

\small
\bibliographystyle{alpha}     
\bibliography{bibli}

\end{document}